\documentclass{memo-l}

\usepackage{epsfig}
\usepackage{mathtools}
\usepackage{amsmath}
\usepackage{amssymb}
\usepackage{amsfonts}

\usepackage{fancyhdr}
\usepackage{amscd}
\usepackage{graphicx}
\usepackage{color}
\usepackage{verbatim}
\usepackage{bbm}
\usepackage{upgreek}
\RequirePackage{textcmds}\relax
\ProvideTextCommandDefault{\cprime}{\tprime}
\newtheorem{theorem}{Theorem}[section]
\newtheorem{lemma}[theorem]{Lemma}
\newtheorem{prop}[theorem]{Proposition}
\newtheorem{cor}[theorem]{Corollary}

\theoremstyle{definition}
\newtheorem{definition}[theorem]{Definition}

\newtheorem{ques}[theorem]{Question}

\theoremstyle{remark}
\newtheorem{remark}[theorem]{Remark}

\numberwithin{equation}{section}

\newcommand{\CT}{\mathcal{T}}

\newcommand{\CS}{\mathcal{S}}

\def \N {\mathbb N}
\def \O {\mathcal O}
\def \B {\mathcal B}

\def \D {\mathcal D}
\def \Z {\mathbb Z}
\def \R {\mathcal R}

\def \EA {E_{\mathsf A}}
\def \F {\mathcal F}
\def \H {\mathcal H}
\def \EF {\mathsf E\mathcal F}
\def \EH {\mathsf E\mathcal H}
\def \Q {\mathcal Q}

\def \P {\mathcal P}
\def \M {\mathcal M}
\def \MGX {\mathcal M_G(X)}
\def \MTX {\mathcal M_T(X)}
\def \MGY {\mathcal M_G(Y)}
\def \MGXP {\mathcal M_G(X')}
\def \htop {h_{\mathsf{top}}}
\def \hsex {h_{\mathsf{sex}}}
\def \sq {sequence}
\def \xg {$(X,G)$}

\def \tl {topological}
\def \im {invariant measure}
\def \inv {invariant}
\def \ds {dynamical system}

\def \diam {\mathsf{diam}}
\def \htop{h_{\mathsf{top}}}
\def \Hom {\mathsf{Hom}}
\def \id {\mathsf{Id}}
\def \ex {\mathsf{ex}}

\def \usc {upper semicontinuous}
\def \se {superenvelope}
\def \ens {entropy structure}
\def \zd {zero-dimen\-sio\-nal}
\def \CT {\mathcal{T}}

\def \eps {\varepsilon}

\def \N {\mathbb N}

\def \Z {\mathbb Z}
\def \R {\mathcal R}
\def \T {\mathbb T}

\def \A {\mathcal A}
\def \F {\mathcal F}
\def \H {\mathcal H}
\def \EF {\mathsf E\mathcal F}
\def \EH {\mathsf E\mathcal H}
\def \Q {\mathcal Q}

\def \P {\mathcal P}
\def \M {\mathcal M}
\def \MGX {\mathcal M_G(X)}
\def \MGY {\mathcal M_G(Y)}
\def \MGXP {\mathcal M_G(X')}
\def \htop {h_{\mathsf{top}}}
\def \hsex {h_{\mathsf{sex}}}
\def \sq {sequence}
\def \xg {$(X,G)$}

\def \tl {topological}
\def \im {invariant measure}
\def \inv {invariant}
\def \ds {dynamical system}

\def \htop{h_{\mathsf{top}}}
\def \Hom {\mathsf{Hom}}
\def \id {\mathsf{Id}}
\def \ex {\mathsf{ex}}
\def \usc {upper semicontinuous}
\def \se {superenvelope}
\def \ens {entropy structure}
\def \qt {quasitiling}
\def \CT {\mathcal{T}}
\def \CS {\mathcal{S}}

\makeindex

\begin{document}

\frontmatter

%\lhead{Tomasz Downarowicz and Guohua Zhang}
%\rhead{Symbolic extensions and the comparison property }

\title{Symbolic extensions of amenable group actions and the comparison property}

\author{Tomasz Downarowicz and Guohua Zhang}

\address{\vskip 2pt \hskip -12pt Tomasz Downarowicz}

\address{\hskip -12pt Faculty of Pure and Applied Mathematics, Wroclaw University of Science and Technology, Wybrze\.ze Wyspia\'nskiego 27, 50-370 Wroc\l aw, Poland}

\email{downar@pwr.edu.pl}

\address{\vskip 2pt \hskip -12pt Guohua Zhang}

\address{\hskip -12pt School of Mathematical Sciences and Shanghai Center for Mathematical Sciences, Fudan University, Shanghai 200433, China}

\email{chiaths.zhang@gmail.com}



\parindent=10pt
\begin{abstract}
In \tl\ dynamics, the \emph{Symbolic Extension Entropy Theorem} (SEET) \cite{BD} describes the possibility of a lossless digitalization of a \ds\ by extending it to a subshift on finitely many symbols. The theorem gives a precise estimate on the entropy of such a symbolic extension (and hence on the necessary number of symbols). Unlike in the measure-theoretic case, where Kolmogorov--Sinai entropy serves as an estimate in an analogous problem, in the \tl\ setup the task reaches beyond the classical theory of measure-theoretic and \tl\ entropy. Necessary are tools from an extended theory of entropy, the \emph{theory of entropy structures} developed in \cite{D0}. The main goal of this paper is to prove the analog of the SEET for actions of (discrete infinite) countable amenable groups:
\parindent=20pt
\bigskip

\begin{minipage}[c]{0,9\textwidth}
\emph{Let a countable amenable group $G$ act by homeomorphisms on a compact metric space $X$ and let $\MGX$ denote the simplex of all $G$-\inv\ Borel probability measures on $X$. A function $\EA$ on $\MGX$ equals the extension entropy function $h^\pi$ of a symbolic extension $\pi:(Y,G)\to (X,G)$, where $h^\pi(\mu)=\sup\{h_\nu(Y,G): \nu\in\pi^{-1}(\mu)\}$ ($\mu\in\MGX$), if and only if $\EA$ is a finite affine superenvelope of the \ens\ of $(X,G)$. }
\end{minipage}
\bigskip

\parindent=10pt
Of course, the statement is preceded by the presentation of the concepts of an \ens\ and its superenvelopes, adapted from the case of $\Z$-actions. In full generality we are able to prove a slightly weaker version of SEET, in which symbolic extensions are replaced by \emph{quasi-symbolic extensions}, i.e., extensions in form of a joining of a subshift with a zero-entropy tiling system. The notion of a tiling system is a subject of several earlier works (e.g. \cite{DHZ}, \cite{DH}) and in this paper we review and complement the theory developed there. The full version of the SEET (with genuine symbolic extensions) is proved for groups which are either residually finite or enjoy the so-called \emph{comparison property}. In order to describe the range of our theorem more clearly, we devote a large portion of the paper to studying the comparison property. Our most important result in this aspect is showing that all subexponential groups have the comparison property (and thus satisfy the SEET).  To summarize, the heart of the paper is the presentation of the following four major topics and the interplay between them:
\begin{itemize}
	\item Symbolic extensions,
	\item Entropy structures,
	\item Tiling systems (and their encodability),
	\item The comparison property.
\end{itemize}
\end{abstract}

\maketitle

\setcounter{tocdepth}{2}
\tableofcontents

\mainmatter

\section{Introduction}\label{sI}

\subsection{Motivation}

One of important tasks of the theory of \ds s is giving criteria for a lossless digitalization of a system. In classical ergodic theory of $\Z$-actions, Krieger's Generator Theorem \cite{Kr} resolves the problem completely using the Kolmogorov--Sinai entropy: if a measure-automorphism $T$ of a standard probability space $(X,\Sigma,\mu)$ has finite Kolmogorov--Sinai entropy $h=h_\mu(X,T)$ then the system has a finite generating partition, moreover, there exists such a partition into $l=\lfloor 2^h\rfloor +1$ atoms. That is to say, the system is isomorphic to a subshift over $l$ symbols equipped with some shift-\inv\ measure. This fact can be interpreted as the possibility of losslessly digitalizing the system, up to a measure-isomorphism, using $l$ symbols. The result was later generalized by \v S. \v Sujan \cite{Su} to free actions of (discrete infinite) countable amenable groups (with later improvements by A. Rosenthal \cite{R} and Danilenko--Park \cite{DP}). It is worth mentioning that recently B. Seward obtained an analog of Krieger's theorem for actions of general countable groups \cite{Sew}.

In topological dynamics, where one is concerned with a homeomorphism $T$ acting on a compact metric space $X$, one can ask an analogous question: how many symbols (or how much entropy) is needed to losslessly encode the system $(X,T)$ in a subshift? In general, it is impossible to represent the system by a \tl ly conjugate subshift, so instead one considers so-called \emph{symbolic extensions}, i.e., subshifts $(Y,S)$ which contain $(X,T)$ as a \tl\ factor. Since there are many such extensions, one is interested in optimizing the \tl\ entropy of the symbolic extension. In this way we are lead to the following parameter:
$$
\hsex(X,T)=\inf\{\htop(Y,S): (Y,S)\text{ is a symbolic extension of }(X,T)\}.
$$

By analogy to the measure-theoretic case, a naive guess would be that this parameter simply equals the topological entropy of the system $(X,T)$. But it is not the case. The parameter $\hsex(X,T)$ may assume values much higher than $\htop(X,T)$, including infinite, in which case symbolic extensions simply do not exist, even though $\htop(X,T)$ is finite. This phenomenon, first discovered by M. Boyle in the early 90's (and published much later in a survey \cite{BFF}), has lead to the development of the theory of symbolic extensions for $\Z$-actions (see \cite{BD,DN,DM, Bu,BuD,Se}, etc.). Technically, the parameter $\hsex(X,T)$ is much more sophisticated than $\htop(X,T)$ and tells us something that \tl\ entropy is incapable of telling: it describes the possibility of losslessly digitalizing the system, in particular it tells us how many symbols are needed for such a digitalization. Nowadays, we have a fairly good understanding of the subject matter and the associated phenomena. We understand why considerations of just \tl\ entropy are insufficient. Insufficient is also observing just the entropy function $\mu\mapsto h_\mu(X,T)$ defined on the simplex $\MTX$ of $T$-\inv\ Borel probability measures on $X$ associating to each measure its Kolmogorov--Sinai entropy. It is the defect in uniformity of the convergence of measure-theoretic entropy of \im s as the resolution improves, that has an essential impact on the entropy of possible symbolic extensions. This defect is captured by the \emph{theory of entropy structures} and their superenvelopes---objects that have no counterpart in ergodic theory. It is crucial that the way to calculate the \tl\ symbolic extension entropy $\hsex(X,T)$ is via computing the refined \emph{symbolic extension entropy function} $\mu\mapsto\hsex(\mu)$ on \im s. By definition, this function equals the pointwise infimum of \emph{extension entropy functions} $h^\pi$ defined on $\MTX$ for each symbolic extension $\pi:(Y,S)\to (X,T)$, as follows
$$
h^\pi(\mu)=\sup\{h_\nu(Y,S): \nu\in\pi^{-1}(\mu)\} \ \ \ (\mu\in\MTX).
$$
The key result of the theory of symbolic extensions, the \emph{Symbolic Extension Entropy Theorem} \cite{BD} asserts that a function $\EA$ on $\MTX$ equals $h^\pi$ in some symbolic extension $\pi:(Y,S)\to (X,T)$ if and only if it is a finite affine superenvelope of the entropy structure of the system $(X,T)$ (the definition of such a superenvelope is too complicated to be presented in the introduction and will be provided later). This allows to compute the function $\mu\mapsto\hsex(\mu)$ as the minimal superenvelope of the entropy structure, and finally, the \tl\ parameter $\hsex(X,T)$ is obtained as the supremum of $\hsex(\mu)$ over $\mu\in\MTX$.
\medskip

A natural direction of generalizing the theory of symbolic extensions is that of actions of countable amenable groups. For such actions we have completely analogous (to that in the $\Z$-case) structure of \im s (forming a Choquet simplex), with similarly defined notions of entropy (both measure-theoretic and \tl), enjoying similar basic properties. Beyond this class, say for arbitrary countable groups, the notion of a symbolic extension still makes sense, but there are serious problems with entropy. Sofic entropy, for instance, may equal minus infinity, or increase when passing to a factor (see e.g. \cite{Bow,KL}). So, for sofic groups there is no hope to create a theory of symbolic extensions with analogous connections to entropy notions as for $\Z$. This is why we believe that the realm of \tl\ actions of countable amenable groups is the most natural environment to carry over the theory of symbolic extensions and their entropy.

\subsection{Subject of the paper}

The goal the authors of this paper have focused on is very simple to formulate: prove an analog of the Symbolic Extensions Entropy Theorem for actions of countable amenable groups. A brief overview of the proof and the tools used for $\Z$-actions was rather reassuring: it should be possible to adapt most of them without too much trouble. This optimism however had a relatively short life. In reality, things turned out much more complex than they seemed, leading us to studying many accompanying topics such as \qt s, tilings and tiling systems and, above all, the mysterious comparison property. Eventually, even though we have acquired quite good insight into these subject matters and nontrivially contributed to their development, we were forced to make some (mild) compromises in the final formulation of the main theorem.

Some steps of the generalization are indeed quite straightforward. For example, most of the notions of the theory of entropy structures, such as uniform equivalence, entropy structure or \se, pass nearly unchanged. Also the proof of the ``easy'' implication of the
main theorem is a relatively painless adaptation from the $\Z$-case (practically only one detail needs to be reworked more carefully, but this does not present a serious challenge).

In the opposite ``hard'' implication the desired generalization becomes much less obvious. The proof of the direction relies on an effective construction of a symbolic extension $(Y,G)$ of a given system $(X,G)$, with an a priori given entropy function on \im s on $Y$ (this function delivered by the theory of entropy structures as a \se). In the $\Z$-case such an extension has the form of a \tl\ joining of two subshifts called \emph{rows}: the first row is the essential encoding of the system, the second row is just a zero entropy encoding of a ``dynamical parsing'' of each orbit into ``pieces'' of equal lengths. The second row is easily built and the description of how it is done occupies just a few lines. It is the first row that requires most of the effort in the construction which is divided into two main steps:
\begin{enumerate}
	\item From the a priori predicted entropy function (superenvelope) one derives an 		
	\emph{oracle}, a special integer-valued function on pieces of orbits which ``prophesies''
	the number of blocks in the symbolic extension that will correspond to each of these
	pieces.
	\item Using the oracle one creates the actual first row of the symbolic extension.
\end{enumerate}
Both steps are done with help of the parsing which must be applied beforehand to the elements of the system $(X,T)$. In order to make the decoding (i.e., the \tl\ factor map from $Y$ to $X$) possible, the parsing must be memorized in the symbolic extension, and this is exactly the role played by the second row.

Now, if $\Z$ is replaced by a general countable amenable group $G$, we encounter several serious obstacles, which we briefly discuss below.

First of all, the notion of a parsing must be replaced by a much more intricate notion of a \emph{tiling system}. In the classical case of $\Z$-actions, systems of parsings exist as factors in any aperiodic \zd\ system, which follows from a marker theorem attributed to W. Krieger (see \cite{Bo}). They occur under various names (as Krieger's markers, Kakutani--Rokhlin partitions or clopen tower partitions, etc.) and have numerous applications, for example in the study of full groups, orbit equivalence and Hopf-equivalence of minimal Cantor systems (see \cite{Sl} for an exposition on this subject, see also \cite{BH,GPS1,GPS2,GW}). For amenable group actions, for a long time, quasitilings of Ornstein and Weiss \cite{OW} have played a crucial role, and Lindenstrauss' Pointwise Ergodic Theorem \cite{L} is one of their most important applications. But we have quickly realized that, for building symbolic extensions, \qt s are rather useless and that we need more precise tilings (we explain why in subsection \ref{7.3}). In the long process of building up the foundations for this paper, we have, among many other things, proved in \cite{DHZ} that the Ornstein--Weiss \qt s can be improved to become tilings. Our tilings have already found numerous applications, see e.g. \cite{D, FH, FH1, S,Z,ZCY}. Further, in \cite{DH} we have proved that \qt s with arbitrarily good F\o lner properties exist \emph{as factors} in any free action of any countable amenable group. The results on tilings from \cite{DHZ} and \cite{DH} play a fundamental role in this paper.

Next encountered technical difficulties are associated with building the oracle (i.e., with the step (1) above) and result from lack of subadditivity of certain conditional entropy functions. This problem was resolved using, among other things, a sophisticated behavior (which we needed to establish in the first place) of entropy with respect to tiling systems. Once the oracle is successfully defined, step (2), i.e., building the analog of the first row, is performed in a manner more or less straightforward adapted from the $\Z$-case.

The most serious difficulty occurs, somewhat unexpectedly, in building the second row  responsible for memorizing the tiling system (in the $\Z$-case this is one of the easiest elements of the construction). It turns out that even though a tiling system created in \cite{DHZ} has zero \tl\ entropy, we are unable to encode it as a factor of a symbolic system. Hence we coined a notion of an \emph{encodable} tiling system and the existence of such tiling systems turns out to be one of the most serious challenges addressed in this paper. We confess, that we have stumbled upon this problem some time ago, and this has delayed the completion of the task undertaken in this work by several years. During these years we studied a new item necessary to put the pieces of the puzzle together: the \emph{comparison property} of countable amenable groups. And this subject became the second most important theme of this paper. The reader will find out that nearly half of the paper is devoted to or depends on this notion.

Comparison originates in the theory of \mbox{$C^*$-al}geb\-ras, but the most important for us ``dynamical'' version concerns group actions on compact metric spaces. In this setup it was defined by J.~Cuntz (see \cite{Cu}) and further investigated by M.~R\o rdam in \cite{MR1,MR2} and by W.~Winter in \cite{W}. As in the case of many other properties and notions in dynamical systems, the most fundamental form of comparison occurs in actions of the group $\Z$. In this context comparison is guaranteed for any action on a \zd\ compact metric space, which follows from the classical marker property of such actions (see \cite{Bo}). See also \cite{B} for more on comparison in $\Z$-actions.
For a wider generality, we refer the reader to a recent paper by D. Kerr \cite{K}, where the notion is defined for other actions including \tl\ and measure-preserving ones. We will focus on a particular case where a countable amenable group acts on a \zd\ compact metric space. In fact, this case also plays one of the leading roles in \cite{K}.

Unlike for $\Z$-actions, in the case of a general countable amenable group acting on a \zd\ compact metric space, it is unknown whether comparison necessarily occurs. There is neither a proof, nor a counterexample, although the problem has been attacked by several specialists for several years. Only a few partial results have been obtained, for instance, it is known (but never published, see \cite{Ph} and also \cite{SG}) that finitely generated groups with a symmetric F\o lner \sq\ satisfying Tempelman's condition (this includes all nilpotent, in particular Abelian, groups) have the comparison property, but beyond this case not much was known. In this paper we succeed in identifying a large class of groups whose any action on a \zd\ compact metric space admits comparison. Namely, it is the class of \emph{subexponential groups}, i.e., such that every finitely generated subgroup has subexponential growth. This covers all virtually nilpotent groups (which have polynomial growth) but also other, with intermediate growth, the most known example of which is the Grigorchuk group \cite{Gr}. By a recent result of E. Breuillard, B. Green and T. Tao \cite{BGT}, our result also covers the above mentioned ``Tempelman groups''; they turn out to be virtually nilpotent.

We establish a strong connection between comparison and the existence of tiling systems as factors of free \zd\ actions. In particular, if a group $G$ enjoys the comparison property then there exists an encodable tiling system of $G$ with zero \tl\ entropy. This opens the possibility of building symbolic extensions. Not counting the (relatively small) class of residually finite amenable groups, we can thus prove the Symbolic Extension Entropy Theorem (analogous as in the $\Z$-case) for groups which enjoy the comparison property, in particular for all subexponential groups. In case of a general countable amenable group we can prove a slightly deficient version, in which symbolic extensions are replaced by \emph{quasi-symbolic extensions}, defined as extensions in form of \tl\ joinings of subshifts with some (perhaps not encodable) zero entropy tiling system.

\subsection{Organization of the paper}

Section \ref{s2} contains rather standard material concerning actions of countable amenable groups, both \tl\ (on compact metric spaces) and measure-theoretic (on standard probability spaces), with special attention paid to subshifts and other \zd\ systems, as well as basic facts about entropy for such actions. The following two sections contain expositions on concepts less familiar to the potential reader, still not quite new. And so, in Section \ref{s3} we review entropy structures and symbolic extensions (including the proof of the easy direction of the main Symbolic Extension Entropy Theorem). What is new about these notions is their application (probably for the first time) to actions of countable amenable groups. But the translation from $\Z$-actions is more or less direct (though not completely trivial). Section \ref{s4} treats about F\o lner systems of \qt s and tiling systems, and is mainly a survey of authors' previous work \cite{DHZ} and \cite{DH}. The section is concluded by the presentation of tiled entropy---a new approach to dynamical entropy, natural in the context of tiling systems and necessary to cope with the difficulties encountered in the construction of symbolic extensions of countable amenable group actions that were not present in the case of $\Z$. In Section \ref{s5} we prove, in full generality, the hard direction of the Symbolic Extension Entropy Theorem, however in a slightly deficient version in which the extension is quasi-symbolic.
We are able to prove the full version of that theorem (with genuine symbolic extensions) for two important classes of groups, and Section \ref{s6} is devoted to introducing and studying one of these classes---groups with the comparison property. We describe alternative forms of this property, prove various auxiliary facts, but above all we prove that this property is enjoyed by the large class of subexponential groups. In Section \ref{s7} we show how comparison property allows to encode a zero entropy tiling system in a subshift on three symbols. This task, which in case of $\Z$-actions can be resolved in one line (using some standard constructions from \tl\ dynamics), in the general case becomes a complicated issue occupying several pages. With this tool in hand, we prove the full version of the Symbolic Extension Entropy Theorem for countable amenable groups which either enjoy the comparison property or are residually finite. At the end of the paper we have put an appendix, of perhaps independent interest, in which we reduce the alphabet used in Section \ref{s7}, from three to two symbols.

\section{Preliminaries on actions of countable amenable groups}\label{s2}
\subsection{Group actions, subshifts, symbolic extensions, block codes}

Thr\-oughout this paper $G$ denotes a (discrete) countable group with the unity $e$. By ``countable'' we will always mean ``infinite countable''. For finite groups, everything we address in this paper becomes trivial. Let $X$ be a compact metric space and let $\Hom(X)$ denote the group of all homeomorphisms $\phi:X\to X$. By an \emph{action} (more precisely, \emph{\tl\ action}) of $G$ on $X$ we will mean a homomorphism from $G$ into $\Hom(X)$, i.e., an assignment $g\mapsto \phi_g$ such that $\phi_{gg'} = \phi_g\circ\phi_{g'}$ for every $g,g'\in G$. It follows automatically that $\phi_{e}=\id$ (the identity homeomorphism) and that $\phi_{g^{-1}}=(\phi_g)^{-1}$ for every $g\in G$. Such an action will be denoted by \xg\ (although a group may act on the same space in many different ways, we will usually fix just one such action, hence this notation should not lead to a confusion). Another term used for $(X,G)$ is a \emph{\tl\ dynamical system} (or briefly a \emph{system}). From now on, to reduce the multitude of symbols used in this paper, we will write $g(x)$ in place of $\phi_g(x)$. The same applies to subsets $\mathsf A\subset X$: $g(\mathsf A)$ will replace $\phi_g(\mathsf A)$.
The action is called \emph{free} provided that $g(x)=x$ for at least one $x\in X$ implies $g=e$. A Borel measurable set $\mathsf A\subset X$ is called \emph{\inv} if $g(\mathsf A)=\mathsf A$ for every $g\in G$.

\medskip
An important example of an action of $G$ is the \emph{shift action on finitely many symbols}. Let $\Lambda$ be a finite set (usually, we assume that $\Lambda$ contains more than one element, otherwise the system is trivial) considered a discrete topological space (in this context the set $\Lambda$ will be called the \emph{alphabet}), and let
$$
\Lambda^G=\{x=(x_g)_{g\in G}:\ \forall_{g\in G}\ x_g\in\Lambda\}
$$
be equipped with any metric compatible with the product topology. Then $\Lambda^G$ is a compact metric space and $G$ acts on it naturally by shifts:
$$
\text{if \ }x=(x_f)_{f\in G}\text{ \ and \ }g\in G\text{ \ then \ }g(x)=(x_{fg})_{f\in G}.
$$
The system $(\Lambda^G,G)$ is called the \emph{full shift} (over $\Lambda$) while any nonempty closed invariant subset of $Y\subset\Lambda^G$ (regarded with the shift action) is called a \emph{subshift} or a \emph{symbolic system}.
If $(X,G)$ and $(Y,G)$ are actions of the same group on two (not necessarily different) spaces, and there exists a continuous surjection $\pi:Y\to X$ which commutes with the action (i.e., for any $g\in G$ and $y\in Y$, $\pi\circ g(y)=g\circ\pi(y)$), then $(X,G)$ is called a \emph{\tl\ factor of $(Y,G)$} and $(Y,G)$ is called a \emph{\tl\ extension of $(X,G)$}. In what follows, we will skip the adjective ``\tl'' and when the acting group is fixed and its action on given spaces is understood, we will also skip it in the denotation of the \ds s (i.e., we will use the letters $X$ and $Y$ in the meaning of $(X,G)$ and $(Y,G)$). If $X$ is a factor of $Y$, then the above map $\pi$ will be referred to as the \emph{factor map}. It has to be remarked that one system, say $X$, may be a factor of another, say $Y$, via many different factor maps. Since this may lead to a confusion, we will often use the phrase \emph{$X$ is a factor of $Y$ via the map $\pi$}. If the factor map is injective, in which case it is a homeomorphism between $Y$ and $X$, we will say that the systems are \emph{\tl ly conjugate}. From the point of view of \tl\ dynamics, conjugate systems are identical.

By a \emph{\tl\ joining} of finitely or countably many systems $X_k$ ($k\in K$ where $K=\{1,2,\dots,l\}$ with $l\in\N$, or $K=\N$) we will mean any closed subset $Z$ of the Cartesian product $\prod_{k\in K}X_k$, which is invariant under the product (coordinatewise) action, and whose projection on every coordinate is surjective. Such a joining will be sometimes denoted by $\bigvee_{k\in K}X_k$ (although this notation is ambiguous, as there may exist many joinings of the same collection of systems). At least one joining always exists---the product joining. The coordinate projections are factor maps from the joining to the respective coordinate systems. A special case of a countable joining is an inverse limit. We assume that $(X_k)_{k\in\N}$ is a \sq\ of systems such that, for each $k\in\N$, $X_k$ is a factor of $X_{k+1}$ via a map $\pi_k$ (referred to as the \emph{bonding map}). Then the \emph{inverse limit} of the \sq\ $(X_k)_{k\in\N}$ is defined as
$$
\overset\leftarrow{\lim_k}\, X_k=\{(x_k)_{k\in\N}: \forall_{k\in\N}\ x_k\in X_k \text{ and }x_k=\pi_k(x_{k+1})\}.
$$
It is elementary to check that the inverse limit is a countable joining of the systems $X_k$
($k\in\N$).
\medskip

The term ``symbolic extension'' which appears in the title of this paper is a very natural concept. By a \emph{symbolic extension} of a system $X$ we simply mean any symbolic system $Y$ (over some finite alphabet $\Lambda$) which is a (\tl) extension of $X$ (via some factor map $\pi$). Symbolic extensions are sometimes also called \emph{subshift covers}. The criteria for a system $X$ to admit at least one symbolic extension, and for computing how close the extension can be to $X$ in terms of information theory, are fairly well understood in case of $\Z$-actions. As explained in the Introduction, the goal of this paper is to see to what extent the same criteria apply to actions of general countable amenable groups.

Since we are discussing \tl\ factors, let us mention in this place the specific form of factor maps between two subshifts.

\begin{definition}\label{bcode} Let $\Lambda$ and $\Delta$ be some finite sets (alphabets). By a \emph{block code} we will mean any function $\Xi:\Lambda^F\to\Delta$, where $F$ is a nonempty finite subset of $G$ (called the \emph{coding horizon} of \,$\Xi$).
\end{definition}

The Curtis--Hedlund--Lyndon Theorem \cite{He} (which holds for actions of any countable group) states:
\begin{theorem}\label{CHL}
Let $Y\subset\Lambda^G$ be a subshift (over some finite alphabet $\Lambda$). Let
$\Delta$ be a finite set. Then $\xi:Y\to X\subset\Delta^G$ is a topological factor map (the image $X$ is then a subshift over $\Delta$) if and only if there exists a finite set $F\subset G$ and a block code $\Xi:\Lambda^F\to\Delta$,
such that, for all $y\in Y$ and $g\in G$ we have the equality
$$
(\xi(y))_g = \Xi(g(y)|_F).
$$
\end{theorem}
The term ``block code'' refers to both $\Xi$ and $\xi$, depending on the context, and $F$ is called a coding horizon of $\xi$ (and of $\Xi$). Clearly, if $F$ is a coding horizon of $\xi$ (and of $\Xi$), so is any finite set containing $F$. It will be convenient to assume that coding horizons always contain the unity.
\smallskip

By this opportunity we will introduce another convention often used in symbolic dynamics. Although it is a slight abuse of precision, it is commonly accepted and does not lead to a confusion. If $F\subset G$ is a finite set then any element $B=(B_f)_{f\in F}\in\Lambda^F$ will be called a \emph{block} (or, if needed, a \emph{block over $F$}). Now, if for some $g\in G$ the block $B'\in\Lambda^{Fg}$ satisfies $\forall_{f\in F}\ B'_{fg}=B_f$, then $B'$ will be called a \emph{shifted copy} of $B$. Shifted copies of the same block will be often denoted by the same letter. For example, for $y\in\Lambda^G$, in place of $g(y)|_F=B$ we will write $y|_{Fg}=B$. This means that $\forall_{f\in F}\ y_{fg}=B_f$.

With each block $B\in\Lambda^F$ we associate the \emph{cylinder set}
$$
[B]=\{x\in\Lambda^G: x|_F=B\}.
$$
The cylinder set is clopen in the full shift $\Lambda^G$ (following a common practice, we use the term ``clopen'' in the meaning of ``closed and open''). When regarding a subshift $X\subset\Lambda^G$ by a cylinder we will often mean the intersection of a cylinder (as defined above) with $X$. In this sense, cylinders are clopen in $X$.

\subsection{An $\varepsilon$-modification, $(K,\varepsilon)$-invariance, F\o lner \sq, amenability}\label{2.2}

This subsection introduces the key notions of amenability for countable groups. Amenability was introduced by J. von Neumann \cite{vN}. There are many equivalent ways of defining amenability and most of them apply to groups much more general than countable (see e.g. \cite{P}). We use the one which fits us best. It relies on the concept of a F\o lner \sq\ introduced by E. F\o lner \cite{Fo}. Note that all notions below depend on comparing cardinalities of certain sets, hence may be considered purely quantitative.

We will use $|F|$ to denote the cardinality of a set $F$.  Given a finite set $F\subset G$ and $\varepsilon>0$, an \emph{$\varepsilon$-modification} of $F$ is any set $F'$ such that $\frac{|F\triangle F'|}{|F|}<\varepsilon$, where $\triangle$ denotes the symmetric difference of sets. An $\varepsilon$-modification of $F$ which is also a subset of $F$ will be called a \emph{$(1\!-\!\varepsilon)$-subset} of $F$. If $K$ is another finite subset of $G$ then $F$ is called \emph{$(K,\varepsilon)$-\inv} if $KF$ is an $\varepsilon$-modification of $F$. For singletons, instead of ``$(\{g\},\varepsilon)$-\inv'' we will write ``$(g,\varepsilon)$-\inv''. Below we list a few easy but useful facts associated to the notions introduced above.
\begin{enumerate}
	\item If $F$ is $(K,\varepsilon)$-invariant then it is $(g,2\epsilon)$-\inv\ for every $g\in
	K$.
	\item By the \emph{$K$-core} of $F$ (denoted by $F_K$) we mean the set $\{f\in
	F:Kf\subset F\}=F\cap\bigcap_{g\in K}g^{- 1}F$. If $F$ is $(K,\varepsilon)$-\inv\ then its
	$K$-core is a $(1\!-\!|K|\varepsilon)$-subset of $F$ (see \cite[Lemma~2.6]{DHZ}).\label{dup}
	\item It follows that if $F$ is $(K,\varepsilon)$-\inv\ then any set $F'$ satisfying
	$F_K\subset F'\subset KF$ is a $(|K|\varepsilon\!+\!\varepsilon)$-modification of $F$.\label{cip}
	\item A $\delta$-modification of a $(K,\varepsilon)$-\inv\ set is $(K,\varepsilon')$-\inv, where
	$\varepsilon'=\frac{|K|\delta\!+\!\delta\!+\!\varepsilon}{1-\delta}$.
\end{enumerate}

\begin{definition}\label{tyi}
A \sq\ of finite subsets of $G$, $(F_n)_{n\in\N}$, is called a \emph{F\o lner \sq} if, for every finite set $K$ and every $\varepsilon>0$, the sets $F_n$ are eventually (i.e., except for finitely many indices $n$) $(K,\varepsilon)$-\inv. A group which possesses a F\o lner \sq\ is called \emph{amenable}.
\end{definition}
An immediate consequence of fact (4) above is that if $(F_n)_{n\in\N}$ is a F\o lner \sq\ in $G$ and, for each $n$, $F_n'$ is a $\delta_n$-modification of $F_n$, where $\delta_n\to0$, then $(F'_n)_{n\in\N}$ is a F\o lner \sq\ as well.

It is known that if a countable group $G$ is amenable then it possesses a F\o lner \sq\ $(F_n)_{n\in\N}$ with the following additional properties:
\begin{itemize}
	\item $\forall_{n\in\N}\ e\in F_n$,
	\item $\forall_{n\in\N}\ F_n\subset F_{n+1}$,
	\item $\forall_{n\in\N}\ F_n=F_n^{-1}$ (by convention, $F_n^{-1}=\{f^{-1}:f\in F_n\}$).
\end{itemize}

In reference to the above three properties of a F\o lner \sq, we will use the terms \emph{centered}, \emph{nested} and \emph{symmetric}, respectively. The first two properties are easily obtained, for the existence of symmetric F\o lner \sq s see \cite[Corollary 5.3]{N}.

\subsection{The Choquet simplex of \inv\ probability measures} Let $X$ be a compact metric space and let $\M(X)$ denote the family of all Borel probability measures on $X$. Since we shall consider no measures other than Borel probabilities, from now on ``measure'' will always mean an element of $\M(X)$. Endowed with the weak-star topology, this set is a metrizable \emph{Choquet simplex}, that is, it is a nonempty compact convex set which possesses a convex metric, and every its element $\mu$ has a unique representation as the integral average of the extreme points. Clearly, the extreme points of $\M(X)$ are the Dirac measures $\delta_x$ ($x\in X$) and the integral average representation of $\mu$ mentioned above is
$$
\mu = \int_X \delta_x\,d\mu(x).
$$
One of the standard convex metrics on $\M(X)$ compatible with the weak-star topology is given by the following formula
\begin{equation}\label{metric}
d_*(\mu,\nu)=\sum_{n=1}^\infty 2^{-n}\left|\int f_n\,d\mu - \int f_n\,d\nu\right|,
\end{equation}
where $(f_n)_{n\in\N}$ is some fixed \sq\ of continuous functions $f_n:X\to[0,1]$, linearly dense in the space $C(X)$ of all continuous real functions on $X$ (with the uniform metric).

If $G$ acts on $X$ then it also acts on $\M(X)$: for $g\in G$, the measure $g(\mu)$ is defined by the formula $g(\mu)(\mathsf A) = \mu(g^{-1}(\mathsf A))$ (where $\mathsf A$ is a Borel subset of $X$). We say that $\mu$ is an \emph{\im} if $\mu = g(\mu)$ for every $g\in G$. An \im\ is called \emph{ergodic} if $\mu(\mathsf A)\in\{0,1\}$ for every \inv\ Borel set $\mathsf A$. In many aspects concerning \im s, actions of amenable groups exhibit the same features as $\Z$-actions. In particular, the following fact holds (follows e.g. from \cite{Va}):

\begin{theorem}\label{thm1}
If a countable amenable group $G$ acts on a compact metric space $X$ then the family of all \im s (denoted by $\MGX$) is a metrizable Choquet simplex whose extreme points are exactly the ergodic measures.
\end{theorem}

The above theorem says that
\begin{enumerate}
\item $\MGX$ is a \emph{nonempty} and \emph{weakly-star closed} (hence compact) subset of $\M(X)$,
\item $\MGX$ is convex and the collection $\ex\MGX$ of its extreme points coincides with the collection of all ergodic measures,
\item every \im\ $\mu$ has a unique representation as the integral average of the ergodic measures, i.e., there exists a unique probability distribution $\xi_\mu$ on $\ex\MGX$ such that
$$
\mu = \int_{\ex\MGX} \nu\ \, d\xi_\mu(\nu).
$$
\end{enumerate}
The above formula is referred to as the \emph{ergodic decomposition} of $\mu$.

A way of proving the existence of \im s is by investigating \emph{empirical measures} of the form
$$
\boldsymbol\upmu_x^{F_n}=\frac1{|F_n|}\sum_{g\in F_n} \delta_{g(x)},
$$
where $F_n$ is a member of the F\o lner \sq\ and $x\in X$, and showing that with increasing $n$ such measures accumulate at \im s. We skip the details of this standard argument, however, we will need the following refinement:
\begin{prop}\label{oh}
Fix some $\gamma>0$. If $(F_n)_{n\in\N}$ is a F\o lner \sq\ in $G$ then there exists $n_0\in\N$ such that for every $n\ge n_0$ and any $x\in X$, the measure $\boldsymbol\upmu^{F_n}_x$ lies within the $\gamma$-neighborhood (in the metric $d_*$) of $\M_G(X)$.
\end{prop}
\begin{proof}
If, for some $\gamma>0$ and arbitrarily large indices $n$, there existed points $x_n\in X$ such that $d_*(\boldsymbol\upmu^{F_n}_{x_n},\MGX)>\gamma$ then the sequence of measures $(\boldsymbol\upmu^{F_n}_{x_n})_{n\in\N}$ would have some accumulation points outside $\MGX$, a contradiction.
\end{proof}

\medskip
Now suppose that a \tl\ \ds\ $X$ is a \tl\ factor of another, $Y$, via a map $\pi$. Then $\pi$ induces a map (which will be denoted by the same letter $\pi$) from $\MGY$ to $\MGX$, by the formula $\pi(\mu)(\mathsf A)=\mu(\pi^{-1}(\mathsf A))$ (where $\mathsf A\subset X$ is a Borel set). The following fact is well known:
\begin{prop}
The map $\pi:\MGY\to\MGX$ is a continuous affine surjection which sends extreme points to extreme points.
\end{prop}
A factor map $\pi:Y\to X$ such that $\pi:\MGY\to\MGX$ is injective (i.e., $\pi$ is an affine homeomorphism between the Choquet simplices $\MGY$ and $\MGX$) is called \emph{faithful}.

\subsection{The ergodic theorem}
Let $G$ be a countable amenable group.

\begin{definition} A F\o lner \sq\ $(F_n)_{n\in\N}$ in $G$ is called \emph{tempered} if, for each $n\in\N$, it satisfies the \emph{Shulman's condition}:
\begin{equation*}\label{tempered}
\Bigl|\bigcup_{i=1}^n F_i^{-1}F_{n+1}\Bigr|\le C|F_{n+1}|.
\end{equation*}
\end{definition}

It is very easy to see that any F\o lner \sq\ $(F_n)_{n\in\N}$ in $G$ contains a tempered sub\sq. It suffices to note that for each $n\in\N$ and then sufficiently large $k\in\N$, $F_{n+k}$ is $(\bigcup_{i=1}^n F_i^{-1},1)$-\inv, which implies that if, in the above condition, we replace $F_{n+1}$ by $F_{n+k}$, the above condition holds for $C=2$.
\smallskip

Let $(X,\Sigma,\mu)$ be a standard probability space (roughly, this means that $(X,\Sigma,\mu)$ can be modeled as a compact metric space with a Borel probability measure).
By a \emph{measure-theoretic action of $G$} we will understand the action on $(X,\Sigma,\mu)$ by measure-automorphisms, i.e., a case in which with each $g\in G$ we have associated a measurable and $\mu$-almost everywhere injective map $\phi_g:X\to X$ such that $\mu(\phi^{-1}_g(\mathsf A))=\mu(\mathsf A)$, for every $\mathsf A\in\Sigma$. Moreover we require that for all $g,h\in G$, $\phi_{gh}=\phi_g\circ\phi_h$. As in the case of a \tl\ action, we will write $g(x)$ and $g^{-1}(\mathsf A)$ in place of $\phi_g(x)$ and $\phi^{-1}_g(\mathsf A)$, respectively. Like in the \tl\ case, a measure-theoretic action of $G$ is called \emph{ergodic} if $\mu(\mathsf A)\in\{0,1\}$ for every \inv\ set $\mathsf A\in\Sigma$. The most important for us example of a measure-theoretic action of $G$ occurs when $G$ acts on a compact metric space $X$ by homeomorphisms, $\Sigma$ is the Borel sigma-algebra in $X$ and $\mu\in\M_G(X)$ (then the notion of ergodicity of the action coincides with the, introduced earlier, notion of ergodicity of the measure).

In the context of measure-theoretic actions of countable amenable groups, the pointwise ergodic theorem was proved by E. Lindenstrauss in \cite[Theorem 1.2]{L} (see also \cite{AJ75} for the necessity of Shulman's condition):

\begin{theorem}\label{ergodic} If $G$ acts by measure-automorphisms on a standard probability space $(X,\Sigma,\mu)$, the action is ergodic, $\mathsf A\in\Sigma$, and $(F_n)_{n\in\N}$ is a tempered F\o lner \sq\ in $G$ then, for $\mu$-almost every point $x\in X$, we have the equality
$$
\mu(\mathsf A)=\lim_{n\to\infty}\frac1{|F_n|}|\{g\in F_n: g(x)\in\mathsf A\}|.
$$
\end{theorem}

\subsection{Entropy}
For actions of countable amenable groups we have well defined notions of topological entropy $\htop(X,G)$, and, for an invariant measure $\mu$, of the measure-theoretic entropy $h_\mu(X,G)$ (later denoted by $h(\mu,X)$). Let us briefly recall the basics.

Let $(X,\Sigma,\mu)$ be a standard probability space and let $\P$ be a finite measurable partition of $X$. The \emph{Shannon entropy} of $\P$ equals
$$
H(\mu,\P)=-\sum_{P\in\P}\mu(P)\log(\mu(P))\le\log|\P|.
$$
Now suppose that a countable group $G$ acts on $(X,\Sigma,\mu)$ by measure-automorphisms. Given a finite measurable partition $\P$ of $X$ and a finite set $F\subset G$, by $\P^F$ we will mean the join
$$
\P^F=\bigvee_{g\in F}g^{-1}(\P)=\Bigl\{\bigcap_{g\in F}g^{-1}(P_g): \forall_{g\in F}\ P_g\in\P\Bigr\}
$$
(which is again a finite measurable partition of $X$). The Shannon entropy of this partition (with respect to $\mu$) will be denoted by $H(\mu,\P^F)$. One of elementary properties of the Shannon entropy, is the following subadditivity property: for any pair of sets $F_1, F_2\subset G$,
$$
H(\mu,\P^{F_1\cup F_2})\le H(\mu,\P^{F_1})+H(\mu,\P^{F_2}).
$$
In fact, strong subadditivity holds (see e.g. \cite{DFR}):
$$
H(\mu,\P^{F_1\cup F_2})\le H(\mu,\P^{F_1})+H(\mu,\P^{F_2})-H(\mu,\P^{F_1\cap F_2}).
$$

If now $G$ is amenable (and countable), then one defines the \emph{dynamical entropy of $\P$ with respect to }$\mu$ by the formula
$$
h(\mu,\P) = \lim_n \frac1{|F_n|}H(\mu,\P^{F_n}),
$$
where $(F_n)_{n\in\N}$ is a F\o lner \sq\ in $G$. Using strong subadditivity, one can prove that the limit defining the dynamical entropy of a partition equals the infimum over all finite subsets $F\subset G$ (see e.g. \cite{HYZ,DFR}):
$$
h(\mu,\P) = \inf_F \frac1{|F|}H(\mu,\P^F).
$$
In particular, this shows that the dynamical entropy of a partition (and hence also the Kolmogorov--Sinai entropy defined below) does not depend on the choice of the F\o lner \sq.
\smallskip

The \emph{Kolmogorov--Sinai entropy} of the measure-theoretic system $(X,\Sigma,\mu,G)$ is defined as
$$
h(\mu,X)=\sup_\P h(\mu,\P),
$$
where $\P$ ranges over all finite measurable partitions of $X$. The Kolmogorov--Sinai entropy can be infinite, however, this case is of marginal interest for us. The analog of the Kolmogorov--Sinai Theorem holds: if a finite partition $\P$ is a \emph{generator} (i.e., the smallest sigma-algebra containing the partitions $\P^F$ for all finite sets $F\subset G$ equals $\Sigma$), then the Kolmogorov--Sinai entropy is attained on $\P$:
$$
h(\mu,X)=h(\mu,\P).
$$
In any case, there exists a \emph{refining \sq\ of finite partitions} $(\P_k)_{k\in\N}$, i.e., such that for every $k\in\N$, $\P_{k+1}$ \emph{refines} $\P_k$ (meaning that every atom of $\P_{k+1}$ is contained in some atom of $\P_k$; we will write $\P_{k+1}\succcurlyeq\P_k$) and jointly they generate $\Sigma$. Then
$$
h(\mu,X)=\lim_k\uparrow h(\mu,\P_k).
$$

\medskip
If $(X,G)$ is a \tl\ \ds, then the Kolmogorov--Sinai entropy can be regarded as a function on $\M_G(X)$. In this case, $h(\mu,X)$ will be denoted shortly by $h(\mu)$ and the function $\mu\mapsto h(\mu)$ on $\M_G(X)$ will be called the \emph{entropy function}.

\medskip
Now consider two finite measurable partitions of $X$, $\P$ and $\Q$. In this context one defines the \emph{conditional Shannon entropy of $\P$ given $\Q$} (with respect to $\mu$) as
$$
H(\mu,\P|\Q) = \sum_{B\in\Q}\mu(B)H(\mu_B,\P) = H(\mu,\P\vee\Q)-H(\mu,\Q),
$$
where $\mu_B$ is the normalized conditional measure $\mu$ on $B$. Subadditivity still holds for conditional entropy (see formula 1.6.11 in \cite{D1}):
$$
H(\mu,\P^{F_1\cup F_2}|\Q^{F_1\cup F_2})\le H(\mu,\P^{F_1}|\Q^{F_1})+H(\mu,\P^{F_2}|\Q^{F_2}),
$$
but strong subadditivity in general fails.

The \emph{conditional dynamical entropy of $\P$ given $\Q$}, with respect to $\mu$ is defined analogously, as
$$
h(\mu,\P|\Q) = \lim_n \frac1{|F_n|}H(\mu,\P^{F_n}|\Q^{F_n}) = h(\mu,\P\vee\Q)-h(\mu,\P).
$$
Here also the limit can be replaced by the infimum over all finite sets $F$, which follows from the following three facts:
\begin{enumerate}
	\item for each finite set $F$, $H(\mu,\P^F|\Q^F)\ge H(\mu,\P^F|\Q^G)$, where $Q^G$ is the
	\inv\ sigma-algebra generated by $\Q$,
	\item $h(\mu,\P|\Q) = \lim_n \frac1{|F_n|}H(\mu,\P^{F_n}|\Q^G)$ (Abramov-Rokhlin formula \cite[Theorem 4.4]{WZ} or \cite[Lemma 1.1]{GTW}),
	\item the conditional entropy $H(\mu,\P^F|\Q^G)$ is strongly subadditive, which implies
	that $h(\mu,\P|\Q) = \inf_F \frac1{|F|}H(\mu,\P^F|\Q^G)\le \inf_F
	\frac1{|F|}H(\mu,\P^F|\Q^F)$ (both infima range over finite sets $F$).
\end{enumerate}

An important consequence of the above facts is the following observation:
\begin{lemma}\label{doda}
If $X$ is a \tl\ \ds\ and $\P$ and $\Q$ are finite partitions such that the boundary of any atom of either $\P$ or $\Q$ has measure zero for any $\mu\in\M$, where $\M\subset\MGX$, then the function $\mu\mapsto h(\mu,\P|\Q)$ is \usc\ on $\M$.
\end{lemma}
\begin{proof}
The ``small boundary property'' of $\P$ and $\Q$ easily implies that, for each finite $F\subset G$, the function $\mu\mapsto \frac1{|F|}H(\mu,\P^F|\Q^F)$ is continuous on $\M$. The infimum of any family of continuous functions on any metric space is \usc.
\end{proof}

In case $G=\Z$ and $F_n=\{1,2,\dots,n\}$ it is known that the \sq s $\frac1nH(\mu,\P^n)$ and
$\frac1nH(\mu,\P^n|Q^n)$ are in fact nonincreasing (see e.g., \cite[Fact 2.3.1]{D1}). This cannot be claimed in the case of a general countable amenable group. We will need to cope with this difficulty later.
\smallskip

Throughout this paper, we will be using the following convention: if $\pi:Y\to X$ is any map (between any spaces) and $\P$ is a (finite) partition of $X$, then the \emph{lifted partition}, $\{\pi^{-1}(P):P\in\P\}$ (which is a (finite) partition of $Y$) will be denoted by the same letter $\P$. We will take care to avoid any confusion caused by this convention.
Note that if $\pi$ is continuous then lifting partitions preserves measurability and the property of having clopen atoms.

Now consider a \tl\ factor map between two \tl\ \ds s, $\pi:Y\to X$. If $\nu\in\MGY$ and  then we can also define the \emph{conditional entropy of $\nu$ given $X$}, as follows
$$
h(\nu,Y|X)= \sup_\Q\inf_\P h(\nu,\Q|\P),
$$
where $\Q$ ranges over all finite measurable partitions of $Y$, while $\P$ ranges over all finite measurable partitions of $X$ (lifted to $Y$). If $h(\mu,X)<\infty$, where $\mu=\pi(\nu)\in\MGX$, then $h(\nu,Y|X)$ is simply the difference $h(\nu,Y)-h(\mu,X)$.

If, for every $\nu\in\MGY$, $h(\nu,Y|X)=0$ then $Y$ is called a \emph{principal extension of $X$}. A particularly good extension is described in the definition below:
\begin{definition}
Let $\pi:Y\to X$ be a \tl\ factor map between \tl\ \ds s. We say that $Y$ is an \emph{isomorphic extension} of $X$ (via the map $\pi$) if the associated map $\pi:\MGY\to\MGX$ is injective (i.e., the extension is faithful) and, for each $\nu\in\MGY$ and $\mu=\pi(\nu)\in\MGX$ the measure-preserving actions of $G$ on $(Y,\Sigma_Y, \nu)$ and on $(X,\Sigma_X,\mu)$ ($\Sigma_Y$ and $\Sigma_X$ denote the Borel sigma-algebras in $Y$ and $X$, respectively) are isomorphic in the measure-theoretic sense via the same map $\pi$.
\end{definition}
An isomorphic extension is both faithful and principal. For a  \tl\ extension $\pi:Y\to X$ to be isomorphic it suffices that there are sets $Y'\subset Y$ and $X'\subset X$ such that
$\nu(Y')=\mu(X')=1$ for every $\nu\in\MGY$ and $\mu\in\MGX$, and $\pi|_{Y'}$ is a bijection between $Y'$ and $X'$.
\smallskip

Since \tl\ entropy will play in this paper only a marginal role, we reduce its presentation to a necessary minimum. For actions of countable amenable groups the variational principle is valid (see \cite{STZ,MOP}), hence we can use it instead of a lengthy original definition (in fact one of many possible definitions). So, for our goals the following understanding of \tl\ entropy is completely sufficient:
\begin{definition}
Let a countable amenable group $G$ act on a compact metric space $X$. The \emph{\tl\ entropy} of the system $(X,G)$ equals
$$
\htop(X,G)=\sup_{\mu\in\M_G(X)}h(\mu,X).
$$
\end{definition}

\subsection{Zero-dimensional systems}
Let $X$ be, in addition to being compact and metric, also \zd\ (equivalently, totally disconnected), i.e., such that there exists a basis of the topology consisting of clopen sets. In such a space there exists a \sq\ of finite clopen partitions $(\P_k)_{k\in\N}$ (i.e., partitions whose all atoms are clopen) which is \emph{jointly refining in the \tl\ sense}, that is, denoting $\P_{[1,k]}=\bigvee_{i=1}^k\P_i$, and for a partition $\P$, letting $\diam(\P)$ denote the maximal diameter of an atom of $\P$, we have $\diam(\P_{[1,k]})\to 0$. Note that then the partitions $\P_{[1,k]}$ form a refining \sq\ also in the previously defined measurable sense. For each $k$, let $\Lambda_k$ be a set of labels bijectively associated to the atoms of $\P_k$, so that $\P_k = \{P_a:a\in\Lambda_k\}$. If now a countable group $G$ acts on $X$ (by homeomorphisms), then we introduce the following notation:
\begin{align*}
&\forall_{k\in\N}\,\forall_{g\in G}\ (x_{k,g}:=a\in\Lambda_k \iff g(x)\in P_a\in\P_k),\\
&\forall_{k\in\N}\ \pi_k(x):=(x_{k,g})_{\,g\in G}\in\Lambda_k^G,\\
&\pi(x):=(\pi_k(x))_{k\in\N}=(x_{k,g})_{k\in\N,g\in G}\in\prod_{k\in\N}\Lambda_k^G.
\end{align*}
We will call the double \sq\ $(x_{k,g})_{k\in\N,g\in G}$ the \emph{array-name} of $x$.
We let $X_k$ denote the image of $X$ by the map $\pi_k$. As easily verified, $X_k$ is a subshift over the alphabet $\Lambda_k$. Because all partitions $\P_k$ are clopen, the maps $\pi_k$ and $\pi$ are continuous. Also, they commute with the action, hence each $X_k$ as well as $\pi(X)$ are factors of $X$. The subshift $X_k$ will be called \emph{the $k$th layer} of $X$. We will denote by $X_{[1,k]}$ the projection of $X$ onto the first $k$ layers, which is a subshift over the product alphabet $\Lambda_{[1,k]}=\prod_{i=1}^k\Lambda_i$.
The natural projections provide bonding maps between the successive subshifts $X_{[1,k]}$ and allow to identify the image $\pi(X)$ with the inverse limit
$$
\overset\leftarrow{\lim_k}\,X_{[1,k]}.
$$
Because the diameters of the partitions $\P_{[1,k]}$ converge to zero, different points in $X$ have different array-names, which means that $\pi$ is injective. In this manner, we conclude that $X$ is \tl ly conjugate to the above inverse limit of subshifts. We will call it \emph{the array representation of $X$} and because we treat conjugate systems as one, we will simply write $X=\overset\leftarrow{\lim_k}X_{[1,k]}$. From now on, we will imagine any \zd\ system in its array representation (we always fix one of many possible such representations). Observe that if $X$ is given the array representation, the partitions $\P_k$ can be restored as the \emph{symbol partitions}:
$$
\P_k=\{[a]:a\in\Lambda_k\},
$$
where $[a]$ is the one-symbol cylinder at $e$, $\{x\in X:x_{k,e}=a\}$.

\section{Entropy structure and the easy direction of the main theorem}\label{s3}

Entropy structure for an action of a countable amenable group is defined in exactly the same manner as it is done for $G=\Z$. Let us recall some basic terms from the theory of \ens s for $\Z$-actions. \medskip

\subsection{Structures}

Let $\M$ be a compact metric set. By a \emph{structure} on $\M$ we will understand any nondecreasing \sq\ of commonly bounded nonnegative functions on $\M$, $\F =(f_k)_{k\ge 0}$ with $f_0\equiv 0$. Clearly, the pointwise \emph{limit function} $f=\lim_k f_k$ exists and is nonnegative and bounded.
\medskip

Two structures $\F =(f_k)_{k\ge 0}$ and $\F' = (f'_k)_{k\ge 0}$ are said to be \emph{uniformly equivalent} if
$$
\forall_{\varepsilon>0,\,k_0\ge 0}\ \exists_{k\ge 0}\ \ (f'_k>f_{k_0}-\varepsilon \text{ and } f_k>f'_{k_0}-\varepsilon).
$$
Notice the obvious fact that uniformly equivalent structures have a common limit function.
\medskip

Let $f$ be a nonnegative bounded function on $\M$. By the \emph{\usc\ envelope} of $f$ we shall mean the function $\tilde f$ defined on $\M$ by any of the following formulas
$$
\tilde f(\mu) = \limsup_{\mu'\to\mu}f(\mu') = \inf_{U\ni\mu}\sup\{f(\mu'):\mu'\in U\} = \inf\{g \text{ continuous and } g\ge f\},
$$
where $\mu,\mu'\in\M$ and $U$ ranges over all open neighborhoods of $\mu$. Note that $\tilde f\ge f$. We also define the \emph{defect of $f$} as the difference $\overset{...}f = \tilde f-f$. The function $f$ is \emph{\usc} if $f=\tilde f$ or, equivalently, $\overset{...}f \equiv 0$.
\medskip

We will say that a structure $\F$ has \emph{\usc\ differences}, if the difference functions $f_{k+1}-f_k$ are \usc\ for every $k\ge 0$.

If $\M$ is a compact convex subset of some locally convex linear space and all functions $f_k$ are affine, then we will say that $\F=(f_k)_{k\ge 0}$ is an \emph{affine structure}.

\subsection{Superenvelopes}

\begin{definition}
Given a structure $\F=(f_k)_{k\ge 0}$ on a compact domain $\M$, a nonnegative function $E$ on $\M$ is called a \emph{superenvelope} of $\F$ if $E\ge f_k$ for each $k\ge 0$ and the defects $\overset{...........}{E-f_k}$ tend pointwise to zero. Notice that then $E-f$ (where $f$ is the limit function of $\F$) is upper semicontinuous, hence bounded, thus $E$ is also bounded. A priori a structure may have no such \se s. By default, the constant infinity function is added to the collection of  \se s of any structure.
\end{definition}

We have the following facts (see e.g. \cite[Lemma 8.1.10, Theorem 8.1.25 (2), Lemma 8.1.12 and Theorem 8.2.5]{D1}):

\begin{prop}\label{fact2}
\begin{enumerate}
  \item The infimum of all \se s of a structure $\F$ is a \se\ of $\F$ (in the
  extreme case this is the constant infinity function). This \emph{minimal \se} of
  $\F$ will be denoted by $\EF$.
	\item Uniformly equivalent structures have the same collection of \se s (hence the same minimal \se).
	\item If $\F=(f_k)_{k\ge 0}$ has \usc\ differences then a (finite) function $E$ is its \se\ if and only if $E-f_k$ is nonnegative and \usc\ for every $k\ge 0$ (in particular $E = E-f_0$ is then \usc).
	\item If $\F$ is an affine structure with \usc\ differences, defined on a Choquet simplex, then $\EF$ coincides with the pointwise infimum of all affine \se s of $\F$ (in particular $\EF$ is concave).
\end{enumerate}
\end{prop}

\smallskip
We will need the following terminology: Let $\pi:Y\to X$ be a continuous surjection between compact metric spaces. Given a bounded nonnegative function $f$ on $X$, we define its \emph{lift of $f$ against $\pi$} as the composition $f\circ\pi$ (which is a function defined on $Y$). As in the case of partitions, we will denote the lift of $f$ by the same letter $f$. The lifted function is \emph{constant on fibers}, that is, it is constant on the sets $\pi^{-1}(x)$ ($x\in X$). Going in the opposite direction is less obvious. Let now $f$ be a bounded function on $Y$. We define the \emph{push-down} of $f$ (along $\pi$) as the function $f^\pi$ on $X$ given by $f^{\pi}(x) = \sup\{f(y):y\in\pi^{-1}(x)\}$ ($x\in X$). Lifting reverses the operation of pushing down exclusively for functions constant on fibers.

We have the following facts, some of which are immediate, some are proved in \cite[Facts A.1.26 and A.2.22]{D1}:
\begin{prop}\label{pushdown}
\begin{enumerate}
\item The operation of lifting preserves continuity, upper and lower semicontinuity of a function. If $Y$ and $X$ are convex sets, and $\pi$ is affine, then lifting preserves concavity, convexity and affinity of a function. The same holds for pushing down functions which are constant on fibers.
\item In general, the pushing down preserves upper semicontinuity of a function. If $Y$ and $X$ are
convex sets, and $\pi$ is affine, then pushing down preserves concavity of a function. If, moreover, $Y$ and $X$ are Choquet simplices and $\pi$ sends extreme points of $Y$ to extreme
points of $X$, then pushing down also preserves affinity of a function.
\end{enumerate}
\end{prop}

\subsection{Definition of the entropy structure}
This section is almost identical as in the $\Z$-case \cite{D0}. The only difference is that the cited theorem about the existence of a principal zero-dimen\-sio\-nal extension in the general case (which we will use below) is incomparably more intricate than that for $\Z$-actions. The rest follows the standard scheme and there are no essential differences.
The \ens\ for an action of an amenable group $G$ on a compact metric space $X$ will be introduced in two steps. At first we will do it in case $X$ is \zd, next we will address the general case. In both cases we assume finite \tl\ entropy of $X$.
\smallskip

\subsubsection{The \zd\ case}

\begin{definition}\label{ensz}
Let $X=\overset\leftarrow{\lim_k} X_{[1,k]}$ be a zero-dimen\-sio\-nal system in its array representation. Let $(\P_k)_{k\in\N}$ denote the associated \sq\ of symbol partitions. For every $\mu\in\MGX$ define $h_k(\mu) = h(\mu,\P_{[1,k]})$ (also $h_k(\mu)=h(\mu_k,X_{[1,k]})$, where $\mu_k$ is the image of $\mu$ on the first $k$ layers $X_{[1,k]}$ by the natural projection map $\pi_{[1,k]}:X\to X_{[1,k]}$). Then the structure $\H = (h_k)_{k\ge 0}$ on the Choquet simplex $\MGX$ is called an \emph{\ens} of $X$.
\end{definition}
We have the following crucial fact:

\begin{theorem}\label{crucial}The \ens\ is an affine structure with \usc\ differences, converging nondecreasingly to the entropy function.
\end{theorem}

\begin{proof}
Affinity of the entropy function is a commonly known fact (see e.g. \cite[Theorem 2.5.1]{D1}; the same proof applies to actions of all countable amenable groups). Each function $h_k$ is in fact the entropy function on $X_{[1,k]}$ lifted against the factor map $\pi_{[1,k]}$. Since the factor map applied to \im s is affine, the lifted function is affine, too. The nondecreasing convergence to the entropy function is obvious. For upper semicontinuity of the differences, notice that, for any $k\ge0$ and $\mu\in\MGX$, we have $h_{k+1}(\mu)-h_k(\mu)= h(\mu,\P_{k+1}|\P_{[1,k]})$ (for $k=0$ there is no conditioning). Both involved partitions are clopen, i.e., their atoms have empty boundary. By Lemma~\ref{doda}, the discussed difference function is upper semicontinuous on $\MGX$.
\end{proof}

There exist many \ens s depending on the choice of the array representation, however, all these structures are uniformly equivalent (see below), and hence, by Proposition \ref{fact2} (2), have the same collection of \se s and the same minimal \se.
\smallskip

\subsubsection{The general case}

There are many ways of introducing the entropy structure in actions of $\Z$ on general compact metric spaces (see \cite{D0}). Most of them (but not all) can be adapted to actions of general countable amenable groups. We choose one which seems to pass in a most direct manner.

By a deep result of D. Huczek (see \cite[Theorem 2]{H}), any action of a countable amenable group $G$ on a compact metric space $X$ has a principal \zd\ extension $X'$. Let $\pi':X'\to X$ denote the corresponding factor map. %Although this extension need not be faithful, the entropy function $h'$ on $\MGXP$ is constant on fibers and equals the entropy function on $\MGX$ lifted against $\pi'$. However, if $\H'=(h'_k)_{k\ge 0}$ is an entropy structure of $X'$, we have no guarantee that the functions $h'_k$ are also constant on fibers. This fact makes the verification of correctness of the definition given below a bit harder.
We define the entropy structure of a \tl\ \ds\ $X$ following the idea from \cite[Definition 5.0.1]{D1}. Recall that we assume finite \tl\ entropy of the action on $X$.

\begin{definition}\label{ensf}
The \ens\ of a \tl\ \ds\ $X$ of finite entropy is defined as any structure $\H=(h_k)_{k\ge 0}$ on $\MGX$, such that for any principal \zd\ extension $\pi':X'\to X$, and any entropy structure $\H'=(h'_k)_{k\ge 0}$ on $\MGXP$, the structure $\H=(h_k)_{k\ge 0}$ lifted against $\pi'$ is uniformly equivalent to $\H'$.
\end{definition}

Of course, it is a priori not obvious, that such a structure exists. Once the existence of at least one such structure $\H$ is guaranteed, it becomes obvious that any other structure defined on $\MGX$ is an entropy structure if and only it is uniformly equivalent to $\H$. In particular, it will be now obvious that if $X$ is \zd\ then the \ens\ from Definition \ref{ensz} is consistent with Definition \ref{ensf} and does not depend on the array representation.

The proof of existence is tedious and requires (for example) a concept of entropy via finite families of continuous functions (instead of partitions). We prefer not to copy entire sections from, for example, \cite{D0} or \cite{D1}. The construction does not depend, in any aspect, on the acting group, and its details play no role in this paper. So, we choose to  formulate the existence theorem without a detailed proof and confine ourselves to suitable references.

\begin{theorem}\label{niezalezy}
If $X$ has finite \tl\ entropy, then it has an entropy structure which is affine and has \usc\ differences.
\end{theorem}

\begin{proof}Combine \cite[Definition 6.2.1]{D0} (adapted to a F\o lner \sq\ $(F_n)_{n\in\N}$) and \cite[Lemma 7.1.2]{D0} with part (1) of the proof of \cite[Theorem 7.0.1]{D0}. The proofs for countable amenable groups are identical. In one place, for upper semicontinuity of a conditional entropy function, Lemma \ref{doda} (from this paper) is needed.
\end{proof}

In order to obtain a notion which does not depend on any choices, we will replace the ``individual'' \ens s by entire uniform equivalence class. Nevertheless, instead of saying that $\H$ \emph{belongs to} an \ens\ we will keep saying that it \emph{is} an \ens. The \ens\ defined as a uniform equivalence class is an invariant of \tl\ conjugacy in the following sense:

\begin{theorem}\label{topinv}
Suppose $X$ and $Y$ are \tl ly conjugate, say, $\pi:Y\to X$ is the conjugating map. Then a structure $\H = (h_k)_{k\ge 0}$ defined on $\MGX$ is an \ens\ of $X$ if and only if $\H\circ\pi = (h_k\circ\pi)_{k\ge 0}$ (defined on $\MGY$) is an \ens\ of $Y$.
\end{theorem}
\begin{proof} Conjugate systems have the same principal \zd\ extensions.
\end{proof}

\subsection{Symbolic extensions---the easy direction}

Let us go back to the situation where $Y$ is a \tl\ extension of $X$ via a map $\pi:Y\to X$.

\begin{definition}\label{exenf}
In this context, on $\MGX$ we define the \emph{extension entropy function} as the push-down along $\pi$ of the entropy function on $\MGY$:
$$
h^\pi(\mu) = \sup\{h(\nu,Y):\nu\in\MGY,\ \pi(\nu)=\mu\}.
$$
\end{definition}
\begin{definition}\label{sexenf}
Given a \tl\ \ds\ $X$, on $\MGX$ we define the \emph{symbolic extension entropy function}, as the infimum of all extension entropy functions arising from symbolic extensions $Y$ of $X$:
$$
\hsex(\mu) = \inf\{h^\pi(\mu):\ \pi:Y\to X \text{ is a symbolic extension}\}.
$$
\end{definition}
Every symbolic system $Y$ has finite \tl\ entropy (at most $\log|\Lambda|$, where $Y\subset\Lambda^G$), while, by convention, infimum of an empty set equals $+\infty$. Thus, lack of symbolic extensions of $X$ is equivalent to the condition $\hsex\equiv\infty$ on $\MGX$ (otherwise $\hsex$ is always bounded).
\medskip

We can also define the \emph{\tl\ symbolic extension entropy} of $X$, as
$$
\hsex(X,G)=\inf\{\htop(Y,G): Y\text{ is a symbolic extension of }X\}.
$$
It can be proved, using the same methods as in the $\Z$-case (see \cite{BD}), that
$$
\hsex(X,G)=\sup_{\mu\in\MGX}\hsex(\mu).
$$
This equality, called the \emph{symbolic extension entropy variational principle}, is the reason why independent study of the \tl\ symbolic extension entropy is of lesser interest.
This paper is devoted to proving the analog of the following theorem for $\Z$-actions, known as the \emph{Symbolic Extension Entropy Theorem} \cite{BD}.

\begin{theorem}\label{dlaz}
Let $X=(X,T)$ be a \tl\ \ds\ (the $\Z$-action generated by a single homeomorphism $T:X\to X$ of a compact metric space $X$). If the \tl\ entropy of $X$ is infinite then obviously $X$ admits no symbolic extensions. Otherwise we have the following equivalence:
Let $\EA$ be a (finite) function defined on $\MTX$. There exists a symbolic extension $\pi:Y\to X$ such that $\EA=h^\pi$ if and only if $\EA$ is an affine \se\ of the \ens\ of $X$. In particular, $\hsex\equiv\EH$, where $\H$ is (belongs to) the \ens\ of $X$ (this includes the infinite case: $\EH\equiv\infty$ if and only if $X$ has no symbolic extensions).
\end{theorem}

For actions of general countable amenable groups, one implication is relatively easy to prove, and the proof does not differ much from that for $\Z$-actions. As for the other implication, we encounter an (at the moment) inaccessible problem, and we must make a sacrifice: either add an assumption on $G$ or widen the notion of a symbolic extension. In both cases we will need more terminology, thus the formulation will be provided later. For now, we can prove the ``easy direction'':

\begin{theorem}\label{easy}
Let a countable amenable group $G$ act on a compact metric space $X$. Let $\pi:Y\to X$ be a
symbolic extension of $X$. Then the extension entropy function $h^\pi$ is an affine \se\ of the \ens\ of $X$.
\end{theorem}

\begin{proof}
One fact used in the proof of the easy direction for $\Z$-actions (that asymptotically $h$-expansive systems have principal symbolic extensions) is uncertain for general countable amenable groups. Thus the proof will change in one place. For the sake of completeness, we present it whole.

Recall that the entropy function on $\MGY$ is \usc, affine and the factor map $\pi$ applied to the sets of \im s, $\pi:\MGY\to\MGX$, is an affine surjection between Choquet simplexes, sending extreme points to extreme points (i.e., ergodic measures to ergodic measures). Now Propositon \ref{pushdown} implies that the extension entropy function $h^\pi$ is \usc\ and affine. It remains to show that it is a \se\ of the \ens\ of $X$.

We begin under the additional assumption that the space $X$ is \zd. Then we can choose the entropy structure $\H=(h_k)_{k\ge 0}$ obtained as $h_k(\mu) = h(\mu,\P_{[1,k]})$, where $(\P_k)_{k\in\N}$ is some jointly refining \sq\ of finite clopen partitions of $X$. Clearly, for each $k\ge 0$, we have $h^{\pi}\ge h\ge h_k$. By Proposition~\ref{fact2}~(3), the only thing we need to show is that $h^\pi - h_k$ is \usc. For each $k\in\N$ the partition $\P_{[1,k]}$ lifts against $\pi$ to a clopen partition of the symbolic space $Y$. By convention, this lifted partition will be denoted also by $\P_{[1,k]}$. Let $\Lambda$ denote the alphabet of the symbolic system $Y$. The symbol partition $\P_\Lambda$ of $Y$ generates the entire Borel sigma-algebra, hence the entropy of $h(\nu,Y)$ of every \im\ $\nu\in\MGY$ equals $h(\nu,\P_\Lambda)$. Obviously, it also equals $h(\nu,\P_\Lambda\vee\P_{[1,k]})$. On the other hand, $h(\nu,\P_{[1,k]})=h(\mu,\P_{[1,k]})$, where $\mu=\pi(\nu)$. So, for every $\nu\in\MGY$, the difference $h(\nu) - h_k(\pi(\nu))$ equals $h(\nu,\P_\Lambda\vee\P_{[1,k]}) - h(\nu,\P_{[1,k]}) = h(\nu,\P_\Lambda|\P_{[1,k]})$, and since both partitions are clopen, by Lemma \ref{doda}, the considered difference function is \usc\ on $\MGY$. Finally, by Propositon \ref{pushdown}, the push-down of the function $\nu\mapsto h(\nu) - h_k(\pi(\nu))$ is upper semicontinuous. On the other hand, this push-down evaluated at a measure $\mu\in\MGX$ equals $h^\pi(\mu)-h_k(\mu)$ (note that the function $\nu\mapsto h_k(\pi(\nu))$ is constant on fibers, thus it is not affected by pushing down). We have shown that $h^\pi-h_k$ is \usc\ on $\MGX$. So $h^\pi$ is indeed an affine \se\ of $\H$.
\smallskip

Now we will address the general case. For that we will need a simple entropy lemma (which is the same as for $\Z$-actions):

\begin{lemma}\label{simple}
Consider four \tl\ \ds s $X,X',X''$ and $X'''$, where $X'$ and $X''$ are \tl\ extensions of $X$ via factor maps $\pi'$ and $\pi''$, respectively, while $X'''$ is their \emph{fiber product}:
$$
X''' = \{(x',x''):\pi'(x')=\pi''(x'')\}\subset X'\times X''.
$$
Assume that both $X'$ and $X''$ have finite \tl\ entropy. Then, for any $\mu'''\in\M_G(X''')$ we have
$$
h(\mu''',X'''|X'')\le h(\mu',X'|X) \text{ \ and \ }h(\mu''',X'''|X')\le h(\mu'',X''|X),
$$
where $\mu'$ and $\mu''$ are the projections of $\mu'''$ onto $X'$ and $X''$, respectively.
\end{lemma}

\begin{proof}Clearly, all four considered systems have finite \tl\ entropy and $X'''$ is a \tl\ extension of both $X'$ and $X''$ via the coordinate projections $\mathsf{proj}_1$ and $\mathsf{proj}_2$, respectively. Moreover, it is an extension of $X$ via $\pi'\circ\mathsf{proj}_1=\pi''\circ\mathsf{proj}_2$. With this notation, let $\P$, $\P'$ and $\P''$ be finite measurable partitions of $X$, $X'$ and $X''$, respectively. By lifting, we can treat them as partitions of $X'''$. With this notation, we obviously have
$$
h(\mu''',\P'|\P''\vee\P)\le h(\mu''',\P'|\P),
$$
which can be written as
$$
h(\mu''',\P'\vee\P''\vee\P)-h(\mu''',\P''\vee\P)\le h(\mu''',\P'\vee\P)-h(\mu''',\P).
$$
Notice that since $X'''$ is a joining of $X'$ and $X''$, partitions of the form $\P'\vee\P''$ generate the sigma-algebra in $X'''$. This implies, that if $(\P_k)_{k\in\N}$, $(\P'_k)_{k\in\N}$ and $(\P''_k)_{k\in\N}$ are refining \sq s of partitions in $X$, $X'$ and $X''$, respectively, then applying the above to $\P_k,\ \P'_k,\ \P''_k$ and
letting $k\to\infty$, and because all the terms below are finite, we will obtain
$$
h(\mu''',X''')-h(\mu'',X'')\le h(\mu',X')-h(\mu,X),
$$
or
$$
h(\mu''',X'''|X'')\le h(\mu',X'|X).
$$
The other case is symmetric.
\end{proof}

We return to the main proof. Let $\H=(h_k)_{k\ge0}$ denote an \ens\ of $X$ with \usc\ differences (see Theorem \ref{niezalezy}). By Proposition \ref{fact2} (3), we need to show that $h^\pi-h_k$ is \usc, for each $k\ge 0$. We pick a principal \zd\ extension $\pi':X'\to X$, and a jointly refining \sq\ of clopen partitions $(\P'_k)_{k\in\N}$ of $X'$. The \sq\ $\H'=(h'_k)_{k\ge 0}$, where $h'_k(\mu')=h(\mu',\P'_{[1,k]})$ ($\mu'\in\MGXP$) is an \ens\ of $X'$ and, by definition of $\H$, $\H'$ is uniformly equivalent to $\H$ lifted against $\pi$ from $\MGX$ to $\MGXP$. Unfortunately, the symbolic extension $Y$ of $X$ need not be an extension of $X'$, so we cannot argue directly on $\MGXP$.

Fix some $k\ge1$ and let $Y'$ denote the fiber product of $Y$ and $X'$ (over the common factor $X$). By Lemma \ref{simple}, $Y'$ is a principal extension of $Y$.\footnote{From this place the proof differs from that in \cite{BD} or \cite{D1}.} On $Y'$ consider the partitions $\P_\Lambda$ (the symbol partition lifted from $Y$) and $\P'_{[1,k]}$ (lifted from $X'$). They are both clopen, so
$$
h(\nu',\P_\Lambda|\P'_{[1,k]})=h(\nu',\P_{\Lambda}\vee\P'_{[1,k]})-h(\nu',\P'_{[1,k]})
$$
is an \usc\ function of $\nu'\in\M_G(Y')$. By Proposition \ref{pushdown} (2), the push-down
$$
(h(\cdot,\P_\Lambda|\P'_{[1,k]}))^{\mathsf{proj}_2}
$$
from $\M_G(Y')$ to $\MGXP$ is \usc\ on $\MGXP$.

Because $\P_\Lambda$ generates the Borel sigma-algebra in $Y$, and $\mathsf{proj}_1$ is a principal extension, we have
$$
h(\nu',Y')\ge h(\nu',\P_{\Lambda}\vee\P'_{[1,k]})\ge h(\nu',\P_{\Lambda})=h(\nu,\P_\Lambda)=h(\nu,Y)=h(\nu',Y'),
$$
where $\nu = \mathsf{proj}_1(\nu')$. Clearly, $h(\nu',\P'_{[1,k]})=h(\mu',\P'_{[1,k]})=h'_k(\mu')$, where
$\mu' = \mathsf{proj}_2(\nu')$. We have shown that, on $\M_G(Y')$,
$$
h(\cdot,\P_\Lambda|\P'_{[1,k]})= h-h'_k\circ\mathsf{proj_2}.
$$
Because $h'_k\circ\mathsf{proj_2}$ is constant on fibers of $\mathsf{proj}_2$,
the (\usc) pushed down function $(h(\cdot,\P_\Lambda|\P'_{[1,k]}))^{\mathsf{proj}_2}$ on $\MGXP$ equals $h^{\mathsf{proj_2}}-h'_k$. This proves that $h^{\mathsf{proj_2}}$ is a \se\ of the entropy structure $\H'$ on $X'$. By definition, the entropy structure $\H$ lifted against $\pi'$ from $\MGX$ is uniformly equivalent to $\H'$, so, by Propositon~\ref{fact2}, $h^{\mathsf{proj_2}}$ is also a \se\ of the lifted structure $\H$. Since $\H$ has \usc\ differences, this means that $h^{\mathsf{proj_2}}-h_k$ is \usc\ on $\MGXP$, for each $k\ge 0$. We fix $k\ge 0$ again. Invoking Proposition~\ref{pushdown}~(2), one more time, we obtain that the push-down along $\pi'$, $(h^{\mathsf{proj_2}}-h_k)^{\pi'}$ is \usc\ on $\MGX$. But since here $h_k$ denotes the function lifted against $\pi'$, it is constant on fibers of $\pi'$, and the above \usc\ function equals $(h^{\mathsf{proj_2}})^{\pi'}-h_k=h^{\pi'\!\circ\,\mathsf{proj_2}}-h_k$. Because, on $Y'$,  $\pi'\circ\mathsf{proj_2}=\pi\circ\mathsf{proj_1}$, we obtain that, on $\MGX$, the function
$$
h^{\pi\circ\,\mathsf{proj_1}}-h_k=(h^{\mathsf{proj_1}})^{\pi}-h_k
$$
is \usc. Finally, because $Y'$ is a principal extension of $Y$,
$h^{\mathsf{proj_1}}=h$ on $\MGY$. We have proved that $h^\pi-h_k$ is \usc\ on $\MGX$, for each $k\in\N$. Thus, $h^\pi$ is a \se\ of $\H$.
\end{proof}

\section{Quasitilings and tiling systems}\label{s4}

From this place onward, everything in this paper has but one goal: proving (in possibly largest generality) the opposite direction of the Symbolic Extensions Entropy Theorem. Subsections \ref{nra} and \ref{ra} are based on the papers \cite{DHZ} and \cite{DH}.

\subsection{Terminology and facts not requiring amenability}\label{nra}

\smallskip

\subsubsection{Banach density}\label{bd}
\smallskip
Let $G$ be a countable group.

\begin{definition}\label{1.7}
For a subset $B\subset G$ and a finite set $F\subset G$ we denote
\[
\underline D_F(B)=\inf_{g\in G} \frac{|B\cap Fg|}{|F|}\text{ \ and \ } \overline D_F(B)=\sup_{g\in G} \frac{|B\cap Fg|}{|F|},
\]
\[
\underline D(B)=\sup_{F\subset G} \underline D_F(B)\text{ \ and \ }
\overline D(B)=\inf_{F\subset G} \overline D_F(B),
\]
where $F$ ranges over all finite subsets of $G$. The last two terms are called the \emph{lower} and \emph{upper Banach density} of $B$, respectively.
\end{definition}

\begin{remark}
The notions of upper and lower Banach density have been studied from several points of view.
For example, in \cite{BBF} the reader will find a different definition. It can be shown that that definition is in fact equivalent to ours.
\end{remark}

\begin{definition}\label{bad}
For two sets $A$ and $B$ of $G$ we define the following quantities
\[
\underline D_F(B,A)=\inf_{g\in G} \frac1{|F|}(|B\cap Fg|-|A\cap Fg|), \ \ \
\underline D(B,A)=\sup_{F\subset G} \underline D_F(B,A),
\]
where, as before, $F$ ranges over all finite subsets of $G$. The latter number will be called the \emph{Banach density advantage} of $B$ over $A$ (which can be negative, but we will never consider such a case).
\end{definition}

The following lemma will be repeatedly used in many of our considerations.
\begin{lemma}\label{bdc}
Let $F, F_1$ be finite subsets of $G$ and let $A,B$ be some arbitrary subsets of $G$.
If $F_1$ is $(F,\eps)$-\inv\ then $\underline D_{F_1}(B,A)\ge \underline D_F(B,A)-4\eps$.
\end{lemma}

\begin{proof}
Given $g\in G$, we have
$$
|B\cap Fhg|-|A\cap Fhg|\ge \underline D_F(B,A)|F|,
$$
for every $h\in F_1$. This implies that
\begin{multline*}
|\{(f,h): f\in F, h\in F_1, fhg\in B\}| - |\{(f,h): f\in F, h\in F_1, fhg\in A\}| \ge\\
\underline D_F(B,A)|F||F_1|.
\end{multline*}
This in turn implies that there exists at least one $f\in F$ for which
$$
|B\cap fF_1g|-|A\cap fF_1g|\ge \underline D_F(B,A)|F_1|.
$$
Since $f\in F$ and $F_1$ is $(F,\eps)$-invariant (and hence so is $F_1g$), we have
$$
\bigl||B\cap fF_1g|-|B\cap F_1g|\bigr|\le |fF_1\triangle F_1|=2|fF_1\setminus F_1|\le2|FF_1\setminus F_1|\le2\eps|F_1|,
$$
and the same for $A$, which yields
\begin{equation}
|B\cap F_1g|-|A\cap F_1g|\ge (\underline D_F(B,A)-4\eps)|F_1|.
\end{equation}
To end the proof, it remains to apply the infimum over all $g\in G$ on the left, and divide both sides by $|F_1|$.
\end{proof}

A set $A\subset G$ is called \emph{syndetic} (more precisely \emph{left syndetic}) if there exists a finite set $U\subset G$ such that $UA=G$ (equivalently, for each $g\in G$, $A\cap U^{-1}g\neq\emptyset$). The set $U$ will be referred to as the \emph{syndeticity set} for $A$, we will also say that $A$ is $U$-syndetic. In noncommutative groups left and right syndeticity are independent notions and throughout this paper right syndeticity will not be used. The following is an easy exercise:

\begin{prop}\label{p45}
A set $A\subset G$ is syndetic if and only if it has positive lower Banach density. The lower Banach density is at least $\frac1{|U|}$ where $U$ is a syndeticity set for~$A$.
\end{prop}

A set $A\subset G$ is called \emph{$F$-separated} (more precisely \emph{left $F$-separated}), where $F$ is another finite subset of $G$, if the sets $Fg$ for $g\in A$ are pairwise disjoint. The upper Banach density of an $F$-separated set is at most $\frac1{|F|}$.
Every $F$-separated set $A$ is contained in a maximal $F$-separated set $A'$ (i.e., such that $A'\cup\{g\}$ is not $F$-separated for any $g\in G\setminus A'$). Another nearly obvious fact is this:
\begin{prop}
Any maximal $F$-separated set is $(F^{-1}F)$-syndetic.
\end{prop}

\subsubsection{Quasitilings}\label{qtl}

\begin{definition}
A \emph{\qt} of $G$ is a countable family $\CT$ of finite sets $T\subset G$, called the \emph{tiles}, together with an injective map from $\CT$ to $G$ assigning to each tile $T$ a point $c_T\in T$ called the \emph{center} of $T$. The image of this injection, i.e., the set $C(\CT)=\{c_T:T\in\CT\}$ will be referred to as the \emph{set of centers} of $\CT$.
For each tile $T$, the set $S_T=Tc_T^{-1}$ will be called the \emph{shape} of $T$ (note that every shape contains the unity $e$). The \emph{collection of shapes} $\{S_T:T\in\CT\}$ will be denoted by $\CS(\CT)$. Given $S\in\CS(\CT)$, the set $C_S=C_S(\CT)=\{c_T\in C(\CT):S_T=S\}$ will be called the \emph{set of centers for the shape $S$}. Note that the sets of centers for different shapes are disjoint and their union over all shapes equals $C(\CT)$. A \qt\ $\CT$ is \emph{proper} if the collection of shapes $\CS(\CT)$ is finite. From now on, by a \qt\ we shall always mean a proper \qt.
\end{definition}

\begin{definition}\label{qt} Let $\eps\in[0,1)$ and $\alpha\in(0,1]$, and let $K\subset G$ be a finite set. A \qt\ $\CT$ is called
\begin{enumerate}
	\item \emph{$(K,\eps)$-invariant} if all shapes of $\CT$ are \emph{$(K,\eps)$-invariant};
	\item \emph{$\eps$-disjoint} if there exists a mapping $T\mapsto T^\circ$ (\,$T\in\CT$) such that
	\begin{itemize}
	\item $T^\circ$, is a $(1\!-\!\eps)$-subset of $T$ and
	\item the family $\{T^\circ:T\in\CT\}$ is disjoint;
	\end{itemize}
    \item \emph{disjoint} if the tiles of $\CT$ are pairwise disjoint;
	\item \emph{$\alpha$-covering} if $\underline D(\bigcup\CT)\ge\alpha$;
	\item a \emph{tiling} if it is a partition of $G$.
\end{enumerate}
By an \emph{$\varepsilon$-\qt} we shall mean a \qt\ which is both $\varepsilon$-disjoint and $(1-\varepsilon)$-covering.
\end{definition}

We have the following elementary fact:
\begin{prop}\label{syn}
For any $0<\varepsilon<1$, the set of centers $C(\CT)$ of a $(1\!-\!\varepsilon)$-covering \qt\ $\CT$ is syndetic.
\end{prop}
\begin{proof}
There exists a finite set $F$ such that for every $g\in G$, $\frac{|Fg\cap \bigcup\CT|}{|F|}\ge 1-\varepsilon>0$. In particular $Fg\cap \bigcup\CT\neq\emptyset$. Let $T_g$ denote a tile which intersects $Fg$. Since $\CT$ is proper, the set $V=\bigcup\CS(\CT)$ finite. The center $c_g$ of $T_g$ satisfies $T_gc_g^{-1}\subset V$ hence there exists $f\in F$ with $fgc_g^{-1}\in V$ and thus $g\in F^{-1}VC(\CT)$. We have shown that $C(\CT)$ is $F^{-1}V$-syndetic.
\end{proof}

We are about to define dynamical \qt s. For better differentiation of the notions, the \qt s defined so far will be referred to as \emph{static}. A static \qt\ $\CT$ can be identified with an element of the symbolic space ${\rm V}^G$ where ${\rm V}=
\{``S\,":S\in\CS(\CT)\}\cup\{0\}$. Namely, for each $S\in\CS(\CT)$ we place the symbol $``S\,"$ at all the centers $c\in C_S$, and we place the symbol $0$ at all remaining positions. Formally, we can write $\CT=\{\CT_g:g\in G\}$, where
$$
\CT_g=\begin{cases}``S\,"\,;& g\in C_S, S\in\CS(\CT),\\
0;& g\notin C(\CT).
\end{cases}
$$
\begin{definition} By a \emph{dynamical \qt} with the finite collection of shapes $\CS$ we will understand any subshift $\T\subset {\rm V}^G$, where ${\rm V}=\{``S\,":S\in\CS\}\cup\{0\}$ (the elements of $\T$ are interpreted as static \qt s $\CT$ with $\CS(\CT)\subset\CS$). The set $\CS$ will be also denoted as $\CS(\T)$. We will say that the dynamical \qt\ $\T$ is $(K,\varepsilon)$-\inv, $\varepsilon$-disjoint, disjoint, $\alpha$-covering, a \emph{dynamical $\varepsilon$-\qt} or a \emph{dynamical tiling} if all its elements are $(K,\varepsilon)$-\inv, $\varepsilon$-disjoint, disjoint, $\alpha$-covering, $\varepsilon$-\qt\ or tilings, respectively.
\end{definition}

Every static \qt\ $\CT$ generates a dynamical \qt\ $\T$ as its orbit-closure (under the shift action): $\T=\bar O(\CT)=\overline{\{g(\CT):g\in G\}}$.

\begin{lemma}\label{obvious}If a static \qt\ $\CT$ is $(K,\varepsilon)$-\inv\ ($\varepsilon$-disjoint, disjoint, $\alpha$-covering, or a tiling, where $\varepsilon,\alpha\in(0,1)$) with the collection of shapes $\CS(\CT)$ then $\T=\bar O(\CT)$ is a dynamical \qt\ with the collection of shapes $\CS(\T)=\CS(\CT)$, which is $(K,\varepsilon)$-\inv\ ($\varepsilon$-disjoint, disjoint, $\alpha$-covering, or a tiling, respectively).
\end{lemma}

\begin{proof}
 First assume that $\CT$ is $(K,\varepsilon)$-\inv. This property depends on the collection of shapes, so it is obviously maintained by other elements of $\bar O(\CT)$. The properties of $\varepsilon$-disjointness, disjointness, $\alpha$-covering or being a tiling are shift-\inv, so they pass to $g(\CT)$ for all $g\in G$. Next, the properties of $\varepsilon$-disjointness, disjointness and being a tiling are easily seen to be closed properties (i.e., inherited by limit points), so they pass to all elements of the orbit closure. It remains to show that if $\CT$ is $\alpha$-covering and for some \sq\ $(g_n)_{n\in\N}$ of elements of $G$ we have the convergence $g_n(\CT)\to\CT'$, then $\CT'$ is also $\alpha$-covering. So, suppose that $\CT'$ is not $\alpha$-covering, i.e., there is a positive number $\beta<\alpha$ such that $\CT'$ is at most $\beta$-covering. This means that for every finite set $F\subset G$ there exists $g_{\!_F}\in G$ such that
$$
\Bigl|\bigcup\CT'\cap Fg_{\!_F}\Bigr|\le \beta|F|.
$$
The convergence $g_n(\CT)\to\CT'$ implies that $\bigcup\CT'\cap Fg_{\!_F}=\bigcup g_n(\CT)\cap Fg_{\!_F}$ for all sufficiently large $n$. We pick one such $n$ and denote it by $n_{\!_F}$. Note that $\bigcup g_{n_{\!_F}}(\CT)=\bigcup\CT g^{-1}_{n_{\!_F}}$. Thus
$$
\beta|F|\ge\Bigl|\bigcup\CT g^{-1}_{n_{\!_F}}\cap Fg_{\!_F}\Bigr|=\Bigl|\bigcup\CT\cap Fg_{\!_F}g_{n_{\!_F}}\Bigr|.
$$
Because this holds for every finite $F\subset G$, the lower Banach density of $\bigcup\CT$ is at most $\beta$, i.e., $\CT$ is at most $\beta$-covering, a contradiction.
\end{proof}

Sometimes, we will be using a convention by which the term ``\qt'' will have a slightly extended meaning. Namely, we will admit that in the set of shapes $\CS(\CT)$ there are repeated terms treated as separate objects. In the symbolic representation, such shapes will be marked by different symbols. The process in which one shapes is marked by multiple
(still finitely many) symbols will be referred to as \emph{duplicating (the shapes)}. A dynamical \qt\ after duplicating becomes a \tl\ extension of the original. Duplicating does not affect any of the properties listed in Definition \ref{qt}.

% A dynamical \qt\ is \emph{minimal} if it is minimal as a subshift. Standard facts about minimal subshifts apply, in particular, a minimal dynamical \qt\ $\T$ equals the orbit-closure of every element $\CT\in\T$. For brevity, we will call a member of a minimal dynamical \qt\ a \emph{minimal \qt}. In a minimal dynamical \qt\ $\T$ every symbol (shape) $S$ which occurs at least once in at least one static \qt\ $\CT\in\T$, occurs in every $\CT\in\T$ syndetically, with a common syndeticity set.

\subsection{Terminology and facts requiring amenability}\label{ra}

\subsubsection{Banach density and syndeticity revisited}\label{bdr}

If $G$ is a countable amenable group with a F\o lner \sq\ $(F_n)_{n\in\N}$ then the upper and lower Banach densities can be evaluated as the limits along a F\o lner \sq:

\begin{prop}\label{bd1}
For any $A\subset G$, we have
$$
\underline{D}(A) = \lim_{n\to\infty}\underline D_{F_n}(A) \text{ \ \ and \ \ } \overline{D}(A) = \lim_{n\to\infty}\overline D_{F_n}(A),
$$
and for any $A,B\subset G$ we also have
$$
\underline{D}(B,A) = \lim_{n\to\infty}\underline D_{F_n}(B,A).
$$
\end{prop}

\begin{proof}
We will prove the third equality. Then, plugging in $A=\emptyset$ we will get the first equality and passing to the complement $B^c$ we will get the second equality. The inequality
$\limsup_{n\to\infty}\underline D_{F_n}(B,A)\le\sup\{\underline D_F(B,A): F\subset G, F \text{ is finite}\}$
is obvious. It remains to show that
$$
\liminf_{n\to\infty}\underline D_{F_n}(B,A)\ge\sup\{\underline D_F(B,A): F\subset G, F \text{ is finite}\}.
$$
At the same time this will prove the existence of all three limits.

Let $F\subset G$ be a finite set. Given $\eps>0$, for any $n$ large enough $F_n$ is $(F,\eps)$-invariant, hence Lemma \ref{bdc} implies that $\liminf_{n\to \infty} \underline D_{F_n}(B,A)\ge \underline D_F(B,A)-4\eps$. Since $\eps>0$ is arbitrary, it can be ignored.\end{proof}

\begin{cor}\label{coro}We have
$$
\underline D(B)-\overline D(A)\le \underline D(B,A).
$$
\end{cor}
\begin{proof} Fix a F\o lner \sq\ $(F_n)_{n\in\N}$. By the above lemma, we can write
\begin{multline*}
\underline D(B,A)= \lim_{n\to\infty} \ \inf_{g\in G} \frac{|B\cap F_ng|-|A\cap F_ng|}{|F_n|}\ge\\
\lim_{n\to\infty} \ \inf_{g\in G} \frac{|B\cap F_ng|}{|F_n|}- \lim_{n\to\infty} \ \sup_{g\in G}\frac{|A\cap F_ng|}{|F_n|}= \underline D(B)-\overline D(A).
\end{multline*}
\end{proof}

We also have the following:

\begin{prop}\label{suba}
Upper Banach density is subadditive: for any $A,B\subset G$,
$$
\overline{D}(A\cup B)\le \overline{D}(A) + \overline{D}(B).
$$
\end{prop}

\begin{proof}
For every $n\in\N$ and $g\in G$, we have
$$
\frac{|(A\cup B)\cap F_ng|}{|F_n|}\le\frac{|A\cap F_ng|}{|F_n|}+\frac{|B\cap F_ng|}{|F_n|}.
$$
Hence,
$$
\sup_{g\in G}\frac{|(A\cup B)\cap F_ng|}{|F_n|}\le \sup_{g\in G}\frac{|A\cap F_ng|}{|F_n|}+\sup_{g\in G}\frac{|B\cap F_ng|}{|F_n|}.
$$
Passing to the limit over $n$ ends the proof.\footnote{We remark that in non-amenable groups, instead of taking the limit we would have to apply the infimum over all finite sets $F$, which spoils the proof.}
\end{proof}

%The next fact connects upper Banach density with \im s.
%\begin{definition}
%Let a countable amenable group $G$ act on a compact metric space $G$. The \emph{orbit capacity} of a set $\mathsf A\subset X$ is defined as
%$$
%\(\mathsf A)=\sup_{x\in X}\overline D(\{g\in G:g(x)\in\mathsf A\}).
%$$
%\end{definition}

%\begin{lemma}\label{ocap}
%If $\mathsf A\subset X$ is a measurable set then
%$$
%\sup\{\mu(\mathsf A):\mu\in\MGX\}\le\ocap(\mathsf A).
%$$
%\end{lemma}
%\begin{proof}
%The proof is standard and involves the ergodic theorem. For a countable amenable group $G$ we shall use the version proved by E. Lindestrauss \cite{L}: Every F\o lner \sq\ in $G$, $(F_n)_{n\in\N}$, has a sub\sq\ $(F_{n_k})_{k\in\N}$ such that for every ergodic action of $G$ on a standard probability space $(X,\Sigma,\mu)$ and every set $A\in\Sigma$ we have, for almost every $x\in X$,
%$$
%\mu(\mathsf A)=\lim_k \frac{|\{g\in F_{n_k}:g(x)\in\mathsf A\}|}{|F_{n_k}|}.
%$$
%With this theorem in hand, pick any number $a<\sup\{\mu(\mathsf A):\mu\in\MGX\}$. There exists an ergodic measure $\mu$ with $\mu(\mathsf A)>a$, implying $\frac1{|F_n|}|\{g\in F_n:f(x)\in\mathsf A\}|=\frac1{|F_n|}|\{g\in G:g(x)\in\mathsf A\}\cap F_n|>a$ for some $x\in X$ and arbitrarily large $n$. This implies that $\ocap(\mathsf A)\ge a$, which, by the choice of $a$, ends the proof.
%\end{proof}

At some point, we will be needing the following elementary fact (this is basically \cite[Lemma 3.4]{DHZ}, where it is formulated using the language of \qt s).

\begin{lemma}\label{1.5}
Let $(A_k)_{k\ge 1}$ and $(g_k)_{k\ge 1}$ be a sequence of subsets of $G$ and a \sq\ of elements of $G$, respectively, such that:
\begin{enumerate}
	\item the union $\bigcup_{k=1}^\infty A_k$ is finite,
	\item $A=\bigcup_{k=1}^\infty A_kg_k$ is a disjoint union.
\end{enumerate}
For each $k$ let $B_k\subset A_k$ and let $B=\bigcup_{k=1}^\infty B_kg_k$. Then
$$
\underline D(B)\ge\underline D(A)\cdot\inf_k\frac{|B_k|}{|A_k|}.
$$
\end{lemma}

\begin{proof} Let $\alpha = \inf_k\frac{|B_k|}{|A_k|}$.
Given $n\in\N$ and $g\in G$, denote
$$
A(n,g)=\bigcup\{A_kg_k:A_kg_k\subset F_ng\} \text{ \ \ and \ \ } B(n,g)=\bigcup\{B_kg_k:A_kg_k\subset F_ng\}.
$$
Clearly
$$
\frac{|B(n,g)|}{|A(n,g)|}\ge \alpha.
$$
Denote $K=\bigcup_{k=1}^\infty A_k$ (by assumption, this is a finite set).  As easily verified, each element of the difference $F_ng\cap A\setminus A(n,g)$ lies outside the $KK^{-1}$-core of $F_ng$. Clearly, this core equals $(F_n)_{KK^{-1}}g$, where $(F_n)_{KK^{-1}}$ is the $KK^{-1}$-core of $F_n$.
So,
$$
|F_ng\cap B|\ge|B(n,g)|\ge\alpha|A(n,g)|\ge\alpha(|F_ng\cap A|-|F_n\setminus(F_n)_{KK^{-1}}|).
$$
Taking infimum over all $g\in G$ and dividing both sides by $|F_n|$, we obtain that
$$
\underline D_{F_n}(B)\ge \alpha\underline D_{F_n}(A)-\epsilon_n
$$
where, by the property (2) in subsection \ref{2.2}, $\epsilon_n\to 0$ with $n$. Applying the limit as $n$ tends to infinity we complete the proof.
\end{proof}

\subsubsection{Special $\varepsilon$-\qt s}\label{spqt}
We begin with citing some theorems about the existence of special \qt s in any countable amenable group $G$ in which we fix a symmetric, centered F\o lner \sq\ $(F_n)_{n\in\N}$.

\begin{theorem}\label{OW}(\cite[Lemma 4.1]{DHZ}, see also \cite{OW})
For any $\varepsilon>0$ there exists an integer $r(\varepsilon)$ such that for any $n$ there exists
a static $\varepsilon$-\qt\ $\CT$ of $G$ with the collection of shapes $\CS(\CT)\subset\{F_{n_1}, F_{n_2}, \dots, F_{n_{r(\varepsilon)}}\}$, where $n< n_1<n_2<\cdots<n_{r(\varepsilon)}$.
\end{theorem}

It is seen that, for large $n$, the \tl\ entropy of the dynamical \qt\ generated by the above \qt\ $\CT$ is small (symbols other than zero appear with small upper Banach density). Since we need this entropy to be zero, we shall use a different combination of results:

\begin{theorem}\label{DHZ}(\cite[Theorem 6.1]{DHZ})
If $G$ is a countable amenable group then there exists of a free action of $G$ on
a \zd\ space, which has \tl\ entropy zero.
\end{theorem}

\begin{theorem}\label{DH}(\cite[Lemma 3.4]{DH})
Given a free action of $G$ on a compact metric \zd\ space $X$, $\varepsilon>0$ and $n\in\N$, there exists a dynamical $\varepsilon$-\qt\ $\T$ of $G$ with the collection of shapes $\CS(\T)\subset\{F_{n_1}, F_{n_2}, \dots, F_{n_{r(\varepsilon)}}\}$, where $n<n_1<n_2<\cdots<n_{r(\varepsilon)}$ (the dependence $\varepsilon\mapsto r(\varepsilon)$ is the same as in Theorem \ref{OW}), and which is a \tl\ factor of $X$.
\end{theorem}

\begin{theorem}\label{DH1}(\cite[Corollary 3.5]{DH}) The above \qt\ $\T$ can be transformed to a \qt\ $\hat\T$ with the following properties
\begin{enumerate}
	\item $\hat\T$ remains a \tl\ factor of $X$,
	\item $\hat\T$ is disjoint and covers the same part of $G$ as $\T$ (hence $\hat\T$ is
	$(1\!-\!\varepsilon)$-covering),
	\item each shape of $\hat\T$ is a $(1\!-\!\varepsilon)$-subset of one of the shapes of $\T$.
\end{enumerate}
\end{theorem}

The disjoint \qt\ $\hat\T$ is created from the non-disjoint \qt\ $\T$ by replacing the tiles of each $\CT\in\T$ by their subsets. But then the centers may fall outside the new tiles and we need to perform the following \emph{adjustement of centers}. For each shape $\hat S\in\CS(\hat\T)$ we choose a point $a_{\hat S}\in\hat S$ (a ``new center''). We define a new set of shapes $\hat\CS'=\{\hat Sa_{\hat S}^{-1}:\hat S\in\CS(\hat\T)\}$ (notice that each new shape contains the unity, as required). Next, for each $\hat\CT\in\hat\T$ we rewrite each tile $\hat T=\hat Sc\in\hat\CT$ ($\hat S\in\CS(\hat\T)$ and $c\in C_{\hat S}$), as $\hat T=\hat T'=\hat S'c'$ where $\hat S'=\hat Sa_{\hat S}^{-1}$ and $c'=a_{\hat S}c$. We let $\hat\CT'$ be the \qt\ with the same (disjoint) tiles as $\hat\CT$ but interpreted as $\hat T'$ rather than $\hat T$. Clearly, $\hat\CT'$ is a \qt\ with the set of shapes $\hat\CS'$ and the center sets $C_{\hat S'}$ ($\hat S'\in\hat\CS'$) equal to $a_{\hat S}C_{\hat S}$, where $C_{\hat S}$ is the center set for the old shape $\hat S\in\CS(\hat\T)$. Since the \qt\ $\hat\CT$ is disjoint, the new center sets for different shapes in $\hat\CS'$ are disjoint, as required. It is clear that the mapping $\hat\CT\mapsto\hat\CT'$ is a \tl\ conjugacy between $\hat\T$ and it image $\hat\T'$. This map preserves the tiles and just moves the centers to new locations within the tiles.
\smallskip

As far as tilings are concerned, we have at our disposal the following general result:

\begin{theorem}\label{ext} ({\rm follows from \cite[Theorem 5.2]{DHZ}}) For any countable amenable group $G$, any $\varepsilon>0$ and any finite set $K\subset G$, there exists a $(K,\varepsilon)$-\inv\ dynamical tiling $\T$ of $G$ of entropy zero.
\end{theorem}
As can be seen from the construction, the above tiling is ``made from'' the disjoint \qt\ of Theorem \ref{OW} and its collection of shapes can be divided into $r(\varepsilon)$ classes such that each shape in the $i$th class is an $\varepsilon$-modification of $F_{n_i}$ ($i=1,2,\dots,r(\varepsilon)$). However, the algorithm of creating the tiling from the \qt\ is not given by a block code (i.e., it is not a \tl\ factor map), moreover, there is no estimate on the number of shapes of the tiling, which a priori can be uncontrollably large. Zero entropy is due to a relatively small number of configurations of tiles in F\o lner sets much larger than the tiles.
\smallskip

\subsubsection{F\o lner systems of \qt s and tiling systems}
One dynamical \qt\ (or tiling) is insufficient for the construction of a symbolic extension. What we need is a countable joining of a \sq\ of dynamical \qt s (tilings), with improving disjointness, covering and invariance properties. In case of dynamical \qt s, this is all we are asking for. We make the following definition.

\begin{definition}
Let $(\epsilon_k)_{k\in\N}$ be a decreasing to $0$ \sq\ of positive numbers. Let $\mathbf T=\bigvee_{k\in\N}\T_k$ be a \tl\ joining of a \sq\ of dynamical \qt s of $G$, such that for every $k\in\N$, $\T_k$ is a (dynamical) $\epsilon_k$-\qt, and for every $\varepsilon>0$ and finite set $K\subset G$, for $k$ sufficiently large all shapes of $\T_k$ are $(K,\varepsilon)$-\inv. Such $\mathbf T$ will be called a \emph{F\o lner system of \qt s}. The elements of $\mathbf T$ will be denoted by $\boldsymbol\CT=(\CT_k)_{k\in\N}$ ($\forall_{k\in\N}\,\CT_k\in\T_k$). The collection of shapes of $\T_k$, $\CS(\T_k)$, will be abbreviated as $\CS_k$.
\end{definition}

The term ``F\o lner system of \qt s'' comes from the fact that it is a \tl\ \ds\ and from the
observation that the last requirement in the above definition can be formulated differently: the joint collection of shapes $\bigcup_{k\in\N}\CS_k$, indexed (bijectively, but in an arbitrary order) by natural numbers, is a F\o lner \sq\ in $G$.

The existence of F\o lner systems of \qt s in any countable amenable group follows directly from Theorem \ref{OW}, also directly from Theorem \ref{DH} one can find such a system as a \tl\ factor of any free action of $G$ on a compact metric \zd\ space. As a corollary, there exist F\o lner systems of \qt s with \tl\ entropy zero. Using Theorem \ref{DH1} we can even require the \qt s to be disjoint.

We can use Theorem \ref{ext} to deduce a similar corollary for tilings, but we cannot claim that a F\o lner system of tilings (or even one dynamical tiling) appears as a \tl\ factor in every free \zd\ action of $G$. We will claim this later under an additional assumption on $G$.

In case of a F\o lner system of tilings (rather than \qt s) we can demand the members of the joining to ``interact'' with each other in a more specific manner. Two key such interactions are congruency and determinism, as defined below:

\begin{definition}
\begin{enumerate}
	\item A F\o lner system of tilings $\mathbf T=\bigvee_{k\in\N}\T_k$ is
	\emph{congruent} if for each $\boldsymbol{\CT}=(\CT_k)_{k\in\N}\in\mathbf T$, for every
	$k\in\N$, every tile of $\CT_{k+1}$ is a union of some tiles of $\CT_k$.
	\item A congruent F\o lner system of tilings $\mathbf T=\bigvee_{k\in\N}\T_k$ is
	\emph{deterministic}, if, for each $k\in\N$ and every shape $S'\in\CS_{k+1}$,
	there exist sets $C_S(S')\subset S'$ ($S\in\CS_k$) such that
	$$
	S'=\bigcup_{S\in\CS_k}\ \bigcup_{c\in C_S(S')}Sc \text{\ \ (disjoint union),}
	$$
	and for each $\boldsymbol{\CT}=(\CT_k)_{k\in\N}\in\mathbf T$, whenever $S'c'$ is
	a tile of $\CT_{k+1}$ then the sets $Scc'$ with $S\in\CS_k$ and $c\in C_S(S')$ are tiles
	of $\CT_k$. We also define $C_k(S')=\bigcup_{S\in\CS_k}C_S(S')$.
\end{enumerate}
\end{definition}

In the deterministic case, each static tiling $\CT_{k+1}$ \emph{determines} the tiling
$\mathcal T_k$ joined with it, because each of the tiles of $\CT_{k+1}$ is partitioned into the tiles of $\CT_k$ in a unique way determined by its shape. Clearly, the assignment $\CT_{k+1}\mapsto\CT_k$ is given by a block code. Thus, the joining $\mathbf T$ is in fact an inverse limit
$$
\mathbf T=\overset\leftarrow{\lim_k}\T_k.
$$

\begin{remark}\label{remark}
Any congruent F\o lner system of tilings $\mathbf T=(\T_k)_{k\in\N}$ can be easily made deterministic in an inductive process of duplicating the shapes (see end of subsection \ref{qtl}), as follows: For each $S'\in\CS_{k+1}$ there are only finitely many, say $m(S')$, possible partitions of $S'$ into tiles from $\CS_k$. We duplicate the tile $S'$ into $m(S')$ copies (identical as subsets of $G$, but in the symbolic representation of the tiling we will now associate to them different symbols, say $``S'_1",``S'_2",\dots,``S'_{m(S')}"$) and use them for tiles with the original shape $S'$ according to how the tile is subdivided by the tiles of $\CT_k$. Clearly, this process increases the cardinality of $\CS_k$ for each $k\ge 2$,
but it is a \tl\ conjugacy, so the dynamical properties (for example the \tl\ entropy) of $\mathbf T$ remain unchanged.
\end{remark}

\begin{definition}
A congruent, deterministic, F\o lner system of tilings will be briefly called a \emph{tiling system}.
\end{definition}

The already cited \cite[Theorem 5.2]{DHZ} in full strength, translated to the terminology introduced above, states:

\begin{theorem}\label{fs}
If $G$ is a countable amenable group then there exists a tiling system of $G$ with \tl\ entropy zero.
\end{theorem}
(Determinism is implicit in the proof in a way essential in obtaining \tl\ entropy zero).
As mentioned before, in general we cannot claim that an arbitrary free \zd\ action of $G$ has a tiling system as a factor.

\subsection{Tiled entropy}\label{te}
Throughout this section we assume that $\mathbf T=\overset\leftarrow{\lim_k}\T_k$ is a tiling system of $G$ with \tl\ entropy zero. Recall that the set of shapes of $\T_k$ is denoted by $\CS_k$, and given $\boldsymbol\CT=(\CT_k)_{k\in\N}\in\mathbf T$, the set of centers of $\CT_k$ of the tiles with shape $S\in\CS_k$ is denoted by $C_S(\CT_k)$. The set of all centers of $\CT_k$ is $C(\CT_k)=\bigcup_{S\in\CS_k}C_S(\CT_k)$. We now introduce more notation. For $S\in\CS_k$, and $s\in S$ by $[S,s]$ we denote the set of elements $\boldsymbol\CT\in\mathbf T$ for which $s^{-1}$ belongs to $C_S(\CT_k)$. If $\boldsymbol\CT\in[S,s]$ then $Ss^{-1}$ is the tile of $\CT_k$ which contains the unity, i.e., \emph{the central tile of $\CT_k$}. The set $[S,e]$ will be abbreviated as $[S]$.
Observe that  $\boldsymbol\CT\in[S]$ if and only if $\CT_{k,e}=``S\,"$, so the notation is consistent with that of one-symbol cylinders over the alphabet ${\rm V}_k=\{``S\,": S\in\CS_k\}\cup\{0\}$.
The family $\D_{\CS_k}=\{[S,s]:S\in\CS_k,\ s\in S\}$ is a partition of $\mathbf T$. Also note that $\boldsymbol\CT\in[S,s]$ if and only if $s^{-1}(\boldsymbol\CT)\in[S]$, i.e., $[S,s]=s([S])$. So, if $\nu$ is a shift-\im\ on $\mathbf T$, then $\nu([S,s])=\nu([S])$ for all $s\in S$.

Recall, that by congruency and determinism of the tiling system, whenever $k'>k$, every shape $S'$ of $\T_{k'}$ decomposes in a unique way as a concatenation of shifted shapes of $\T_k$ and the set of centers of these tiles is denoted by $C_k(S')$. The subset of $C_k(S')$ consisting of centers of tiles with a particular shape $S\in\CS(\T_k)$ is denoted by $C_S(S')$ (now the subscript $k$ is not needed; $k$ is determined by $S$).

\smallskip
Let $G$ act on a \zd\ compact metric space $X$ given in its array representation $X=\overset\leftarrow{\lim_k} X_{[1,k]}$, and let us assume that $X$ has the tiling system $\mathbf T$ as a \tl\ factor (in the following chapter, we will replace $X$ by its joining with $\mathbf T$, so this assumption will be fulfilled). Then we can combine the layers of $X$ (i.e., the subshifts $X_k$) with the layers of $\mathbf T$ (i.e., the dynamical tilings $\T_k$) by replacing each $X_k$ by its \tl\ joining $\bar X_k$ with $\T_k$ realized naturally in the common extension $X$. The combined alphabet of $\bar X_k$ is $\bar\Lambda_k=\Lambda_k\times \rm V_k$, while that of $\bar X_{[1,k]}$ is $\bar\Lambda_{[1,k]}=\prod_{l=1}^k\bar\Lambda_l$.  In this manner, the system on $X$ is replaced by its \tl ly conjugate model which is the inverse limit $\bar X = \overset\leftarrow{\lim_k} \bar X_{[1,k]}$. We will call $\bar X$ the \emph{tiled array representation of $X$}.

We will use the following notational convention. For $\bar x\in\bar X$ and $S\in\CS_k$, the expression $\bar x_g=``S\,"$ means that in the tiling $\CT_k$ apparent in the $k$th layer of $\bar x$, at the position $g$ there occurs a center of a tile with shape $S$. Formally, this means that if $\bar x_k$ is the $k$th layer of $\bar x$ then
$$
\bar x_{k,g} \in\Lambda_k\times\{``S\,"\}\subset\bar\Lambda_k.
$$
With this convention, the notions $[S]$ and $[S,s]$ ($S\in\CS_k,\ s\in S$) may be applied to $\bar X$ in the following way:
$$
[S]=\{\bar x\in \bar X: \bar x_e=``S\,"\}, \ \ [S,s]=s([S]) =\{\bar x\in \bar X: \bar x_{s^{-1}}=``S\,"\}.
$$
\subsubsection{Tiled entropy and its monotonicity}
We continue to assume that a \zd\ \ds\ $X$ has a tiling system $\mathbf T$ as a \tl\ factor and we denote by $\bar X$ the tiled array representation of $X$.
\begin{definition}
Let $\P$ and $\Q$ be two finite measurable partitions of $\bar X$, $|\P|>1$. Let $\mu$ be  a probability measure on $\bar X$. By the \emph{$k$th tiled entropy of $\P$} and \emph{$k$th conditional tiled entropy of $\P$ given $\Q$} with respect to $\mu$ we will mean the following terms:
$$
H_{\T_k}(\mu,\P)=\sum_{S\in\CS_k}\mu([S])H(\mu_{[S]},\P^S), \ \ H_{\T_k}(\mu,\P|\Q)=\sum_{S\in\CS_k}\mu([S])H(\mu_{[S]},\P^S|\Q^S),
$$
where $\mu_{[S]}$ is the normalized conditional measure $\mu$ on $[S]$.
\end{definition}
Alternatively, one can define just the unconditional version and then put
$$
H_{\T_k}(\mu,\P|\Q)=H_{\T_k}(\mu,\P\vee\Q)-H_{\T_k}(\mu,\Q).
$$
In general, the tiled entropy cannot be easily reduced to a standard notion of conditional entropy (except in some cases, see formula \eqref{przek} below) and indeed requires a separate definition; it resembles a conditional entropy given the partition $\D_{\CS_k}$, but it takes into account only selected elements of this partition (the cylinders $[S]$, $S\in\CS_k$) and on each cylinder a different power of $\P$ is considered.
%The above notions can be viewed as ``partial'' conditional entropies of $\P$ given $\D_{\S_k}$ and of $\P$ given the join $\Q\vee\D_{\S_k}$, respectively. NOT TRUE.
%We will say that a \sq\ of real functions $(f_k)_{k\in\N}$ defined on some common domain converge to some $f$ \emph{almost decreasingly} if there exists a decreasing to zero \sq\ of positive numbers $(\delta_k)_{k\in\N}$, such that the \sq\ of functions $(f_k+\delta_k)_{k\in\N}$ is decreasing.

\begin{theorem}\label{tte}
On the simplex of \im s $\M_G(\bar X)$, the \sq s of tiled entropies $(H_{\T_k}(\mu,\P))_{k\in\N}$ and $(H_{\T_k}(\mu,\P|\Q))_{k\in\N}$, converge decreasingly to $h(\mu,\P)$ and $h(\mu,\P|\Q)$, respectively.
\end{theorem}

\begin{proof}
We begin by showing that $H_{\T_{k+1}}(\mu,\P)\le H_{\T_k}(\mu,\P)$. The proof for the conditional entropy is identical (we only use subadditivity).
Recall that each shape $S'\in\CS_{k+1}$ decomposes as a disjoint union of shifted shapes of $\T_k$:
\begin{equation}\label{hj}
S'=\bigcup_{S\in\CS_k}\ \bigcup_{c\in C_S(S')}Sc.
\end{equation}
Consider a point $\bar x\in[S]$,  i.e., such that in the tiling $\CT_k$ associated to $\bar x$, a tile of shape $S$ occurs centered at $e$. Let $S'c^{-1}$ denote the central tile in the tiling $\CT_{k+1}$ associated to $\bar x$ (i.e., the tile containing $e$). Then $\bar x\in [S',c] = c([S'])$. By congruency, $S'c^{-1}$ contains the tile $S$ as its component in the decomposition into the shifted shapes from $\CS_k$. Equivalently, $S'$ contains $Sc$ in its decomposition, which means that $c\in C_S(S')$. We conclude that
\begin{equation}\label{dupa}
[S]=\bigcup_{S'\in\CS_{k+1}}\ \bigcup_{c\in C_S(S')}[S',c],
\end{equation}
which is a (disjoint) union of some atoms of the partition $\D_{\CS_{k+1}}$.
Obviously, we have
$$
H(\mu_{[S]},\P^S)\ge H(\mu_{[S]},\P^S|\D_{\CS_{k+1}}),
$$
and, according to \eqref{dupa}, whenever $S'\in \CS_{k+1}$ and $c\in C_S(S')$, one has $(\mu_{[S]})_{[S',c]}=\mu_{[S',c]}$. Thus the conditional entropy on the right equals
$$
\sum_{S'\in\CS_{k+1}}\ \sum_{c\in C_S(S')}\frac{\mu([S',c])}{\mu([S])}H(\mu_{[S',c]},\P^S).
$$

On the other hand, the term $H(\mu_{[S']},\P^{S'})$ represents the entropy of $\P^{S'}$ restricted to $[S']$ (and with regard to the normalized measure on $[S']$). By the decomposition of $S'$ (see \eqref{hj}) and subadditivity of entropy (and using invariance of $\mu$ for the first and second equalities), we have
\begin{multline*}
H(\mu_{[S']},\P^{S'})\le \sum_{S\in\CS_k}\ \sum_{c\in C_S(S')} H(\mu_{[S']},\P^{Sc})=
\sum_{S\in\CS_k}\ \sum_{c\in C_S(S')} H(\mu_{c([S'])},\P^S)=\\
\frac1{\mu([S'])} \sum_{S\in\CS_k}\ \sum_{c\in C_S(S')}\mu([S',c])H(\mu_{[S',c]},\P^S).
\end{multline*}
After multiplying both sides by $\mu([S'])$ and summing over $S'\in\CS_{k+1}$, we obtain
\begin{multline*}
H_{\T_{k+1}}(\mu,\P)\le
\sum_{S'\in\CS_{k+1}}\ \sum_{S\in\CS_k}\ \sum_{c\in C_S(S')}\mu([S',c])H(\mu_{[S',c]},\P^S) =\\
\sum_{S\in\CS_k}\ \mu([S])\sum_{S'\in\CS_{k+1}}\ \sum_{c\in C_S(S')}\frac{\mu([S',c])}{\mu([S])}H(\mu_{[S',c]},\P^S) =\\
\sum_{S\in\CS_k}\ \mu([S])H(\mu_{[S]},\P^S|\D_{\CS_{k+1}})
\le \sum_{S\in\CS_k}\ \mu([S])H(\mu_{[S]},\P^S)=H_{\T_k}(\mu,\P).
\end{multline*}

\smallskip
We pass to proving the convergence to the appropriate limits. It suffices to prove the unconditional version; the conditional version will follow straightforward, by subtraction
(recall that $H_{\T_k}(\mu,\P|\Q)=H_{\T_k}(\mu,\P\vee\Q)-H_{\T_k}(\mu,\Q)$).
We will show that given $k\in\N$ and $\delta>0$, we have
\begin{equation}\label{wte}
\frac1{|F_n|}H(\mu,\P^{F_n})\le H_{\T_k}(\mu,\P)+\delta
\end{equation}
for all sufficiently large $n$, and conversely, that given $n\in\N$ and $\delta>0$ we have
\begin{equation}\label{wtamte}
H_{\T_k}(\mu,\P)\le \frac1{|F_n|}H(\mu,\P^{F_n})+\delta
\end{equation}
for all sufficiently large $k$. This will imply the desired convergence.

Fix $k$ and $\delta$. From now on we will skip the index $k$ in objects associated to the tiling $\T_k$ (for instance, the set of tiles of $\T$ will be denoted by $\CS$). Let $\gamma=\frac\delta{2\log|\P|}$. Let $n$ be so large that $F_n$ is $(\bigcup\CS,\frac\gamma{2|\bigcup\CS|})$-\inv. We will abbreviate $F_n$ as $F$. Given a tiling $\CT\in\T$ let $\F_\CT$ denote collection of all tiles of $\CT$ with centers in $F$, i.e.,
$$
\F_\CT=\{Sc:S\in\CS,\ c\in C_S(\CT)\cap F\}.
$$
We also let $F_\CT = \bigcup\F_\CT$. The parameters of the invariance of $F$ were selected so that $F_\CT$ is a $\gamma$-modification of $F$ ($F_\CT$ is contained in $\bigcup\CS F$ and contains the $(\bigcup\CS)^{-1}$-core of $F$; now use property (3) above Definition \ref{tyi}). With $\CT$ ranging over $\T$, there are finitely many possibilities for $\F_\CT$. Let $\D$ denote the (obviously finite and measurable) partition of $\bar X$ according to which of these possibilities occurs. For each $D\in\D$ the (common for all $\CT\in D$) corresponding family $\F_\CT$ will be denoted by $\F_D$. The union $F_D=\bigcup\F_D$ is a $\gamma$-modification of $F$. Also, for any $S\in\CS$, by $C_S^D(F)$ we will denote the (common for all $\CT\in D$) set of elements of $F$ which are centers of tiles of $\CT$ of the shape $S$. We can write
\begin{multline*}
H(\mu,\P^F)\le H(\mu,\P^F|\D)+H(\mu,\D)=\sum_{D\in\D}\mu(D)H(\mu_D,\P^F)+H(\mu,\D) \le\\
\sum_{D\in\D}\mu(D)H(\mu_D,\P^{F_D})+\eta+H(\mu,\D) = L+\eta+H(\mu,\D),
\end{multline*}
where $L=\sum_{D\in\D}\mu(D)H(\mu_D,\P^{F_D})$ and $\eta = \gamma|F|\log(|\P|)=|F|\frac\delta2$. We need to estimate the term $L$ (we will take care $H(\mu,\D)$ later). By subadditivity of entropy,
$$
H(\mu_D,\P^{F_D})\le\sum_{S\in\CS}\ \sum_{c\in C_S^D(F)}H(\mu_D,\P^{Sc}),
$$
and by invariance of $\mu$, $H(\mu_D,\P^{Sc})$ can be replaced by $H(\mu_{c(D)},\P^S)$. So,
\begin{multline*}
L\le\sum_{D\in\D}\mu(D)\ \sum_{S\in\CS}\ \sum_{c\in C_S^D(F)}H(\mu_{c(D)},\P^S)=\\
\sum_{S\in\CS}\ \sum_{D\in\D}\ \sum_{c\in C_S^D(F)}\mu(c(D))H(\mu_{c(D)},\P^S).
\end{multline*}
Every set $c(D)$ with $D\in\D$ and $c\in C_S^D(F)$ is contained in $[S]$. Moreover, every point $\bar x\in[S]$ belongs to the sets $c(D)$ for exactly $|F|$ pairs $(D,c)$ with $D\in\D,\ c\in C_S^D(F)$, each time for a different value of $c$ (although we do not claim that for different pairs $(D,c)$ the sets $c(D)$ are always different). Indeed, for every $c\in F$ the point $c^{-1}(\bar x)$ has a tile of shape $S$ centered at $c$, hence it belongs to some $D\in\D$ such that $c\in C_S^D(F)$, and then $\bar x$ belongs to $c(D)$. If $\bar x$ belonged to $c(D)$ for more than $|F|$ pairs $(D,c)$ with $D\in\D, c\in C_S^D(F)$ then some value of $c\in F$ would have to repeat, implying that $c^{-1}(\bar x)$ would belong to two different sets $D$, which is impossible. From these facts we conclude that for each $c\in F$, the family $\{c(D): D\in\D \text{ such that } c\in C_S^D(F)\}$ is a partition of $[S]$. We will denote it by $\mathcal E^c_{[S]}$.

Since for each $E\in\mathcal E^c_{[S]}$ we have $(\mu_{[S]})_E=\mu_E$, the last triple sum can be rearranged, as follows:
\begin{multline*}
\sum_{S\in\CS}\ \sum_{c\in F}\ \sum_{E\in \mathcal E^c_{[S]}}\mu(E)H(\mu_{E},\P^S)=
\sum_{c\in F}\ \sum_{S\in\CS}\ \mu([S])H(\mu_{[S]},\P^S|\mathcal E^c_{[S]})\le  \\
\sum_{c\in F}\ \sum_{S\in\CS}\ \mu([S])H(\mu_{[S]},\P^S) = |F|H_{\T}(\mu,\P).
\end{multline*}
We have obtained
$$
\frac1{|F|}H(\mu,P^F)\le H_{\T}(\mu,\P) + \frac\delta2 + \frac1{|F|}H(\mu,\D).
$$

The partition $\D$ depends solely on the tiling $\CT$ restricted to the set $F$. We can write this as $\D\preccurlyeq{\rm V}^F$ (recall that ${\rm V}$ is the alphabet used by the tiling $\T$; here it is identified with the zero-coordinate partition of $\T$). Thus $\frac1{|F|}H(\mu,\D)\le \frac1{|F|}H(\mu,{\rm V}^F)$. Because our tiling $\T$ has \tl\ entropy zero, by the choice of large enough $F=F_n$, this term can be made smaller than $\frac\delta2$. This ends the proof of the inequality \eqref{wte}.

\medskip
For the other inequality, \eqref{wtamte}, we will use Shearer's inequality, which is weaker than strong subadditivity and thus holds for unconditional entropy (see e.g. \cite{DFR}).

Having fixed $\delta$ and $n$ we will abbreviate $F_n$ as $F$. Let $k$ be so large that every shape $S$ of $\T_k$ is $(F,\gamma)$-\inv\, where $\gamma=\frac\delta{|F|\log(|\P|)}$. From now on we will skip the index $k$ in objects associated to the tiling $\T_k$. In the definition of $H_{\T}(\mu,\P)=\sum_{S\in\CS}\mu([S])H(\mu_{[S]},\P^S)$ we will estimate the term $H(\mu_{[S]},\P^S)$. First, we replace it by $H(\mu_{[S]},\P^{\tilde S})$, where $\tilde S$ is the $F^{-1}$-core of $S$ (the set of points $g\in S$ such that $F^{-1}g\subset S$). It follows from the property (2) above Definition \ref{tyi} that $\tilde S$ is a $(1-|F|\gamma)$-subset of $S$. Thus
$$
H(\mu_{[S]},\P^S)\le H(\mu_{[S]},\P^{\tilde S})+|S||F|\gamma\log|\P|=H(\mu_{[S]},\P^{\tilde S})+|S|\delta.
$$
It remains to estimate $H(\mu_{[S]},\P^{\tilde S})$. Consider the family $\{Fs:s\in S\}$.
Every element of $\tilde S$ is contained in precisely $|F|$ sets from this family (it belongs to all $Fs$ with $s\in F^{-1}g)$. That is to say, the above family is \emph{an $|F|$-cover} of $\tilde S$, and the Shearer's inequality applies, yielding
$$
H(\mu_{[S]},\P^{\tilde S})\le\frac1{|F|}\sum_{s\in S}H(\mu_{[S]},\P^{Fs})=\frac1{|F|}\sum_{s\in S}H(\mu_{[S,s]},\P^F),
$$
by invariance of $\mu$. So,
$$
H_{\T}(\mu,\P)\le\frac1{|F|}\sum_{S\in\CS}\sum_{s\in S}\mu([S,s])H(\mu_{[S,s]},\P^F)+\delta
$$
(to obtain $\delta$ at the end we have used $|S|\sum_{S\in\CS}\mu([S])=1$). Because the family $\{[S,s]:S\in\CS,s\in S\}$ is the partition $\D_{\CS}$ of $\bar X$, we have obtained
$$
H_{\T}(\mu,\P)\le\frac1{|F|}H(\mu,\P^F|\D_{\CS})+\delta\le\frac1{|F|}H(\mu,\P^F)+\delta,
$$
and the proof is finished.
\end{proof}
\subsubsection{The language of rectangles}\label{rl}
In this section we introduce the key objects in the construction of symbolic extensions, the rectangles. Although for actions of general countable amenable group these objects no longer resemble rectangles (more appropriate would be calling them ``stacks''), still, by analogy to $\Z$-actions, we will use the term ``$k$-rectangles''.

Let $\bar X$ be the tiled array representation of some zero-dimensional action of $G$, which has the tiling system $\mathbf T$ as a \tl\ factor. We continue to use the notation from the preceding subsection.

\begin{definition} Given $k\in\N$, by a \emph{$k$-rectangle} (\emph{extended $k$-rectangle}) we will mean any block $R\in\bar\Lambda_{[1,k]}^S$ (resp. $\hat R\in(\bar\Lambda_{[1,k]}\times\Lambda_{k+1})^S$), where $S\in\CS_k$, which occurs in some $x\in[S]\subset\bar X$. In particular, $R$ (resp. $\hat R$) has the symbol $``S\,"$ at the position $e$ of the $k$th layer. In either case, $S$ will be referred to as the \emph{shape} of the $k$-rectangle $R$ (resp. extended $k$-rectangle $\hat R$) and $R$ (resp. $\hat R$) will be called a $k$-rectangle (resp. extended $k$-rectangle) over $S$. By $|R|$ (resp. $|\hat R|$) we will always mean the size $|S|$ of the shape. The collection of all $k$-rectangles (resp. extended $k$-rectangles) will be denoted by $\R_k$ (resp. $\hat\R_k$). We will also denote
\begin{gather*}
\R_S = \{R\in\R_k:\text{ the shape of $R$ is }S\} \ \ (S\in\CS_k),\\
\hat\R_S = \{\hat R\in\hat\R_k:\text{ the shape of $\hat R$ is }S\} \ \ (S\in\CS_k),\\
\R=\bigcup_{k\in\N}\R_k\text{ \ \ and \ \ }\hat\R=\bigcup_{k\in\N}\hat\R_k.\phantom{ \ \ (S\in\CS_k)}
\end{gather*}
\end{definition}

By congruency and determinism of the \sq\ of tilings, any $(k\!+\!1)$-rectangle $R'$ is a concatenation of several (precisely $|C_k(S')|$, where $S'\in\CS_{k+1}$ is the shape of $R'$) shifted extended $k$-rectangles, and the projection of $R'$ on the first $k$ layers (denoted by $R'_{[1,k]}$) is a concatenation of $|C_k(S')|$ shifted $k$-rectangles. Although, the component $k$-rectangles (extended $k$-rectangles) are, in the general case, not linearly ordered, we will write these concatenations (also ignoring the shifting of the components) as
$$
R'=\hat R^{(1)}\hat R^{(2)}\dots\hat R^{(q)}, \ \ \ R'_{[1,k]}=R^{(1)}R^{(2)}\dots R^{(q)}
$$
($\hat R^{(i)}\in\hat\R_k,\ R^{(i)}\in\R_k$, $i=1,2,\dots,q$, $q=|C_k(S')|$, $S'\in\CS_{k+1}$ is the shape of $R'$). This will not lead to a confusion, as long as we are only interested in quantitative parameters of the concatenation. (Formally, in writing $R'=\hat R^{(1)}\hat R^{(2)}\dots\hat R^{(q)}$ we make one more imprecision: the concatenation on the right is missing the symbol $``S'"$ at the position $e$ and zeros at other positions of the $(k\!+\!1)$st layer. This should cause no confusion.)
\medskip

With each $k$-rectangle $R\in\R_k$ (extended $k$-rectangle $\hat R\in\hat \R_k$) we will associate its \emph{cylinder set}
$$
[R]=\{x\in\bar X: \bar x\in [S], \bar x_{[1,k]}|_S=R\},\ \ [\hat R]=\{x\in\bar X: \bar x\in [S], \bar x_{[1,k+1]}|_S=\hat R\},
$$
where $S\in\CS_k$ is the shape of $R$ (and of $\hat R$), and $\bar x_{[1,k]}$ (resp. $\bar x_{[1,k+1]}$) is the projection of $\bar x$ on the first $k$ (resp. $k\!+\!1$) layers. For any $R\in\R_k$ we have
\begin{equation}\label{cyl}
[R] = \bigcup\bigl\{[\hat R]:\hat R\in\hat\R_k,\ \hat R_{[1,k]}=R\bigr\},
\end{equation}
where $\hat R_{[1,k]}$ is the $k$-rectangle obtained by projecting $\hat R$ on the first $k$ layers. For a fixed $S\in\CS_k$, with a slight abuse of notation (by identifying the $k$-rectangles or extended $k$-rectangles with their cylinders), we can view $\hat\R_S$ and $\R_S$ as partitions of $[S]$, and then $\hat\R_S\succcurlyeq\R_S$.
\medskip

The language of rectangles will play a crucial role in the forthcoming considerations. In particular, a special case of conditional tiled entropy can be conveniently expressed using rectangles. Let $\mu$ be a probability measure on $\bar X$ and fix some $k\in\N$. Consider the $k$th conditional tiled entropy
$$
H_{\T_k}(\mu,\Lambda_{k+1}|\Lambda_{[1,k]})=\sum_{S\in\CS_k}\mu([S])H(\mu_{[S]},\Lambda_{k+1}^S|\Lambda_{[1,k]}^S)
$$
(where $\Lambda_{k+1}$ and $\Lambda_{[1,k]}$ are considered as symbol partitions of $\bar X$). For each $S\in\CS_k$, all points $\bar x\in[S]$ have the symbol $``S\,"$ in row $k$ at the position $e$. By determinism of the tiling system, this determines all other symbols from the alphabets ${\rm V}_l$ with $l\le k$ at all positions within $S$. In other words, $[S]$ is contained in one atom of the partition ${\rm V}^S_{[1,k]}$. This implies that on $[S]$, the partitions $\Lambda_{[1,k]}^S$ and $\bar\Lambda_{[1,k]}^S$ are identical. Furthermore, the latter partition coincides with $\R_S$ (which is the same as $\R_k$ restricted to $[S]$). Likewise, the partition $\Lambda^S_{[1,k+1]}$ coincides on $[S]$ with $\hat\R_S$, which is the same as $\hat\R_k$ restricted to $[S]$. We conclude that
$$
H_{\T_k}(\mu,\Lambda_{k+1}|\Lambda_{[1,k]})= \sum_{S\in\CS_k}\mu([S])H(\mu_{[S]},\hat\R_k|\R_k).
$$
This looks very much like a conditional entropy, however, $\sum_{S\in\CS_k}\mu([S])$ does not equal $1$. It equals $\mu([C_k])$, where $[C_k]$ is the set of points $\bar x$ which have a tile of $\CT_k$ centered at $e$ (or $\bar x_e=``S\,"$ for some $S\in\CS_k$). If $\mu_{[C_k]}$ denotes the normalized measure $\mu$ restricted to $[C_k]$ then $\mu_{[C_k]}([S])=\frac{\mu([S])}{\mu([C_k])}$. Moreover, since $\mu_{[S]}$ is already normalized, there is no difference between $\mu_{[S]}$ and $(\mu_{[C_k]})_{[S]}$.
Applying this normalization, we obtain
\begin{multline*}
H_{\T_k}(\mu,\Lambda_{k+1}|\Lambda_{[1,k]})=\\
\mu([C_k])\sum_{S\in\CS_k}\mu_{[C_k]}([S])H\bigl((\mu_{[C_k]})_{[S]},\hat\R_k|\R_k\bigr)=\mu([C_k])H(\mu_{[C_k]},\hat\R_k|\R_k\vee\CS_k),
\end{multline*}
where $\CS_k, \R_k$ and $\hat\R_k$ are viewed as partitions of $[C_k]$ (indeed, the set $[C_k]$ consists of all points which have the central tile of $\T_k$ centered at $e$, and the above three partitions classify such points according to the shape of the central tile, the $k$-rectangle and the extended $k$-rectangle over that tile, respectively). But notice that $\R_k\succcurlyeq\CS_k$, because each $k$-rectangle $R$ carries the information about its shape (indeed, the symbol in the $k$th layer of $R$ at the position $e$ is $``S\,"$, which encodes the shape $S$ of $R$). So, the conditioning with respect to $\CS_k$ can be skipped and we have just proved the following, very useful formula:
\begin{equation}\label{przek}
H_{\T_k}(\mu,\Lambda_{k+1}|\Lambda_{[1,k]})=\mu([C_k])H(\mu_{[C_k]},\hat\R_k|\R_k).
\end{equation}
\smallskip

Next, with each $k$-rectangle we will associate a (usually not invariant) empirical measure:

\begin{definition}
Let $k\in\N$. For each $k$-rectangle $R\in\R_k$ we select one point $\bar x_{\!R}$ belonging to the cylinder $[R]$, and we define the \emph{empirical measure associated with~$R$}, as follows:
$$
\boldsymbol\upmu^R= \frac1{|R|}\sum_{g\in S}\delta_{g(\bar x_{\!R})},
$$
where $S\in\CS_k$ is the shape of $R$.
\end{definition}
Although the definition depends on the choice of the point $\bar x_{\!R}$, this choice will turn out to be of no importance. This is why we skip $\bar x_{\!R}$ in the denotation of~$\boldsymbol\upmu^R$.

Recall, that one of the key properties of a F\o lner system of tilings is that the shapes form a F\o lner \sq, which implies that the measures $\boldsymbol\upmu^R$ have the general form $\boldsymbol\upmu^{F_n}_x$ as defined prior to Proposition \ref{oh}, and by that proposition, for sufficiently large $k$, lie in a small neighborhood of $\M_G(\bar X)$. The lemma below shows that if the shapes of the $k$-rectangles are large enough then the measures $\boldsymbol\upmu^R$ (more pecisely, their projections $\boldsymbol\upmu^R_{[1,k]}$ on $\bar X_{[1,k]}$) depend insignificantly on the choice of the points $\bar x_{\!R}$.

\begin{lemma}\label{comb} Choose some $\delta>0$. Let $R\in\R_k$, where $k\in\N$. If the shape $S$ of $R$ is a sufficiently far member of a F\o lner \sq\ then for any two points $\bar x, \bar x'\in[R]$, the empirical measures
$$
\boldsymbol\upmu^R= \frac1{|R|}\sum_{s\in S}\delta_{s(\bar x)} \text{ \ \ and \ \ }
{\boldsymbol\upmu'}^R= \frac1{|R|}\sum_{s\in S}\delta_{s(\bar x')}
$$
satisfy $d_*({\boldsymbol\upmu}^R_{[1,k]},{\boldsymbol\upmu'}^R_{[1,k]})<\delta$.
\end{lemma}

\begin{proof} First of all, notice that $\boldsymbol\upmu^R_{[1,k]}=\frac1{|R|}\sum_{g\in S}\delta_{g(\bar x_{[1,k]})}$. Let $K\subset G$, be a finite set such that if $\bar x_{[1,k]}, \bar x'_{[1,k]}$ agree on $K$ then the corresponding Dirac measures $\delta_{\bar x_{[1,k]}}, \delta_{\bar x'_{[1,k]}}$ are closer to each other than $\frac\delta2$ in $\M(\bar X_{[1,k]})$. If the shape $S$ is a sufficiently far member of a F\o lner \sq\ then it is $(K,\frac\delta{2|K|})$-\inv. We can write
$$
\boldsymbol\upmu^R_{[1,k]}= \frac1{|R|}\sum_{s\in S_K}\delta_{s(\bar x_{[1,k]})}+\frac1{|R|}\sum_{s\in S\setminus S_K}\delta_{s(\bar x_{[1,k]})},
$$
where $S_K$ is the $K$-core of $S$. The sum representing ${\boldsymbol\upmu'}^R_{[1,k]}$ splits analogously. The points $\bar x$ and $\bar x'$ belong to the same cylinder $[R]$, which means that $\bar x_{[1,k]}$ and $\bar x'_{[1,k]}$ agree on $S$, thus, for $s\in S_K$, the points $s(\bar x_{[1,k]})$ and $s(\bar x'_{[1,k]})$ agree at least on $K$ implying that the measures $\delta_{s(\bar x_{[1,k]})}$ and $\delta_{s(\bar x'_{[1,k]})}$ are at most $\frac\delta2$ apart. Because $S_K$ is a $(1\!-\!\frac\delta2)$-subset of $S$ (see property (2) in subsection \ref{2.2}), using convexity of the metric $d_*$ (we also use that $d_*
\le 1$) we get
$$
d_*\bigl(\boldsymbol\upmu^R_{[1,k]}, {\boldsymbol\upmu'}^R_{[1,k]}\bigr)\le \bigl(1-\tfrac\delta2\bigr)\cdot\tfrac\delta2+\tfrac\delta2\cdot 1<\delta,
$$
which ends the proof.
\end{proof}

From the above lemma we draw two conclusions which will be used later. They look very similar, however, the first one deals with \im s, while the other with measures that are not necessarily \inv. This is why the statements are presented separately, with slightly different proofs.

\begin{cor}\label{cor1} Fix some $\delta>0$. If all shapes $S\in\CS_k$ ($k\in\N$) are sufficiently far members of a F\o lner \sq, then, for any \im\ $\mu\in\M_G(\bar X)$, we have
$$
d_*\Bigl(\mu_{[1,k]}\,,\,\sum_{R\in\R_k}\mu([R])|R|\boldsymbol\upmu^R_{[1,k]}\Bigr)<\delta.
$$
\end{cor}

%Let $R'\in\R_{k+1}$ and let $R'_{[1,k]}=R^{(1)}R^{(2)}\dots R^{(q)}$ be the decomposition of the restriction $R'_{[1,k]}$ of $R'$ to the first $k$ layers into $k$-rectangles ($q=|C_k(S')|$, where $S'$ is the shape of $R'$). If all shapes $S\in\CS_k$ are sufficiently far members of a F\o lner \sq, then:
%$$
%d_*\Bigl(\ \mu^{R'}_{[1,k]}\,,\ \frac1{|R'|}\sum_{i=1}^q |R^{(i)}|\,\mu^{R^{(i)}}_{[1,k]}\Bigr)\le\delta.
%$$

\begin{proof} We have:
\begin{multline}\label{al}
\sum_{R\in\R_k}\mu([R])|R|\boldsymbol\upmu^R_{[1,k]}=\sum_{S\in\CS_k}\sum_{R\in\R_S}\mu([R])|R|\frac1{|R|}\sum_{s\in S}(\delta_{s(\bar x_{\!R})})_{[1,k]}=\\
\sum_{S\in\CS_k}\sum_{R\in\R_S}|R|\left(\int_{[R]}\frac1{|R|}\sum_{s\in S}\delta_{s(\bar x_{\!R})}\,d\mu(\bar x)\right)_{[1,k]},
\end{multline}
where $\bar x$ ranges over~$[R]$, while the integrated (measure-valued) function is constant.

We now pass to analyzing $\mu_{[1,k]}$. Given $S\in\CS_k, R\in\R_S$ and $s\in S$, points $\bar x\in s([R])$ are characterized by two properties:
\begin{enumerate}
	\item[(a)] the central tile of the tiling $\CT_k$ associated with $\bar x$ is centered at
	$s^{-1}$, and
 \item[(b)] the $k$-rectangle appearing in $\bar x$ over the central tile is $R$.
\end{enumerate}
It is thus obvious that the sets $s([R])$ with $S$ ranging over $\CS_k$, $R$ ranging over $\R_S$ and $s$ ranging over $S$, form a finite measurable partition of $\bar X$.
Therefore
\begin{multline}\label{ak}
\mu_{[1,k]}=\left(\int_{\bar X}\delta_{\bar x}\,d\mu(\bar x)\right)_{[1,k]}=
\left(\sum_{S\in\CS_k}\sum_{R\in\R_S}\sum_{s\in S}\int_{s([R])}\delta_{\bar x}\,d\mu(\bar x)\right)_{[1,k]}=\\
\sum_{S\in\CS_k}\sum_{R\in\R_S}|R|\left(\int_{[R]}\frac1{|R|}\sum_{s\in S}\delta_{s(\bar x)}\,d\mu(\bar x)\right)_{[1,k]}
\end{multline}
(in the last equality it is essential that $\mu$ is \inv).
Comparing the right hand sides of formulas \eqref{al} and \eqref{ak}, we find out that they differ only in having the variable point $\bar x$ ranging over $[R]$ replaced by the constant point $\bar x_{R}$ (also belonging to $[R]$). Since the shape $S$ of $R$ is a far member of the F\o lner \sq, by Lemma \ref{comb}, for each $\bar x\in[R]$, the measures $(\frac1{|R|}\sum_{s\in S}\delta_{s(\bar x)})_{[1,k]}$ and $(\frac1{|R|}\sum_{s\in S}\delta_{s(\bar x_{\!R})})_{[1,k]}$ are less than $\delta$ apart. The measures $\mu_{[1,k]}$ and $\sum_{R\in\R_k}\mu([R])|R|{\boldsymbol\upmu}^R_{[1,k]}$ are represented as identical integral representations of $(\frac1{|R|}\sum_{s\in S}\delta_{s(\bar x)})_{[1,k]}$ and $(\frac1{|R|}\sum_{s\in S}\delta_{s(\bar x_{\!R})})_{[1,k]}$, with respect to a distribution whose total mass equals
$$
\sum_{S\in \mathcal{S}_k} \sum_{R\in \mathcal{R}_S} |R| \mu ([R])= \sum_{S\in \mathcal{S}_k} |S| \mu ([S])= 1.
$$
So we are dealing with generalized convex combinations, and by convexity and continuity of the metric $d_*$ the proof is finished.
\end{proof}

\begin{cor}\label{cor2}
Fix some $\delta>0$. If all shapes $S\in\CS_k$ ($k\in\N$) are sufficiently far members of a F\o lner \sq, then, the following holds: Let $R'\in\R_{k+1}$ be a $(k\!+\!1)$-rectangle and  let $R'_{[1,k]}=R^{(1)}R^{(2)}\dots R^{(q)}$ be the decomposition of the restriction $R'_{[1,k]}$ of $R'$ to the first $k$ layers into $k$-rectangles ($q=|C_k(S')|$, where $S'$ is the shape of $R'$). Then
$$
d_*\Bigl(\ \boldsymbol\upmu^{R'}_{[1,k]}\,,\ \frac1{|R'|}\sum_{i=1}^q |R^{(i)}|\,\boldsymbol\upmu^{R^{(i)}}_{[1,k]}\Bigr)\le\delta.
$$
\end{cor}

\begin{proof}
In precise terms, the above decomposition of $R'_{[1,k]}$ means that the set $C_k(S')$ can be enumerated as $\{c^{(1)},c^{(2)},\dots,c^{(q)}\}$ and then $S'=\bigcup_{i=1}^q S^{(i)}c^{(i)}$ is the partition of $S'\in\CS_{k+1}$ by the tiles of $\T_k$, and, for each $i=1,2,\dots,q$, $R'_{[1,k]}|_{S^{(i)}c^{(i)}}=R^{(i)}$. With this notation, we have
$$
\boldsymbol\upmu^{R'}= \frac1{|R'|}\sum_{g\in S'}\delta_{g(\bar x_{R'})}=\frac1{|R'|}\sum_{i=1}^q |R^{(i)}|\,\frac1{|R^{(i)}|}\sum_{g\in S^{(i)}}\delta_{gc^{(i)}(\bar x_{R'})}.
$$
On the other hand,
$$
\frac1{|R'|}\sum_{i=1}^q |R^{(i)}|\,\boldsymbol\upmu^{R^{(i)}}=\frac1{|R'|}\sum_{i=1}^q |R^{(i)}|\,\frac1{|R^{(i)}|}\sum_{g\in S^{(i)}}\delta_{g(\bar x_{R^{(i)}})}.
$$
Comparing the right hand sides above we find out that they differ only in having the points $c^{(i)}(\bar x_{R'})$ replaced by $\bar x_{R^{(i)}}$ (selected from the respective cylinders $[R^{(i)}]$, $i=1,2,\dots,q$). But observe that the points $c^{(i)}(\bar x_{R'})$ also belong to the respective cylinders $[R^{(i)}]$. By Lemma \ref{comb}, once all shapes $S\in \CS_k$ are sufficiently far members of a F\o lner sequence, then for each $i=1,2,\dots, q$, the measures
$\frac1{|R^{(i)}|}\sum_{g\in S^{(i)}}\delta_{gc^{(i)}(\bar x_{R'})}$ and $\frac1{|R^{(i)}|}\sum_{g\in S^{(i)}}\delta_{g(\bar x_{R^{(i)}})}$ are less than $\delta$ apart and the assertion follows from convexity of the metric $d_*$ and the fact that
$\sum_{i=1}^q|R^{(i)}|= |R'|$.
\end{proof}

\section{Quasi-symbolic extensions---the hard direction of the main theorem}\label{s5}

In full generality we can prove the hard direction of the Symbolic Extension Entropy Theorem in a somewhat deficient version, where the symbolic extensions are replaced by what we call quasi-symbolic extensions, as defined below:

\begin{definition}
By a \emph{quasi-symbolic system} $\bar Y$ we mean a \tl\ joining $Y\vee\mathbf T$ of a subshift $Y$ with a zero entropy tiling system $\mathbf T$. By a \emph{quasi-symbolic extension} of a system we mean a \tl\ extension which is a quasi-symbolic system.
\end{definition}

We are in a position to prove the hard direction of the main theorem.

\begin{theorem}\label{quasi}
Let a countable amenable group $G$ act on a compact metric space $X$ and let $\H$ denote
the \ens\ of $X$. Then $\EA$ is a (finite) affine \se\ of $\H$ if and only if there exists a quasi-symbolic extension $\pi:\bar Y\to X$ such that $\EA=h^\pi$.
\end{theorem}

\begin{proof}
The ``easy'' direction requires just a comment. We need to show that the extension entropy function $h^{\pi}$ in any quasi-symbolic extension $\pi:\bar Y\to X$, where $\bar Y=Y\vee\mathbf T$, is a \se\ of the \ens\ $\H$ of $X$. We can repeat the proof of Theorem \ref{easy} almost unchanged. The only change is that in the first part (for zero-dimensional $X$) we need to replace $Y$ by $\bar  Y$. But because $\bar Y$ is a principal extension of $Y$, we can still use the partition $\P_\Lambda$ (lifted to $\bar Y$) and the equality $h(\nu,\bar Y)=h(\nu,\P_\Lambda)$ will hold for all $\nu\in\M_G(\bar Y)$. The rest of the proof passes with no further modifications.

It is the other, ``hard'' direction, that requires a lot of work.
We begin by replacing $X$ with its principal zero-dimensional extension $X'$ provided by \cite[Theorem 2]{H}. We choose an array representation $X'=\overset\leftarrow{\lim_k} X'_{[1,k]}$, where, for each $k\in\N$, $X'_k$ is a subshift over some alphabet $\Lambda_k$ and the alphabet of $X'_{[1,k]}$ is $\Lambda_{[1,k]}=\prod_{l=1}^k\Lambda_l$. Next, on $G$ we fix a tiling system $\mathbf T=\overset\leftarrow{\lim_k}\T_k$ of \tl\ entropy zero (whose existence is guaranteed by Theorem \ref{fs}). Later we may need to come back to this starting point and replace $\mathbf T$ by its sub\sq\ (a process which we will call ``speeding up the tiling system''), but at the moment we consider $\mathbf T$ as fixed. We extend $X'$ by joining it (in any case, one can take the direct product) with $\mathbf T$. We denote this joining by $\bar X$. This is still a principal extension of $X$. Now, since $\bar X$ has $\mathbf T$ as a \tl\ factor, we can use its tiled array representation, $\bar X = \overset\leftarrow{\lim_k} \bar X_{[1,k]}$, where, for each $k\in\N$, $\bar X_k$ is a subshift over the alphabet $\bar\Lambda_k=\Lambda_k\times{\rm V}_k$ (recall that ${\rm V}_k=\{``S\,":S\in\CS_k\}\cup\{0\}$), and the alphabet of $\bar X_{[1,k]}$ is $\bar\Lambda_{[1,k]}=\prod_{j=1}^k\bar\Lambda_j$. The \ens\ $\bar\H=(\bar h_k)_{k\in\N}$ of $\bar X$ is given by $\bar h_k(\mu)=h(\mu,\bar\Lambda_{[1,k]})$. Note that since $\mathbf T$ has \tl\ entropy zero, for each $k\in\N$ we have $\bar h_k = h(\cdot,\Lambda_{[1,k]})$. Notice also, that $\bar h_k(\mu)=h(\mu_{[1,k]})$, where $\mu_{[1,k]}$ is the projection of $\mu$ on $\bar X_{[1,k]}$, i.e., $\bar h_k$ is in fact a function on $\M_G(\bar X_{[1,k]})$ (lifted to $\M_G(\bar X)$).

Clearly, $\bar X$ is a principal \zd\ extension of $X$, which, by definition of the \ens\ on $X$, implies that the \ens\ $\H$ lifted to $\M_G(\bar X)$ is uniformly equivalent to $\bar\H$. Thus, the lift of $\EA$ (which clearly is a \se\ of the lift of $\H$), is also a \se\ of $\bar\H$. If we construct a quasi-symbolic extension of $\bar X$ whose extension entropy function equals (the lift of) $\EA$ then the same extension will be a quasi-symbolic extension of $X$ whose extension entropy function equals $\EA$. From now on, we will skip the ``bar'' in the denotation of $\bar\H$ and $\bar h_k$.

The construction mimics the one presented in \cite{BD} for $\Z$-actions (a corrected and slightly simplified version is given in \cite{D1}), but many details have to be reworked. There are three main stages: in the first one, basing on $\EA$ and the entropy structure $\H=(h_k)_{k\in\N}$ we create an \emph{oracle}---an integer-valued function on rectangles. In stage 2, based on the oracle, we build a quasi-symbolic extension $\bar Y$ of $\bar X$. The last part of the proof, stage 3, is the verification that the corresponding extension entropy function indeed matches $\EA$.

\smallskip\noindent
{\bf Stage 1}. For each $k\in\N$ we have $\EA\ge h_k,\  \EA-h_k$ is affine and upper semicontinuous. Thus the function $\EA-h=\lim_k\downarrow(\EA-h_k)$ is also nonnegative, affine and upper semicontinuous. We can use varbatim \cite[Lemma 9.2.6]{D1} and find a decreasing \sq\ of nonnegative affine continuous functions $(g_k)_{k\in\N}$ on $\M_G(\bar X)$, such that, for each $k\in\N$, $g_k$ is constant on fibers of the projection $\M_G(\bar X)\to\M_G(\bar X_{[1,k]})$ (i.e., $g_k(\mu)$ depends only on the projection $\mu_{[1,k]}$) and
\begin{enumerate}
	\item $\lim_k\downarrow g_k = \EA-h$,
	\item $\forall_k\ g_k > \EA-h_k$,
	\item $\forall_k\ g_k-g_{k+1} > h_{k+1}-h_k$.
\end{enumerate}
By continuity of $g_k-g_{k+1}$ and upper semicontinuity of $h_{k+1}-h_k$, we can find a decreasing to zero \sq\ of positive numbers $(\delta_k)_{k\in\N}$ such that, for each $k\in\N$,
$$
g_k-g_{k+1} -3\delta_k > h_{k+1}-h_k.
$$

Our next step is ``speeding up'' the tiling system $\mathbf T$, i.e., replacing it by its sub\sq, in order to guarantee some additional properties (note that since $\mathbf T$ has \tl\ entropy zero, this will not affect any of the preceding arrangements). Five desired properties can be achieved in this manner:
\begin{enumerate}
  \item By speeding up we can arrange that for each $S\in\CS_{k+1}$,
  $|S|>\frac1\delta_{\!\!_k}$. This will imply that for every $S\in\CS_{k+1}$ and $t>0$,
  $$
  \lceil2^{|S|t}\rceil\le 2^{|S|(t+\delta_{\!k})}.
  $$
	\item According to Theorem \ref{tte}, for fixed $k$, the conditional tiled entropy functions on $\M_G(\bar X)$, $H_{\T_{k'}}(\mu,\Lambda_{k+1}|\Lambda_{[1,k]})$ converge, as $k'\to\infty$, decreasingly to $h_{k+1}-h_k$. Since these are continuous functions decreasing to an upper semicontinuous function, and $g_k-g_{k+1}-3\delta_k$ is continuous, for large enough $k'$ we have
$$
g_k-g_{k+1} -3\delta_k > H_{\T_{k'}}(\mu,\Lambda_{k+1}|\Lambda_{[1,k]}).
$$
	Because the partitions $\Lambda_{k+1}$ and $\Lambda_{[1,k]}$ do not depend on the
  tiling system, by speeding up we can arrange that, on $\M_G(\bar X)$,
$$
g_k-g_{k+1} -3\delta_k > H_{\T_k}(\mu,\Lambda_{k+1}|\Lambda_{[1,k]}).
$$
	\item The functions $g_k$, being implicitly defined on $\M_G(\bar X_{[1,k]})$, can be prolonged to continuous and affine functions on $\M(\bar X_{[1,k]})$ (and then lifted to functions on $\M(\bar X)$ constant on fibers of the projection $\pi_{[1,k]}$).  With the help of Lemma \ref{comb} and Corollaries \ref{cor1}, \ref{cor2}, by speeding up the tiling system, we can arrange that for each $k\in\N$ the following conditions hold:
\begin{enumerate}
\item$\bigl|g_k(\boldsymbol\upmu^{R'})-\frac1{|R'|}\sum_{i=1}^q |R^{(i)}|\,g_k(\boldsymbol\upmu^{R^{(i)}})\bigr|<\delta_k$, \item$\bigl|g_k(\mu)-\sum_{R\in\R_k}\,\mu([R])|R|g_k({\boldsymbol\upmu}^R)\bigr|<\delta_k$,
\end{enumerate}
whenever $\mu\in\M_G(\bar X)$ and $R'\in\R_{k+1}$ is such that $R'_{[1,k]}=R^{(1)}R^{(2)}\dots R^{(q)}$, $R^{(i)}\in\R_k$  ($i=1,2,\dots,q$, $q=|C_k(S')|$, $S'$ is the shape of $R'$).
  \item  The inequality $g_k-g_{k+1}-3\delta_k > H_{\T_k}(\mu,\Lambda_{k+1}|\Lambda_{[1,k]})$ then holds on a neighborhood of $\M_G(\bar X)$. According to Proposition \ref{oh}, by speeding up the tiling system, we can arrange that all empirical measures associated to $(k\!+\!1)$-rectangles lie in this neighborhood. Then, for every $(k\!+\!1)$-rectangle $R'$, we shall have
$$
g_k(\boldsymbol\upmu^{R'})-g_{k+1}(\boldsymbol\upmu^{R'}) -3\delta_k > H_{\T_k}(\boldsymbol\upmu^{R'},\Lambda_{k+1}|\Lambda_{[1,k]}).
$$
  \item We also need the following to hold for every $k\in\N$: For each concatenation of
  $k$-blocks, $R^{(1)}R^{(2)}\dots R^{(q)}$, which occurs as the first $k$ layers of some
  $(k\!+\!1)$-rectangle, $S'$ denoting the shape of that $(k\!+\!1)$-rectangle and
  $q=|C_k(S')|$, we should have
  $$
  \sum_{R'\in \R_{k+1},\,R'_{[1,k]}=R^{(1)}R^{(2)}\dots R^{(q)}}
  2^{-|R'|H_{\T_k}(\boldsymbol\upmu^{R'},\,\Lambda_{k+1}|\Lambda_{[1,k]})} < 2^{|S'|\delta_k}.
  $$
  This inequality holds whenever $q$ is sufficiently large; the proof will be
  provided in a moment. So, this property can be achieved by speeding up the tiling system.
\end{enumerate}

Once (after appropriate speeding up of the tiling system) all the above conditions are satisfied, we can define an oracle, $\O:\R\to\N$, as follows: for $R\in\R_k$ we let
$$
\O(R)=\lceil 2^{|R|g_k({\boldsymbol\upmu}^R)}\rceil.
$$
Let us explain, that the oracle ``predicts'' how many different blocks (of the same shape as $R$) will appear  in the elements of the future symbolic extension of $\bar X$ ``above'' the occurrences of $R$. According to the definition (see \cite[Definition 9.2.4]{D1}), in order to be an oracle, a function $\O:\R\to\N$ must satisfy, for every $k\in\N$, and every concatenation $R^{(1)}R^{(2)}\dots R^{(q)}$, where $R^{(i)}\in\R_k$ ($i=1,2,\dots,q$), which occurs as the first $k$ layers of some $(k\!+\!1)$-rectangle, the following \emph{oracle condition}:
$$
\sum_{R'\in\R_{k+1},\,R'_{[1,k]}=R^{(1)}R^{(2)}\dots R^{(q)}}\!\!\!\!\!\!\!\!\!\!\!\!\!\!\!\!\!\!\O(R')\ \ \le \ \
\O(R^{(1)})\O(R^{(2)})\cdots\O(R^{(q)}).
$$

\begin{lemma}
The function $\O(R)$ defined above satisfies the oracle condition.
\end{lemma}
\begin{proof}
Fix a concatenation of $k$-blocks, $R^{(1)}R^{(2)}\dots R^{(q)}$, which occurs as the first $k$ layers of some $(k\!+\!1)$-rectangle. If $S'$ denotes the common shape of all such $(k\!+\!1)$-rectangles then $q=|C_k(S')|$. We have:
\begin{multline*}
\sum_{R'}\O(R') =
\sum_{R'}\lceil 2^{|R'|g_{k+1}(\boldsymbol\upmu^{R'})}\rceil\overset{(1)}\le
\sum_{R'}2^{|R'|(g_{k+1}(\boldsymbol\upmu^{R'})+\delta_k)}\overset{(4)}\le\\
\sum_{R'}2^{|R'|(g_k(\boldsymbol\upmu^{R'})- H_{\T_k}(\boldsymbol\upmu^{R'}\!\!,\,\Lambda_{k+1}|\Lambda_{[1,k]})-2\delta_k)}\overset{(3a)}\le\\
2^{|S'|\frac1{|S'|}\sum_{i=1}^q|R^{(i)}|g_k(\boldsymbol\upmu^{R^{(i)}})}\cdot2^{-|S'|\delta_k}\cdot
\sum_{R'}2^{-|R'|H_{\T_k}(\boldsymbol\upmu^{R'}\!\!,\,\Lambda_{k+1}|\Lambda_{[1,k]})},
\end{multline*}
where in each sum $R'$ ranges as in the oracle condition.
In the last line, we see three expressions separated by the multiplication dots.
The first expression equals $\prod_{i=1}^q2^{|R^{(i)}|g_k(\boldsymbol\upmu^{R^{(i)}})}$, which, after rounding up the multipliers, equals precisely $\O(R^{(1)})\O(R^{(2)})\cdots\O(R^{(q)})$ (so is not larger than this product). The last expression (the sum), by (5), does not exceed $2^{|S'|\delta_k}$ which cancels with the central expression $2^{-|S'|\delta_k}$, and the oracle condition is proved.
\end{proof}

We return to the missing proof of the property (5).

\begin{proof}[Proof of (5)] This is almost literally \cite[Lemma 9.2.11]{D1}, which says that whenever $\Lambda=\Lambda_1\times\Lambda_2$ is a product alphabet, then for every $n\in\N$ and $\varepsilon>0$ there exists an $m_{(n,\varepsilon)}\in\N$ such that for every $q\ge m_{(n,\varepsilon)}$ and every $D\in\Lambda_1^q$ the following holds
$$
\sum_{B\in\Lambda^q,\,B_1=D}2^{-qH_n(B|B_1)}\le 2^{q\epsilon},
$$
where $B_1$ denotes the block appearing in the first row of $B$. In this formulation (which is meant for the $\Z$-action of the classical shift), $H_n(B|B_1)$ stands for
$\frac1nH(\mu_B,\Lambda^n|\Lambda_1^n)$, with $\mu_B$ denoting the \im\ supported by the orbit of the \sq\ obtained as the infinite concatenation\, $\dots BBB\dots$\,. We will use this lemma only in case $n=1$, in which only the values of $\mu_B$ on single symbols play a role. These values are simply the frequencies of the symbols in $B$, so the ``spacial'' form of the block $B$ (i.e., whether it is a linear block over $\{1,2,\dots,q\}$ or a block over a differently looking subset of cardinality $q$ of some other group) has no meaning.

For given $R'\in\R_{k+1}$, applying \eqref{przek}, we get
$$
|R'|H_{\T_k}(\boldsymbol\upmu^{R'},\Lambda_{k+1}|\Lambda_{[1,k]})=|R'|\boldsymbol\upmu^{R'}([C_k])H(\boldsymbol\upmu^{R'}_{[C_k]},\hat\R_k|\R_k).
$$
Clearly, $\boldsymbol\upmu^{R'}([C_k])=\frac{|C_k(S')|}{|R'|}=\frac q{|R'|}$, so the expression on the right hand side simplifies as $qH(\boldsymbol\upmu^{R'}_{[C_k]},\hat\R_k|\R_k)$. Since the measure $\boldsymbol\upmu^{R'}_{[C_k]}$ is applied only to the finite collection of extended $k$-rectangles ($k$-rectangles, as cylinders, are unions of extended $k$-rectangles) on which it is normalized, it can be thought of as a measure $\mu_B$ on single symbols, where $B=\hat R^{(1)}\hat R^{(2)}\dots \hat R^{(q)}$ is the imaginary linearly ordered block over the alphabet $\hat\R_k$ viewed as a subset of the product $\R_k\times\B_k$, where $\B_k$ is the family of one layer blocks appearing in the $(k\!+\!1)$st layer of the extended $k$-rectangles. Taking for $D$ the block $R^{(1)}R^{(2)}\dots R^{(q)}$, the family of blocks $B$ with $B_1=D$ becomes the family of all $(k\!+\!1)$-rectangles $R'$ with $R'_{[1,k]}=R^{(1)}R^{(2)}\dots R^{(q)}$. With such an identification, the term $qH(\boldsymbol\upmu^{R'}_{[C_k]},\hat\R_k|\R_k)$ coincides with $qH_1(B|B_1)$ in the notation of \cite[Lemma 9.2.11]{D1}. The lemma now yields that if $q=|C_k(S')|$ is sufficiently large then
\begin{multline*}
\sum_{R'\in\R_{k+1},\,R'_{[1,k]}=R^{(1)}R^{(2)}\dots R^{(q)}}\!\!\!\!\!\!\!\!\!\!\!\!\!\!\!\!\!\!2^{-|R'|H_{\T_k}(\boldsymbol\upmu^{R'},\,\Lambda_{k+1}|\Lambda_{[1,k]})}=\\
\sum_{R'\in\R_{k+1},\,R'_{[1,k]}=R^{(1)}R^{(2)}\dots R^{(q)}}\!\!\!\!\!\!\!\!\!\!\!\!\!\!\!\!\!\!2^{-qH(\boldsymbol\upmu^{R'}_{[C_k]},\hat\R_k|\R_k)}=
\sum_{B\in\hat\R_k^q,\,B_1=D}2^{-qH_1(B|B_1)}\le 2^{q\delta_k}.
\end{multline*}
Obviously, $2^{q\delta_k}\le 2^{|S'|\delta_k}$, and so (5) is proved.
\end{proof}

\smallskip\noindent
{\bf Stage 2}. Given an oracle $\O$ we will build a quasi-symbolic extension $\bar\pi:\bar Y\to\bar X$, where $\bar Y=Y\vee\mathbf T$ and $Y$ is a subshift.
The mapping $\bar\pi$ will preserve the tiling system (which is a \tl\ factor of both $\bar Y$ and $\bar X$), i.e., if $\bar y\in\bar Y$ and $\bar x = \bar\pi(\bar y)$ then $\bar y$ and $\bar x$ have the same \sq\ of tilings $\CT=(\CT_k)_{k\in\N}$ associated to them. The space $\bar Y$ will be obtained as the intersection of spaces $\bar Y_k\subset Y_k\vee\mathbf T$, each factoring via a map $\bar\pi_k$ onto $\bar X_{[1,k]}$. These factor maps will be consistent, i.e., $\bar\pi_k|_{\bar Y_{k+1}}$ will coincide with $\bar\pi_{k+1}$ composed with the natural projection $\pi_{[1,k]}:\bar X_{[1,k+1]}\to\bar X_{[1,k]}$. Then on the intersection $\bar Y=\bigcap_{k\in\N}\bar Y_k$ we will have $\bar\pi(\bar y)$ defined by specifying all its projections: $(\bar\pi(\bar y))_{[1,k]} = \bar\pi_k(\bar y)$. It is elementary to see that this map will be a \tl\ factor map from $\bar Y$ onto~$\bar X$.

\medskip\noindent
\emph{Step 1}. We begin the construction of $\bar Y$ and of the map $\bar\pi$ by establishing the alphabet $\Lambda$ of the symbolic part $Y_1$ of $\bar Y_1$. This alphabet will remain unchanged in the following steps, as each $\bar Y_k$ will be a subsystem of $\bar Y_1$. Of course, only cardinality of $\Lambda$ matters, and we define it to be the smallest integer $l$ such that, for every $S\in\CS_1$,
$$
l^{|S|}\ge \sum_{R\in\R_S}\O(R).
$$
With such a choice of $\Lambda$, for every $S\in\CS_1$ there exists a map assigning to each $R\in \R_S$ a subfamily $\F_S(R)\subset\Lambda^S$ of cardinality $\O(R)$, in such a way that these families are disjoint for different $1$-rectangles~$R\in\R_S$. Now, for each $\bar x_1\in\bar X_1$ and $\CT_1$ denoting the (first) tiling associated with $\bar x_1$, we will create a closed subset $\bar Y_1(\bar x_1)\subset \Lambda^G\times\mathbf T$ which will constitute the preimage of $\bar x_1$ by the map $\bar\pi_1$ (which we are about to define). Namely, we admit $(y,\CT)$ to belong to $\bar Y_1(\bar x_1)$ if and only if the first tiling in $\CT$ equals $\CT_1$ and, for any tile $Sc$ of $\CT_1$, $y|_{Sc}\in\F_S(\bar x_1|_{Sc})$ (note that $\bar x_1|_{Sc}$ is a $1$-rectangle $R\in\R_S$). It is easy to see that the subsets $\bar Y_1(\bar x_1)$ are disjoint for different elements $\bar x_1\in\bar X_1$ (if $\bar x_1$ and $\bar x'_1$ differ in having different first tilings, say $\CT_1\neq\CT'_1$, then this difference passes to any elements $\bar y\in\bar Y_1(\bar x_1)$ and $\bar y'\in\bar Y_1(\bar x'_1)$; if the first tilings are the same then the first layers of $\bar x_1$ and $\bar x'_1$ must differ on some tile $Sc$ of the common first tiling and then any elements $\bar y\in\bar Y_1(\bar x_1)$ and $\bar y'\in\bar Y_1(\bar x'_1)$ differ on this tile). We let $\bar Y_1=\bigcup_{\bar x_1\in\bar X_1}\bar Y_1(\bar x_1)$ and we skip checking that this is a closed shift-\inv\ set. The functioning of the mapping $\bar\pi_1:\bar Y_1\to\bar X_1$ is now obvious: for $\bar y\in\bar Y_1$ and $\CT$ being the \sq\ of tilings associated with $\bar y$, we let $\bar\pi_1(\bar y)$ be the unique $\bar x_1\in\bar X_1$ whose first tiling  $\CT_1$ is the same as the first tiling of $\CT$, and $\bar y\in \bar Y_1(\bar x_1)$. It is fairly easy to see that this map is a block code with coding horizon $\bigcup\CS_1(\bigcup\CS_1)^{-1}$.

\medskip\noindent
\emph{Step k+1}. Suppose that for some $k\ge 1$ we have defined $\bar Y_k$ and a \tl\ factor map $\bar\pi_k:\bar Y_k\to\bar X_{[1,k]}$ (a block code with coding horizon $\bigcup\CS_k(\bigcup\CS_k)^{-1}$) such that for each $S\in\CS_k$, with each $k$-rectangle $R\in\R_S$ we have associated a family $\F_S(R)\subset\Lambda^S$ of cardinality $\O(R)$ in such a way that
these families are disjoint for different $k$-rectangles~$R\in\R_S$ and the preimage of each $\bar x_{[1,k]}\in\bar X_{[1,k]}$ consists of all such elements $\bar y=(y,\CT)\in\Lambda^G\times\mathbf T$ that the $k$th tilings $\CT_k$ associated to $\bar y$ and to $\bar x_{[1,k]}$ coincide, and, for every tile $Sc$ of $\CT_k$ ($S\in\CS_k$), $y|_{Sc}\in\F_S(\bar x_{[1,k]}|_{Sc})$.

We need to define $\bar Y_{k+1}\subset\bar Y_k$ and the map $\bar\pi_{k+1}:\bar Y_{k+1}\to\bar X_{[1,k+1]}$ which, composed with the natural projection of $\bar \pi_{[1,k]}:\bar X_{[1,k+1]}\to\bar X_{[1,k]}$ coincides with $\bar\pi_k|_{\bar Y_{k+1}}$. Here is how we proceed:
Consider a shape $S'\in\CS_{k+1}$ and a concatenation $D = R^{(1)}R^{(2)}\dots R^{(q)}$ of $k$-rectangles which occurs as the first $k$ layers in some $(k\!+\!1)$-rectangle $R'\in\R_{S'}$ (i.e., $D=R'_{[1,k]}$). For each $\bar x_{[1,k+1]}\in\bar X_{[1,k+1]}$ and $c\in C_{S'}(\CT_{k+1})$ (where $\CT_{k+1}$ appears in the $(k\!+\!1)$st layer of $\bar x_{[1,k+1]}$), such that the projection $\bar x_{[1,k]}$ of $\bar x_{[1,k+1]}$ satisfies $\bar x_{[1,k]}|_{S'c} = D$, and any $\bar y=(y,\CT)\in\bar\pi_k^{-1}(\bar x_{[1,k]})$ we have (in spite of the common tiling $\CT_k$ for $\bar x_{[1,k]}$ and $\CT$) the following: if $S^{(i)}c^{(i)}$ denotes the tile of $\CT_k$ contained in $S'c$ on which $\bar x_{[1,k]}$ equals $R^{(i)}$ then
$y|_{S^{(i)}c^{(i)}}\in\F_{S^{(i)}}(R^{(i)})$ ($i=1,2,\dots,q$). Moreover, all restrictions $y|_{S'c}$ which fulfill the above for each $i=1,2,\dots,q$ are present in the preimage by $\bar\pi_k$ of $\bar x_{[1,k]}$. This means that there is a family $\mathcal E_D$ consisting of exactly $\O(R^{(1)})\O(R^{(2)})\dots\O(R^{(q)})$ blocks belonging to $\Lambda^{S'}$ appearing in the elements of $\bar\pi_k^{-1}(\bar x_{[1,k]})$ ``above'' each occurrence of any $(k\!+\!1)$-rectangle $R'$ such that $R'_{[1,k]}=D$ in any $\bar x_{[1,k+1]}\in\bar X_{[1,k+1]}$. Note that the families $\mathcal E_D$ are disjoint for different concatenations $D$ with a common shape $S'$. By the oracle condition, with each $(k\!+\!1)$-rectangle $R'$ satisfying $R'_{[1,k]}=D$ we can associate a subfamily $\F_D(R')\subset\mathcal E_D$ of cardinality $\O(R')$ so that these families are disjoint for different $(k\!+\!1)$-rectangles $R'$ with $R'_{[1,k]}=D$. By disjointness of the families $\mathcal E_D$ (for different $D$ with a common shape $S'$), there will be no confusion if denote $\F_D(R')$ by $\F_{S'}(R')$. For any $\bar x_{[1,k+1]}\in\bar X_{[1,k+1]}$ we now define the set $\bar Y_{k+1}(\bar x_{[1,k+1]}) \subset\bar Y_k(\bar x_{[1,k]})$ (the preimage of $\bar x_{[1,k+1]}$ by the future map $\bar\pi_{k+1}$) by the already familiar rule: $\bar Y_{k+1}(\bar x_{[1,k+1]})$ consists of all such elements $\bar y=(y,\CT)\in\bar Y_k(\bar x_{[1,k]})$ that the $(k\!+\!1)$st tilings $\CT_{k+1}$ associated to $\bar y$ and to $\bar x_{[1,k+1]}$ coincide, and, for every tile $S'c$ of $\CT_{k+1}$, $y|_{S'c}\in\F_{S'}(\bar x_{[1,k+1]}|_{S'c})$. We skip the description of how the map $\bar\pi_{k+1}$ functions; it is fully analogous to the description for $\bar\pi_1$. The coding horizon is now $\bigcup\CS_{k+1}(\bigcup\CS_{k+1})^{-1}$.

\smallskip\noindent
{\bf Stage 3}. Once the induction is completed, we have defined both the quasi-symbolic extension $\bar Y$ of $\bar X$ and the associated factor map $\bar\pi:\bar Y\to\bar X$. What remains to do is to verify that on $\M_G(\bar X)$, $h^{\bar\pi}=\EA$ (or that $h^{\bar\pi}-h=\EA-h$).

\begin{lemma}\label{seven}
Fix an \im\ $\mu\in\M_G(\bar X)$ and let $\mu_{[1,k]}$ denote the projection of $\mu$ onto $\bar X_{[1,k]}$. Then
$$
h^{\bar\pi}(\mu)-h(\mu)= \lim_k\sup_{\nu\in\bar\pi_k^{-1}(\mu_{[1,k]})} H_{\T_k}(\nu,\Lambda|\bar\Lambda_{[1,k]}).
$$
\end{lemma}
\begin{proof}
First observe that for $\nu\in\bar\pi^{-1}(\mu)$, the expressions $H_{\T_k}(\nu,\Lambda|\bar\Lambda_{[1,l]})$ are nonincreasing in both $k$ and $l$, hence both iterated limits and the diagonal limit coincide. By Theorem \ref{tte}, the limit in $k$ (with $l$ fixed) equals $h(\nu,\Lambda|\bar\Lambda_{[1,l]}) = h(\nu)-h(\mu_{[1,l]})$, which converges in $l$ to $h(\nu)-h(\mu)$.

Because $\bar\pi^{-1}(\mu)\subset \bar\pi_k^{-1}(\mu_{[1,k]})$ and by the ``rule of thumb'' ``$\lim_a\sup_b\ge\sup_b\lim_a$'', we have
\begin{multline*}
\lim_k\sup_{\nu\in\bar\pi_k^{-1}(\mu_{[1,k]})} H_{\T_k}(\nu,\Lambda|\bar\Lambda_{[1,k]})\ge
\sup_{\nu\in\bar\pi^{-1}(\mu)} \lim_k H_{\T_k}(\nu,\Lambda|\bar\Lambda_{[1,k]})=\\
\sup_{\nu\in\bar\pi^{-1}(\mu)} h(\nu)-h(\mu) = h^{\bar\pi}(\mu)-h(\mu).
\end{multline*}

On the other hand, if $l$ is fixed then, since eventually $k\ge l$ and hence $\pi_k^{- 1} (\mu_{[1, k]})\subset \pi_l^{- 1} (\mu_{[1, l]})$, we have
$$
\lim_k\sup_{\nu\in\bar\pi_k^{-1}(\mu_{[1,k]})} H_{\T_k}(\nu,\Lambda|\bar\Lambda_{[1,k]})\le
\lim_k\sup_{\nu\in\bar\pi_l^{-1}(\mu_{[1,l]})} H_{\T_k}(\nu,\Lambda|\bar\Lambda_{[1,l]})=\cdots
$$
The functions $H_{\T_k}(\nu,\Lambda|\bar\Lambda_{[1,l]})$ are continuous and decrease in $k$, and we consider the supremum over a fixed compact set. In this situation, the supremum and limit can be exchanged (see e.g. \cite[Fact A.1.24]{D1}), and we can continue
as follows:
\begin{multline*}
\cdots=\sup_{\nu\in\bar\pi_l^{-1}(\mu_{[1,l]})} \lim_k H_{\T_k}(\nu,\Lambda|\bar\Lambda_{[1,l]})=
\sup_{\nu\in\bar\pi_l^{-1}(\mu_{[1,l]})} h(\nu,\Lambda|\bar\Lambda_{[1,l]})=\\
\sup_{\nu\in\bar\pi_l^{-1}(\mu_{[1,l]})} h(\nu)-h(\mu_{[1,l]}).
\end{multline*}
Since the left hand side does not depend on $l$, we can apply the limit in $l$ to the right hand side and the inequality will hold. The function $\nu\mapsto h(\nu)$ is upper semicontinuous (this is true for symbolic systems and $\bar Y$ differs from the symbolic system $Y$ by being joined with a zero-entropy system, which does not alter the entropy function). As easily verified, the sets $\bar\pi_l^{-1}(\mu_{[1,l]})$ decrease in $l$ to $\bar\pi^{-1}(\mu)$. This implies that $\sup_{\nu\in\bar\pi_l^{-1}(\mu_{[1,l]})} h(\nu)$ tends (nonincreasingly with $l$) to $\sup_{\nu\in\bar\pi^{-1}(\mu)} h(\nu)$, while $h(\mu_{[1,l]})$ clearly tends to $h(\mu)$. Thus the right hand side (after applying the limit in $l$) becomes $h^{\bar\pi}(\mu)-h(\mu)$, completing the proof of the lemma.
\end{proof}

The above lemma reduces the problem to finding measures $\nu$ in the preimage $\bar\pi_k^{-1}(\mu_{[1,k]})$ maximizing the conditional tiled entropy $H_{\T_k}(\nu,\Lambda|\bar\Lambda_{[1,k]})$. Recall that

$$
H_{\T_k}(\nu,\Lambda|\bar\Lambda_{[1,k]})=\sum_{S\in\CS_k}\nu([S])H(\nu_{[S]},\Lambda^S|\bar\Lambda_{[1,k]}^S).
$$
Now, on $[S]$ the partition $\bar\Lambda_{[1,k]}^S$ coincides with the partition $\R_S$ into $k$-rectangles with the shape $S$, so the above sum splits further as

\begin{multline*}
\sum_{S\in\CS_k}\nu([S])\sum_{R\in\R_S}\nu_{[S]}([R])H((\nu_{[S]})_{[R]},\Lambda^S)=\\
\sum_{S\in\CS_k}\sum_{R\in\R_S}\nu([R])H(\nu_{[R]},\Lambda^S)=\sum_{S\in\CS_k}\sum_{R\in\R_S}\mu([R])H(\nu_{[R]},\Lambda^S)
\end{multline*}
(we have used $(\nu_{[S]})_{[R]}=\nu_{[R]}$ and because $R$ depends only on the first $k$ layers, we also have used $\nu([R])=\mu_{[1,k]}([R])=\mu([R])$ whenever $\pi_k (\nu)= \mu_{[1, k]}$).

Our task is thus to maximize $H(\nu_{[R]},\Lambda^S)$ for each $S\in\CS_k$ and $R\in\R_S$. By the definition of $\bar\pi_k$, for every $R\in\R_S$, the conditional measure $\nu_{[R]}$ is supported by the family of blocks $\F_S(R)$ and clearly the largest entropy is achieved when all these blocks have equal masses. In fact, this condition defines, for each $\mu$ on $\bar X$ a measure $\nu$ on $\bar Y_k$ belonging to $\bar\pi_k^{-1}(\mu_{[1,k]})$. We will call this measure $\nu_{\max}^{\mu_{[1,k]}}$. Since $|\F_S(R)|=\O(R)$, we have $H((\nu_{\max}^{\mu_{[1,k]}})_{[R]},\Lambda^S)=\log(\O(R))$. So,
$$
\max_{\nu\in\bar\pi_k^{-1}(\mu_{[1,k]})}\!\!\!\!\!\!H_{\T_k}(\nu,\Lambda|\bar\Lambda_{[1,k]}) =\!\!\!\sum_{R\in\R_k}\mu([R])\log(\O(R))=\!\!\!\sum_{R\in\R_k}\mu([R])|R|(g_k({\boldsymbol\upmu}^R)+\xi_R)=\!\cdots
$$
where the error term $\xi_R$ ranges between $0$ and $\delta_k$. Note that the sum of the coefficients $\mu([R])|R|$ over $R\in\R_k$ equals $1$, so what we see above is a convex combination. By the condition (3b), we can continue
$$
\dots<g_k (\mu)+\delta_k+\xi_k (\mu)= g_k (\mu)+ \xi_k' (\mu),
$$
where $0\le \xi_k (\mu)\le \delta_k$ and hence $\xi_k' (\mu)$ lies between $0$ and $2\delta_k$. Combining Lemma \ref{seven} with the properties defining the \sq\ $(g_k)_{k\in\N}$, we obtain
$$
h^{\bar\pi}(\mu)-h(\mu)=\lim_k g_k(\mu) = \EA-h.
$$
This concludes the proof of Theorem \ref{quasi}.
\end{proof}

We end this section by mentioning some obvious consequences. For instance, we obtain a characterization of asymptotic $h$-expansiveness. For our purposes, we define asymptotic $h$-expansiveness by a condition which for $\Z$-actions is known to be equivalent to the original definition by M. Misiurewicz (see \cite{M}, and see \cite[Corollary 8.4.12]{D1} for the equivalence):

\begin{definition}
An action of a countable amenable group $G$ on a compact metric space $X$ is \emph{asymptotically $h$-expansive} if it has finite \tl\ entropy and there exists an \ens\ $\H=(h_k)_{k\in\N}$ which converges to the entropy function $h$ uniformly on $\MGX$.
\end{definition}

Observe that in such case, the constant structure $(h)_{k\in\N}$ is also an \ens, because it is uniformly equivalent to $\H$. Any other \ens\ must be uniformly equivalent to the constant structure, which in turn implies that it converges to $h$ uniformly. We have proved that in an asymptotically $h$-expansive system every \ens\ converges to $h$ uniformly.
\smallskip

In order to escape distractions from studying our main subject (which are symbolic extensions), we refrain from discussing whether the above definition is equivalent to the original definition adapted to the context of actions of countable amenable groups. We refer to \cite{CZ} for more details of the adaptation of Misiurewicz' definition \cite{M} to actions of sofic (including amenable) groups. Our definition reflects the most vital for us features of asymptotic $h$-expansiveness, and allows us to formulate what follows:

\begin{theorem} An action of a countable amenable group $G$ on a compact metric space $X$ is asymptotically $h$-expansive if and only if it admits a principal quasi-symbolic extension.
\end{theorem}
\begin{proof}By Theorem \ref{niezalezy}, we can fix an \ens\ $\H=(h_k)_{k\ge0}$ of $X$ which has \usc\ differences.
If $\bar Y$ is a principal quasi-symbolic extension of $X$ then the extension entropy function equals the entropy function $h$. Theorem \ref{quasi} (the ``easy'' direction) implies that $h$ is a \se\ of the entropy structure $\H$ of $X$. So, by Proposition \ref{fact2} (3), the functions $h-h_k$ are \usc\ and obviously they converge nonincreasingly to $0$. Such convergence is always uniform, proving that $h_k$ tends to $h$ uniformly, and $X$ is asymptotically $h$-expansive.

For the opposite implication assume that $\H$ converges to $h$ uniformly. As a consequence, for each $k\in\N$, $h-h_k$ is a uniform limit of $h_{k'}-h_k$ as $k'\to\infty$. A uniform limit of \usc\ functions is \usc. So, $h-h_k$
is \usc\ (and clearly nonnegative) which implies that $h$ is a (finite) \se\ of $\H$. Also, $h$ is affine on $\MGX$. Now, Theorem \ref{quasi} implies the existence of a quasi-symbolic extension of $X$ with the extension entropy function equal to $h$, which is equivalent to the extension being principal.
\end{proof}

\section{The comparison property}\label{s6}

This section is devoted to isolating and studying a very important ``comparison'' property that an action of a countable amenable group $G$ may or may not have. A version of this property can also be associated with the group $G$ itself. For us, its significance lies in the fact that it enables us to create genuine symbolic extensions in place of quasi-symbolic ones. Ironically, for $\Z$-actions the analogous passage occupies in the proof of the Symbolic Extension Entropy Theorem just one line, and the comparison property is not explicitly invoked (but is implicitly essential).

\subsection{Definition of the comparison property}

We will understand the comparison property as follows (see also \cite{K}):

\begin{definition}\label{defcom}
Let $G$ be a countable amenable group.
\begin{enumerate}
\item Let $G$ act on a \zd\ compact metric space $X$. For two clopen sets $\mathsf A,\mathsf B\subset X$, we say that $\mathsf A$ is \emph{subequivalent} to $\mathsf B$ (and write $\mathsf A\preccurlyeq \mathsf B$), if there exists a finite partition $\mathsf A=\bigcup_{i=1}^k \mathsf A_i$ of $\mathsf A$ into clopen sets and there are elements $g_1,g_2,\dots,g_k$ of $G$ such that $g_1(\mathsf A_1), g_2(\mathsf A_2),\dots,g_k(\mathsf A_k)$ are disjoint subsets of $\mathsf B$. We say that the action \emph{admits comparison} if for any pair of clopen subsets $\mathsf A,\mathsf B$ of $X$, the condition that for each \im\ $\mu$ on $X$ we have $\mu(\mathsf A)<\mu(\mathsf B)$, implies $\mathsf A\preccurlyeq \mathsf B$.
\item If every action of $G$ on any \zd\ compact metric space admits comparison then we will say that $G$ has the \emph{comparison property}.
\end{enumerate}
\end{definition}
Clearly, $\mathsf A\preccurlyeq \mathsf B$ implies $\mu(\mathsf A)\le\mu(\mathsf B)$ for every invariant measure $\mu$, so comparison is ``nearly'' an equivalence between subequivalence and the inequality for all \im s.

%Note the obvious fact, that if $X$ admits comparison and $Z$ is a \zd\ \tl\ factor of $X$ then $Z$ also admits comparison.

\begin{remark}\label{from0}
Let two clopen sets $\mathsf A,\mathsf B$ satisfy $\mu(\mathsf A)<\mu(\mathsf B)$ for every \im\ $\mu$. Because the sets $\mathsf A,\mathsf B$ are clopen, the function $\mu\mapsto\mu(\mathsf B)-\mu(\mathsf A)$ is continuous, and since it is positive on a compact set, it is separated from zero, i.e.,
$$
\inf_{\mu\in\M_G(X)}(\mu(\mathsf B)-\mu(\mathsf A))>0.
$$
\end{remark}

The following definition and the adjacent lemma are not used further in this paper. We provide them for a more complete treatment of the comparison property. Consider the following seemingly weaker property:

\begin{definition}\label{defwcom}
The action of a countable amenable group $G$ on a \zd\ compact metric space $X$ admits \emph{weak comparison} if there exists a constant $C\ge 1$ such that for any clopen sets $\mathsf A,\mathsf B\subset X$, the condition $\sup_\mu\mu(\mathsf A)<\frac1C\inf_\mu\mu(\mathsf B)$ (where $\mu$ ranges over all \im s) implies $\mathsf A\preccurlyeq \mathsf B$.
\end{definition}

Clearly, comparison implies weak comparison. We will show that these properties are in fact equivalent.

\begin{lemma}\label{compisweakcomp}
Weak comparison implies comparison.
\end{lemma}
\begin{proof}
Suppose the action of a countable amenable group $G$ on a \zd\ compact metric space $X$ admits weak comparison with a constant $C$. Let two clopen sets $\mathsf A,\mathsf B$ satisfy $\mu(\mathsf A)<\mu(\mathsf B)$ for every \im\ $\mu$. By
Remark \ref{from0}, $\inf_{\mu\in\M_G(X)}(\mu(\mathsf B)-\mu(\mathsf A))>\eps$ for some positive $\eps$.
We order the group (arbitrarily) by natural numbers, as $G=\{g_1,g_2,\dots\}$ (or $G=\{g_1,\dots,g_n\}$ in case $G$ is finite). We let $\mathsf A_1=\mathsf A\cap g_1^{-1}(\mathsf B)$, and $\mathsf B_1=g_1(\mathsf A_1)$. For each $k>1$ (or $1<k\le n$ in the finite case) we set inductively
$$
\mathsf A_k = \mathsf A\setminus\Bigl(\bigcup_{i=1}^{k-1} \mathsf A_i\Bigr)\cap g_k^{-1}\left(\mathsf B\setminus\Bigl(\bigcup_{i=1}^{k-1} \mathsf B_i\Bigr)\right),
$$
and $\mathsf B_k=g_k(\mathsf A_k)$. It is not hard to see that the sets $\mathsf A_k$ and $\mathsf B_k$ are clopen (some of them possibly empty), disjoint subsets of $\mathsf A$ and $\mathsf B$, respectively and $\mu(\mathsf A_k)=\mu(\mathsf B_k)$ for each $k$ and every \im\ $\mu$. Consider the remainder sets
$$
\mathsf A_0=\mathsf A\setminus\Bigl(\bigcup_{k=1}^{\infty} \mathsf A_k\Bigr)\text{\ \ and \ \ }\mathsf B_0=\mathsf B\setminus\Bigl(\bigcup_{k=1}^{\infty} \mathsf B_k\Bigr),
$$
or in the finite case
$$
\mathsf A_0=\mathsf A\setminus\Bigl(\bigcup_{k=1}^{n} \mathsf A_k\Bigr)\text{\ \ and \ \ }\mathsf B_0=\mathsf B\setminus\Bigl(\bigcup_{k=1}^{n} \mathsf B_k\Bigr).
$$
Clearly, for each \im\ $\mu$ we have $\mu(\mathsf B_0)\ge\eps$. We claim that $\mu(\mathsf A_0)=0$. It suffices to consider an ergodic measure. But if $\mu(\mathsf A_0)$ was positive, then, by ergodicity, there would exist an $x\in \mathsf A_0$ and $g=g_k$ (for some $k$) such that $g_k(x)\in \mathsf B_0$. This is a contradiction, as, by construction, no orbit starting in $\mathsf A_0$ visits the set $\mathsf B_0$. Now, by countable additivity of the measures, we obtain, for each \im\ $\mu$,
$$
\lim_{k\to\infty} \mu\!\left(\mathsf A\setminus\Bigl(\bigcup_{i=1}^k \mathsf A_i\Bigr)\right) =0.
$$
Clearly, the limit is decreasing. Since the measured sets are clopen, the above measure values viewed as functions on the set of \im s are continuous, and thus the convergence is uniform. Let $\delta>0$ be strictly smaller than $\frac\eps C$. Then, for $k$ large enough we have, simultaneously for all \im s $\mu$,
$$
\mu\!\left(\mathsf A\setminus\Bigl(\bigcup_{i=1}^k \mathsf A_i\Bigr)\right)\le\delta<\frac\eps C\le
\frac1C\,\mu\!\left(\mathsf B\setminus\Bigl(\bigcup_{i=1}^k \mathsf B_i\Bigr)\right).
$$
By the weak comparison assumption, we get
$$
\mathsf A\setminus\Bigl(\bigcup_{i=1}^k \mathsf A_i\Bigr)\preccurlyeq \mathsf B\setminus\Bigl(\bigcup_{i=1}^k \mathsf B_i\Bigr),
$$
which, together with the obvious fact that $\bigcup_{i=1}^k \mathsf A_i\preccurlyeq \bigcup_{i=1}^k \mathsf B_i$, completes the proof of $\mathsf A\preccurlyeq \mathsf B$.
\end{proof}

\begin{remark}\label{finitecomp}
The above proof shows also that every finite group $G=\{g_1,g_2,\dots,g_n\}$ has the comparison property. For such a group we have $\mathsf A_0=\mathsf A\setminus(\bigcup_{i=1}^n\mathsf A_i)$. The fact that $\mathsf A_0$ has measure $0$ for all \im s implies that it is empty.
\end{remark}

\begin{remark}\label{disjoint}In the definition of comparison, it suffices to consider only disjoint clopen sets $\mathsf A, \mathsf B$. Indeed, $\{\mathsf A\cap\mathsf B, \mathsf A\setminus\mathsf B\}$ is a clopen partition of $\mathsf A$, and $g_0=e$ sends $\mathsf A\cap\mathsf B$ inside $\mathsf B$, so if $(\mathsf A\setminus\mathsf B)\preccurlyeq(\mathsf B\setminus\mathsf A)$ then also $\mathsf A\preccurlyeq\mathsf B$.
Also note that, for any measure~$\mu$, $\mu(\mathsf A)<\mu(\mathsf B)$ if and only if $\mu(\mathsf A\setminus\mathsf B)<\mu(\mathsf B\setminus\mathsf A)$.
\end{remark}

It is known that many important countable amenable groups, for instance $\Z$, $\Z^d$,
have the comparison property. However, the following question remains open:

\begin{ques}\label{3.7}
Does every countable amenable group have the comparison property?
\end{ques}

Later in this section we will provide a positive answer in a large class of groups.

\subsection{Banach density interpretation of the comparison property}\label{two}

Now we provide a characterization of the comparison property of a countable amenable group in terms of Banach density advantage for subsets of the group.
\smallskip

\subsubsection{Passing between clopen subsets of $X$ and subsets of $G$}
This subsection contains fairly standard tools, often exploited in symbolic dynamics. We include them for completeness and as an opportunity to introduce our notation. We continue to assume that $G$ is a countable amenable group.

\smallskip
\emph{(A) From clopen subsets of $X$ to subsets of $G$}.
First suppose that $G$ acts on a \zd\ compact metric space $X$ in which we have two disjoint clopen sets $\mathsf A$ and $\mathsf B$. Define a map $\pi_{\mathsf A\mathsf B}:X\to\{\mathsf 0,\mathsf 1,\mathsf 2\}^G$ by the formula
$$
(\pi_{\mathsf A\mathsf B}(x))_g=\begin{cases}\mathsf 1\\\mathsf 2\\\mathsf 0\end{cases} \iff g(x)\in \begin{cases} \mathsf A\\\mathsf B\\(\mathsf A\cup \mathsf B)^c,\end{cases}
$$
respectively ($g\in G$). As easily verified, $\pi_{\mathsf A\mathsf B}$ is continuous and intertwines the action on $X$ with the shift action, in other words, it is a \tl\ factor map onto its image $Y_{\mathsf A\mathsf B}=\pi_{\mathsf A\mathsf B}(X)$, which is a subshift,  in which we can distinguish two natural clopen sets, the cylinders $[\mathsf 1]$ and $[\mathsf 2]$. Notice that $\pi_{\mathsf A\mathsf B}^{-1}([\mathsf 1])=\mathsf A$ and $\pi_{\mathsf A\mathsf B}^{-1}([\mathsf 2])=\mathsf B$, hence for every \im\ $\mu$ on $X$ we have $\mu(\mathsf A)=\nu([\mathsf 1])$ and $\mu(\mathsf B)=\nu([\mathsf 2])$, where $\nu=\pi_{\mathsf A\mathsf B}(\mu)$. The set of all shift-\im s on $Y_{\mathsf A\mathsf B}$ will be abbreviated as $\mathcal{M}_{\mathsf A\mathsf B}$. For each $x\in X$ we define two subsets of $G$,
\begin{align}
A_x &= \{g:g(x)\in \mathsf A\}=\{g:(\pi_{\mathsf A\mathsf B}(x))_g=\mathsf 1\}=\{g:g(\pi_{\mathsf A\mathsf B}(x))\in[\mathsf 1]\},\label{kiki}\\
B_x &= \{g:g(x)\in \mathsf B\}=\{g:(\pi_{\mathsf A\mathsf B}(x))_g=\mathsf 2\}=\{g:g(\pi_{\mathsf A\mathsf B}(x))\in[\mathsf 2]\}.\label{kiko}
\end{align}

In the above context we can define new notions:
\begin{definition}\label{dbar}
We fix in $G$ a F\o lner \sq\ $(F_n)_{n\in\N}$. The terms
\begin{align*}
\underline D(\mathsf B)&=\limsup_{n\to\infty}\ \inf_{x\in X}\underline D_{F_n}(B_x),\\
\overline D(\mathsf B)&=\liminf_{n\to\infty}\ \sup_{x\in X}\overline D_{F_n}(B_x),\\
\underline D(\mathsf B,\mathsf A)&=\limsup_{n\to\infty}\ \inf_{x\in X}\underline D_{F_n}(B_x,A_x),
\end{align*}
will be called the \emph{uniform lower Banach density of (visits of the orbits in) $\mathsf B$}, \emph{uniform upper Banach density of $\mathsf B$} and \emph{uniform Banach density advantage of $\mathsf B$ over $\mathsf A$}.
\end{definition}

A statement analogous to Proposition \ref{bd1} holds:

\begin{lemma}\label{bbb}
The values of $\underline D(\mathsf B)$, $\overline D(\mathsf B)$ and $\underline D(\mathsf B,\mathsf A)$ do not depend on the choice of the F\o lner sequence, the limits superior and inferior in the definition are in fact limits, and moreover
\begin{align*}
\underline D(\mathsf B)&=\sup_F\ \inf_{x\in X}\underline D_F(B_x),\\
\overline D(\mathsf B)&=\inf_F\ \sup_{x\in X}\overline D_F(B_x),\\
\underline D(\mathsf B,\mathsf A)&=\sup_F\ \inf_{x\in X}\underline D_F(B_x,A_x),
\end{align*}
where $F$ ranges over all finite subsets of $G$.
\end{lemma}

\begin{proof} The proof is identical as that of Proposition \ref{bd1}, with the only difference that Lemma \ref{bdc} applies to the sets $A_x,B_x$ whenever $F_n$ is $(F,\eps)$-\inv, simultaneously for all $x\in X$.
\end{proof}

In a moment we will connect the above notions with the values assumed by the invariant measures on $X$ on the sets $\mathsf A$ and $\mathsf B$.
\smallskip

\emph{(B) From subsets of $G$ to clopen subsets of $X$}.
We will now describe the opposite passage: from subsets of $G$ to clopen subsets of some \zd\ compact metric space on which we have a $G$-action.  Suppose we have two disjoint subsets $A$ and $B$ of $G$. Then they determine an element $y^{AB}$ of the symbolic space $\{\mathsf 0,\mathsf 1,\mathsf 2\}^G$, given by the rule
$$
y^{AB}_g=\begin{cases}\mathsf 1\\\mathsf 2\\\mathsf 0\end{cases} \iff g\in \begin{cases} A\\B\\(A\cup B)^c,\end{cases}
$$
respectively ($g\in G$).
The shift-orbit closure of $y^{AB}$, i.e., the set
$$
Y^{AB}=\overline{\{g(y^{AB}):g\in G\}}
$$
is a subshift, which we will call the \emph{subshift associated with the sets $A, B$}. The set of its \im s, $\M_G(Y^{AB})$, will be abbreviated as $\M^{AB}$. In this subshift we will distinguish two clopen sets, $\mathsf A=[\mathsf 1]$ and $\mathsf B=[\mathsf 2]$. It is almost immediate to see that if we apply the definitions of the preceding paragraph to the shift action on $Y^{AB}$ and the above sets $\mathsf A,\mathsf B$ then the factor map $\pi_{\mathsf A\mathsf B}$ is the identity, and $A_{y^{AB}}=\{g:y^{AB}_g=\mathsf 1\} = A$ and $B_{y^{AB}}=\{g:y^{AB}_g=\mathsf 2\}=B$.

\begin{prop}\label{prop}
\begin{enumerate}
\item Suppose $G$ acts on a \zd\ compact metric space $X$ in which we are given two disjoint clopen sets, $\mathsf A,\mathsf B$. Then
\begin{align*}
\inf_{\mu\in\M_G(X)}\mu(\mathsf B)&=\underline D(\mathsf B)=\inf_{x\in X}\underline D(B_x), \\
\sup_{\mu\in\M_G(X)}\mu(\mathsf B)&=\overline D(\mathsf B)=\sup_{x\in X}\overline D(B_x),\\
\inf_{\mu\in\M_G(X)}(\mu(\mathsf B)-\mu(\mathsf A))&=\underline D(\mathsf B,\mathsf A)=\inf_{x\in X}\underline D(B_x,A_x).
\end{align*}
\item
Next suppose that $A$ and $B$ are disjoint subsets of $G$. Consider the cylinders $[\mathsf 1]$ and $[\mathsf 2]$ in the subshift $Y^{AB}$ associated with these sets. Then
\begin{gather*}
\inf_{\mu\in\M^{AB}}\mu([2])=\underline D(B),\\
\sup_{\mu\in\M^{AB}}\mu([2])=\overline D(B),\\
\inf_{\mu\in\M^{AB}}(\mu([2])-\mu([1]))=\underline D(B,A).
\end{gather*}
\end{enumerate}
\end{prop}

\begin{proof}
In (1) we will only show the last line of equalities. The first line will then follow by plugging in $\mathsf A=\emptyset$ and the the second one by considering the complement of $\mathsf B$. First suppose that we have sharp inequality $\inf_{\mu\in\M_G(X)}(\mu(\mathsf B)-\mu(\mathsf A))>\underline D(\mathsf B,\mathsf A)$. By Lemma \ref{bbb}, there exists an $\eps>0$ such that for every finite set $F$, $\inf_{\mu\in\M_G(X)} (\mu(\mathsf B)-\mu(\mathsf A))-\eps > \inf_{x\in X} \underline D_F(B_x,A_x)$.
In particular for every set $F_n$ in an \emph{a priori} selected F\o lner \sq, there exists some $x_n\in X$ and $g_n\in G$ with
$$
\inf_{\mu\in\M_G(X)}(\mu(\mathsf B)-\mu(\mathsf A))-\eps > \frac1{|F_n|}(|B_{x_n}\cap F_ng_n|-|A_{x_n}\cap F_ng_n|).
$$
Note that $|B_{x_n}\cap F_ng_n|=|\{f\in F_n:fg_n(x_n)\in \mathsf B\}|$ (and analogously for $\mathsf A$),
thus the right hand side takes on the form
$$
\frac1{|F_n|}(|\{f\in F_n:fg_n(x_n)\in \mathsf B\}|-|\{f\in F_n:fg_n(x_n)\in \mathsf A\}|).
$$
The function $\mathsf W\mapsto \frac1{|F_n|}|\{f\in F_n:fg_n(x_n)\in \mathsf W\}|$ defined on Borel subsets of $X$ is equal to the probability measure $\frac1{|F_n|}\sum_{f\in F_n}\delta_{fg_n(x_n)}$. This \sq\ of measures has a sub\sq\ convergent in the weak-star topology to some $\mu_0\in\M_G(X)$. Since the characteristic functions of the clopen sets $\mathsf A, \mathsf B$ are continuous, we have
$$
\inf_{\mu\in\M_G(X)}(\mu(\mathsf B)-\mu(\mathsf A))-\eps\ge \mu_0(\mathsf B)-\mu_0(\mathsf A),
$$
which is a contradiction. We have proved that $\inf_{\mu\in\M_G(X)}(\mu(\mathsf B)-\mu(\mathsf A))\le\underline D(\mathsf B,\mathsf A)$. The inequality $\underline D(\mathsf B,\mathsf A)\le\inf_{x\in X}\underline D(B_x,A_x)$ is trivial; both sides differ by changing the order of $\limsup_n$ and $\inf_x$ and on the left the infimum is applied earlier.

For the last missing inequality, $\inf_{x\in X}\underline D(B_x,A_x)\le \inf_{\mu\in\M_G(X)}(\mu(\mathsf B)-\mu(\mathsf A))$, we shall invoke the ergodic theorem (Theorem \ref{ergodic}). Notice that, given $\eps>0$, there exists an ergodic measure $\mu_0\in\M_G(X)$ with $\mu_0(\mathsf B)-\mu_0(\mathsf A) < \inf_{\mu\in\M_G(X)}(\mu(\mathsf B)-\mu(\mathsf A))+\eps$. The ergodic theorem now implies that there exists a sub\sq\ $(F_{n_k})_{k\in\N}$ of the F\o lner \sq\ (any tempered sub\sq\ will do), and a point $x\in X$ (in fact, $\mu_0$-almost every point is good)
such that
$$
\lim_{k\to\infty}\frac1{|F_{n_k}|}|\{f\in F_{n_k}:f(x)\in\mathsf A\}|= \mu_0(\mathsf A)
$$
and an analogous formula holds for $\mathsf B$. Obviously, $f(x)\in\mathsf A$ or $\mathsf B$, if and only if $f\in A_x$ or $B_x$, respectively. Hence
$$
\lim_{k\to\infty}\frac1{|F_{n_k}|}|A_x\cap F_{n_k}| =
\mu_0(\mathsf A)\ \ \  (\text{and similarly for $B_x$ and $\mathsf B$}).
$$
Thus, for each sufficiently large $k$ we have
$$
\frac1{|F_{n_k}|}\bigl(|B_x\cap F_{n_k}|-|A_x\cap F_{n_k}|\bigr)< \mu_0(\mathsf B)-\mu_0(\mathsf A)+\eps<\inf_{\mu\in\M_G(X)}(\mu(\mathsf B)-\mu(\mathsf A))+2\eps.
$$
Clearly, the left hand side is not smaller than
$$
\inf_{g\in G}\frac1{|F_{n_k}|}(|B_x\cap F_{n_k}g|-|A_x\cap F_{n_k}g|)=\underline D_{F_{n_k}}(B_x,A_x).
$$
Passing to the limit over $k$ and then applying infimum over all $x\in X$ we obtain
$\inf_{x\in X}\underline D(B_x,A_x)\le \inf_{\mu\in\M_G(X)}(\mu(\mathsf B)-\mu(\mathsf A))+2\eps$.
Since this is true for every $\eps>0$, (1) is proved.

\smallskip
We pass to proving (2). As before, the last equality suffices. From (1) applied to the cylinders $\mathsf A=[\mathsf 1]$ and $\mathsf B=[\mathsf 2]$, and applying Lemma \ref{bbb}, we obtain
$$
\inf_{\mu\in\M^{AB}}(\mu([2])-\mu([1]))=\underline D([\mathsf 2],[\mathsf 1])
=\lim_{n\to\infty}\ \,\inf_{y\in Y^{AB}}\ \inf_{g\in G} \frac1{|F_n|} (|B_y\cap F_ng|-|A_y\cap F_ng|).
$$
The above difference $|B_y\cap F_ng|-|A_y\cap F_ng|$ depends on the block $y|_{F_ng}$.
Notice that we are considering a transitive subshift with the transitive point $y^{AB}$ (i.e., whose orbit is dense in the subshift), so every block $y|_{F_ng}$ (for any $y\in Y^{AB}$ and any $g\in G$) occurrs also in $y^{AB}$ as a block $y^{AB}|_{F_ng'}$ for some $g'$ (the converse need not be true, unless $y$ is another transitive point).
Thus, for any $n$, the infimum over $y\in Y^{AB}$ on the right hand side of the formula displayed above is the smallest for $y=y^{AB}$. Recall that $A_{y^{AB}}=A$ and $B_{y^{AB}}=B$. We have proved that
$$
\inf_{\mu\in\M^{AB}}(\mu([2])-\mu([1])) = \lim_{n\to\infty}\ \inf_{g\in G}\frac1{|F_n|} (|B\cap F_ng|-|A\cap F_ng|).
$$
The right hand side is precisely $\underline D(B,A)$.
\end{proof}

Recall the definition of a block code (Definition \ref{bcode}). Based on that, we define the following notion:

\begin{definition}\label{local rule}
Let $X\subset\Lambda^G$ be a subshift. For each $x\in X$ let $A_x\subset G$ and let $\tilde\varphi_x: A_x\to G$ be some function. For $X'\subset X$, we will say that the family $\{\tilde\varphi_x\}_{x\in X'}$ is \emph{determined by a block code} if there exists a block code $\Xi:\Lambda^F\to E$, where $E$ is a finite subset of $G$ (and so is $F$), such that if we denote
$$
\varphi_x(g) = \Xi(g(x)|_F),
$$
($x\in X, g\in G$),
then, for each $x\in X'$, the mapping from $A_x$ to $G$, defined by
$$
a\mapsto \varphi_x(a)a,
$$
($a\in A_x$), coincides with $\tilde\varphi_x$. The elements $\varphi_x(a)$ (belonging to $E$) will be called the \emph{multipliers} of $\tilde\varphi_x$.
\end{definition}

A simple way of checking, that a family $\{\tilde\varphi_x\}_{x\in X'}$ is determined by a block code, is finding a finite set $F$ such that,
for any $x_1,x_2\in X'$ and $a_1\in A_{x_1}, a_2\in A_{x_2}$,
\begin{equation}\label{clr}
a_1(x_1)|_F = a_2(x_2)|_F \ \implies \
\tilde\varphi_{x_1}(a_1)a_1^{-1} = \tilde\varphi_{x_2}(a_2){a_2}^{-1}.
\end{equation}

The following theorem connects the above definition with the relation of subequivalence.

\begin{theorem}\label{tutka}
\begin{enumerate}
\item Let $X\subset\Lambda^G$ be a subshift. Consider the pair of disjoint clopen subsets $\mathsf A, \mathsf B\subset X$. Then $\mathsf A\preccurlyeq\mathsf B$ if and only if there exists a family of functions $\tilde\varphi_x:G\to G$, indexed by $x\in X$, determined by a block code, such that for all $x\in X$, $\tilde\varphi_x$ restricted to $A_x=\{g:g(x)\in\mathsf A\}$ is an injection to $B_x=\{g:g(x)\in\mathsf B\}$.
\item If, moreover, $X$ is transitive with a transitive point $x^*$, then the above condition $\mathsf A\preccurlyeq\mathsf B$ is equivalent to the existence of just one function $\tilde\varphi_{x^*}$ determined by a block code, whose restriction to $A_{x^*}$ is an injection to $B_{x^*}$.
\end{enumerate}
\end{theorem}

\begin{proof}
(1) Firstly, suppose that $\mathsf A\preccurlyeq\mathsf B$. Let $\{\mathsf A_1,\mathsf A_2,\dots,\mathsf A_k\}$ be the clopen partition of $\mathsf A$ and let $g_1,g_2,\dots, g_k$ be the elements of $G$ such that the sets $\mathsf B_i = g_i(\mathsf A_i)$ are disjoint subsets of $\mathsf B$. Let $E = \{g_1,g_2,\dots,g_k\}$. Consider the mapping $\xi:X\to E^G$ given by the following rule
$$
(\xi(x))_g = \begin{cases}
g_i &\text{ if \ }g(x)\in\mathsf A_i, \ i=1,2,\dots,k,\\
g_1 &\text{ otherwise},
\end{cases}
$$
($g\in G$).
Since the sets $\mathsf A_i$ and $X\setminus \mathsf A$ are clopen in $X$, the above map is continuous and, as easily verified, it is shift-equivariant. Thus, it is a \tl\ factor map from $X$ into $E^G$. By Theorem \ref{CHL}, there exists a block code $\Xi:\Lambda^F\to E$ (with some finite coding horizon $F$) satisfying, for all $x\in X$ and $g\in G$, the equality
$$
(\xi(x))_g = \Xi(g(x)|_F).
$$
For each $x\in X$ we define $\varphi_x:G\to E$ by $\varphi_x(g)= (\xi(x))_g$ and $\tilde\varphi_x:G\to G$ by $\tilde\varphi_x(g)=\varphi_x(g)g$, i.e., the family of maps $\{\tilde\varphi_x\}_x$ is determined by the block code $\Xi$. We need to show that, for every $x\in X$, $\tilde\varphi_x$ restricted to $A_x$ is an injection to $B_x$.

Throughout this paragraph we fix some $x\in X$ and skip the subscript $x$ in the writing of $A_x, B_x$, $\varphi_x$ and $\tilde\varphi_x$. For $i=1,2,\dots,k$ let $A_i = A\cap\varphi^{-1}(g_i)$. Clearly, $\{A_1,A_2,\dots,A_k\}$ is a partition of $A$ and for every $a\in A$ we have:
$$
a\in A_i \iff \varphi(a)=g_i \iff (\xi(x))_a = g_i \iff  a(x) \in \mathsf A_i, \ \  (i=1,2,\dots,k).
$$
Further, $a(x)\in\mathsf A_i$ yields $g_ia(x)\in\mathsf B_i\subset \mathsf B$, which implies that $g_ia\in B$. Since $g_ia = \varphi(a)a = \tilde\varphi(a)$, we have shown that $\tilde\varphi$ sends $A$ into $B$. For injectivity of the restriction $\tilde\varphi|_A$, observe that if $a_1\neq a_2$ and both elements belong to the same set $A_i$ then their images by $\tilde\varphi$, equal to $g_ia_1$ and $g_ia_2$, respectively, are different by cancellativity. If $a_1\in A_i$ and $a_2\in A_j$ with $i\neq j$, then
$\tilde\varphi(a_1)(x)=g_ia_1(x)\in\mathsf B_i$ and $\tilde\varphi(a_2)(x)=g_ja_2(x)\in\mathsf B_j$. Since $\mathsf B_i$ and $\mathsf B_j$ are disjoint, the elements $\tilde\varphi(a_1)$ and $\tilde\varphi(a_2)$ must be different.

\smallskip
Now suppose that there exist injections $\tilde\varphi_x:A_x\to B_x$ (for all $x\in X$) determined by a block code $\Xi:\Lambda^F\to E=\{g_1,g_2,\dots,g_k\}\subset G$, where the elements $g_i$ are written without repetitions, i.e., are different for different indices $i=1,2,\dots,k$. That is, denoting, for each $g\in G$,
$$
\varphi_x(g) = \Xi(g(x)|_F),
$$
we obtain maps $\varphi_x$ such that $g\mapsto\varphi_x(g)g$ restricted to $A_x$ coincides with $\tilde\varphi_x$. Now, for each $i=1,2,\dots,k$ we define
$$
\mathsf A_i=\mathsf A\cap [\Xi^{-1}(g_i)]=\{x\in\mathsf A:\Xi(x|_F)=g_i\}.
$$
Clearly, $\{\mathsf A_1,\mathsf A_2,\dots,\mathsf A_k\}$ is a clopen partition of $\mathsf A$. Let $x\in \mathsf A_i$ (for some $i=1,2,\dots, k$). Then $e\in A_x$ and thus $\tilde\varphi_x(e)\in B_x$, i.e., $\tilde\varphi_x(e)(x)\in\mathsf B$. But $\tilde\varphi_x(e)=\varphi_x(e)=\Xi(x|_F)=g_i$. We have shown that $g_i(\mathsf A_i)\subset\mathsf B$.

It remains to show that the sets $g_i(\mathsf A_i)$ are disjoint. Suppose that for some $i\neq j$ there exists $x\in X$ belonging to both $g_i(\mathsf A_i)$ and $g_j(\mathsf A_j)$.
This implies that $g_i^{-1}$ and $g_j^{-1}$ both belong to $A_x$, and $\varphi_x(g_i^{-1}) = g_i$, $\varphi_x(g_j^{-1}) = g_j$. But then
$$
\tilde\varphi_x(g_i^{-1})= \varphi_x(g_i^{-1}) g^{-1}_i= g_ig^{-1}_i=e\text{ \ and \ }\tilde\varphi_x(g_j^{-1})=\varphi_x(g_j^{-1}) g^{-1}_j= g_jg^{-1}_j=e,
$$
which contradicts the injectivity of $\tilde\varphi_x$ on $A_x$.

\smallskip
(2) In view of (1), it suffices to show that if a block code $\Xi:\Lambda^F\to E$ determines an injection $\tilde\varphi_{x^*}:A_{x^*}\to B_{x^*}$ then it also determines (as usually, by the formulas $\varphi_x(g)=\Xi(g(x)|_F)$ and $\tilde\varphi_x(a)=\varphi_x(a)a$\,)
injections $\tilde\varphi_x:A_x\to B_x$ for all $x\in X$.
Fix some $x\in X$ and let $a_1\neq a_2$ belong to $A_x$, i.e., $a_1(x), a_2(x)\in\mathsf A$. Since $x^*$ is a transitive point,
a point $g(x^*)$ (for some $g\in G$) is so close to $x$ that:
\begin{enumerate}
	\item[(a)] $a_1g(x^*),\ a_2g(x^*)\in\mathsf A$,
	\item[(b)] the blocks $g(x^*)|_{Fa_1\cup Fa_2}$ and $x|_{Fa_1\cup Fa_2}$ are equal,
	\item[(c)] $(\forall f\in Ea_1\cup Ea_2) \ \ fg(x^*)\in\mathsf B \iff f(x)\in\mathsf
	B$.
\end{enumerate}
By (a), both $a_1g$ and $a_2g$ belong to $A_{x^*}$. Thus $\tilde\varphi_{x^*}(a_1g)$
and $\tilde\varphi_{x^*}(a_2g)$ are \emph{different} elements of $B_{x^*}$. But
$$
\tilde\varphi_{x^*}(a_1g)=\varphi_{x^*}(a_1g)a_1g\text{ \ \ and \ \ }\tilde\varphi_{x^*}(a_2g)=\varphi_{x^*}(a_2g)a_2g,
$$
which, after canceling $g$, yields
$$
\varphi_{x^*}(a_1g)a_1\neq\varphi_{x^*}(a_2g)a_2.
$$
On the other hand, by (b), $x|_{Fa_1} = g(x^*)|_{Fa_1}$, whence $a_1(x)|_F = a_1g(x^*)|_F$, and
$$
\varphi_{x}(a_1)=\Xi(a_1(x)|_F) = \Xi(a_1g(x^*)|_F) = \varphi_{x^*}(a_1g),
$$
which means that $\tilde\varphi_{x}(a_1)=\varphi_{x}(a_1)a_1=\varphi_{x^*}(a_1g)a_1$.
Analogously, $\tilde\varphi_{x}(a_2)=\varphi_{x^*}(a_2g)a_2$. We have shown that
$\tilde\varphi_x(a_1)\neq \tilde\varphi_x(a_2)$, i.e., $\tilde\varphi_x$ restricted to $A_x$ is injective.

Further, the fact that $\tilde\varphi_{x^*}(a_1g)\in B_{x^*}$ yields
$$
\mathsf B\ni\tilde\varphi_{x^*}(a_1g)(x^*)=\varphi_{x^*}(a_1g)a_1g(x^*)=
\varphi_{x}(a_1)a_1g(x^*).
$$
Since $\varphi_x(a_1)a_1\in Ea_1$, by (c) we get
$$
\mathsf B\ni\varphi_{x}(a_1)a_1(x)=\tilde\varphi_x(a_1)(x),
$$
and hence $\tilde\varphi_x(a_1)\in B_x$. We have shown that $\tilde\varphi_x$ sends $A_x$ injectively to $B_x$.
\end{proof}

\subsubsection{Banach density comparison property of a group}
\begin{definition}
We say that $G$ has the \emph{Banach density comparison property} if whenever
$A\subset G$ and $B\subset G$ are disjoint and satisfy $\underline D(B,A)>0$ then,
in the subshift \,$Y^{AB}$ there exists an injection $\tilde\varphi:A\to B$ determined by a block code (recall that $y^{AB}$ is a transitive point in $Y^{AB}$ and $A=A_{y^{AB}},\ B=B_{y^{AB}}$, so the above condition is the same as that in Theorem \ref{tutka} (2)).
\end{definition}

\begin{remark}
It is immediate to see that any finite group has the Banach density comparison property.
\end{remark}

%We continue to consider two subsets $A,B$ of $G$. We have the following relation of upper and lower Banach densities, and the Banach density advantage, with invariant measures.

We can now completely characterize the comparison property of a countable amenable group in terms of the Banach density comparison property.

\begin{theorem}\label{ujowe}
A countable amenable group $G$ has the comparison property if and only if it has the Banach density comparison property.
\end{theorem}

\begin{proof}
The theorem holds trivially for finite groups, so we can restrict to infinite groups $G$. Assume that $G$ has the comparison property and let $A,B\subset G$ be disjoint and satisfy $\underline D(B,A)>0$. Then, by Proposition \ref{prop} (2), taking in the subshift $Y^{AB}$ the clopen sets: $\mathsf A=[\mathsf 1]$ and $\mathsf B=[\mathsf 2]$, we have $\inf_{\mu\in\M^{AB}}(\mu(\mathsf B)-\mu(\mathsf A))>0$. By the assumption, $\mathsf A\preccurlyeq\mathsf B$. Now, a direct application of Theorem \ref{tutka} (2) completes the proof of the Banach density comparison property.

Let us pass to the proof of the opposite implication.
Suppose that a countable amenable group $G$ having the Banach density comparison property acts on a \zd\ compact metric space $X$, in which we have selected two clopen sets $\mathsf A$ and $\mathsf B$ satisfying, for each \im\ $\mu$ on $X$, the inequality $\mu(\mathsf A)<\mu(\mathsf B)$.
By Remark \ref{disjoint}, we can assume that $\mathsf A$ and $\mathsf B$ are disjoint; and
by Remark \ref{from0}, we have $\inf_{\mu\in\M_G(X)}(\mu(\mathsf B)-\mu(\mathsf A))>0$. This translates to $\inf_{\nu\in\M_{\mathsf A\mathsf B}}(\nu([\mathsf 2])-\nu([\mathsf 1]))>0$ in the factor subshift $Y_{\mathsf A\mathsf B}$. By Proposition \ref{prop} (1) applied to this subshift, we get $\underline D([\mathsf 2],[\mathsf 1])>0$.

Since we intend to use the Banach density comparison property and Theorem~\ref{tutka}~(2), we need to embed $Y_{\mathsf A\mathsf B}$ in a transitive subshift $Y$ (over the alphabet $\{\mathsf0,\mathsf1,\mathsf2\}$). We also desire a transitive point $y^*$ which satisfies
$\underline D(B_{y^*},A_{y^*})>0$. Below we present the construction of such a transitive subshift.

Choose some positive $\gamma<\underline D([\mathsf 2],[\mathsf 1])$. Fix an increasing (w.r.t. set inclusion) F\o lner \sq\ $(F_n)_{n\in\N}$ such that $\bigcup_{n=1}^\infty F_n=G$. By choosing a sub\sq\ we can assume that $\sum_{i=1}^{n-1}|F_i|<\frac{1-\gamma}2|F_n|$ for every $n$ (in this place we use the assumption that $G$ is infinite). Next, we need to find a \sq\ of blocks $C_n\in\{\mathsf0,\mathsf1,\mathsf2\}^{F_n}$ each appearing as $y_n|_{F_n}$ in some $y_n\in Y_{\mathsf A\mathsf B}$, such that every $y\in Y_{\mathsf A\mathsf B}$ is a coordinatewise limit of a subsequence $C_{n_k}$ of the selected blocks. Finally, we need to find a \sq\ $g_n$ of elements of $G$ such that the sets $F_nF_n^{-1}F_ng_n$ are disjoint. All the above steps are possible and easy. Once they are completed, $y^*$ is defined by the rule: for each $n$ and $f\in F_n$ we put $y^*_{fg_n}=C_n (f)$, and for all $g$ outside the union $\bigcup_{n=1}^\infty F_ng_n$, we put $y^*_g = \mathsf 2$. We let $Y$ be the closure of the orbit of $y^*$.

The following properties hold:
\begin{itemize}
	\item $Y\supset Y_{\mathsf A\mathsf B}$,
	\item $\underline D(B_{y^*},A_{y^*})\ge \gamma>0$.
\end{itemize}
The first property is obvious by construction: each $y\in Y_{\mathsf A\mathsf B}$ is the limit of a sequence of blocks $C_{n_k}$, hence it is also the limit of the \sq\ of elements $g_{n_k}(y^*)$, and thus it belongs to $Y$.

We need to prove the latter property. By the definition of $\underline D([\mathsf 2],[\mathsf 1])$ in the subshift $Y_{\mathsf A\mathsf B}$, there exist arbitrarily large indices $n_k$ such that
\begin{equation}\label{noco}
|\{f\in F_{n_k}:y_{fg}=\mathsf 2\}|-|\{f\in F_{n_k}:y_{fg}=\mathsf 1\}|\ge\gamma|F_{n_k}|,
\end{equation}
for all $y\in Y_{\mathsf A\mathsf B}$ and $g\in G$. It suffices to show an analogous property for $y^*$.

Fix some $g\in G$ and observe the block $y^*|_{F_{n_k}g}$. The set $F_{n_k}g$ either does not intersect any of the sets $F_mg_m$ with $m\ge n_k$ or intersects one of them (say $F_{m_0}g_{m_0}$ with $m_0\ge n_k$).

In the first case, the block $y^*|_{F_{n_k}g}$ consists mostly of symbols $\mathsf 2$; as all symbols different from $\mathsf 2$ appear in $y^*$ only over the intersection of $F_{n_k}g$ with the union of the sets $F_ig_i$ with $i<n_k$, the percentage of such symbols in $y^*|_{F_{n_k}g}$ is at most
$$
\frac1{|F_{n_k}g|}\sum_{i=1}^{n_k-1}|F_ig_i|=\frac1{|F_{n_k}|}\sum_{i=1}^{n_k-1}|F_i|<\frac{1-\gamma}2.
$$
Thus, in this case we have
\begin{equation}\label{noco1}
|\{f\in F_{n_k}:y^*_{fg}=\mathsf 2\}|-|\{f\in F_{n_k}:y^*_{fg}=\mathsf 1\}|\ge\gamma|F_{n_k}|.
\end{equation}
In the latter case, we have $g\in F_{n_k}^{-1}F_{m_0}g_{m_0}$, hence $F_{n_k}g\subset F_{n_k}F_{n_k}^{-1}F_{m_0}g_{m_0}\subset F_{m_0}F_{m_0}^{-1}F_{m_0}g_{m_0}$. By disjointness of the sets $F_nF_n^{-1}F_ng_n$, $F_{n_k}g$ does not intersect any set $F_nF_n^{-1}F_ng_n$ (and hence also $F_ng_n$) with $n\neq m_0$. We will compare the block $y^*|_{F_{n_k}g}$ with the block $y_{m_0}|_{F_{n_k}gg_{m_0}^{-1}}$. We can write
$$
F_{n_k}g = (F_{n_k}g\cap F_{m_0}g_{m_0}) \cup (F_{n_k}g\setminus F_{m_0}g_{m_0}),
$$
and likewise
$$
F_{n_k}gg_{m_0}^{-1} = (F_{n_k}gg_{m_0}^{-1}\cap F_{m_0})\cup (F_{n_k}gg_{m_0}^{-1}\setminus F_{m_0}).
$$
By the definition of $y^*$, the block $y^*|_{F_{n_k}g\cap F_{m_0}g_{m_0}}$ is identical to $y_{m_0}|_{F_{n_k}gg_{m_0}^{-1}\cap F_{m_0}}$, while $y^*|_{F_{n_k}g\setminus F_{m_0}g_{m_0}}$ contains just the symbols $\mathsf 2$. Thus
the difference
$$
|\{f\in F_{n_k}:y^*_{fg}=\mathsf 2\}|-|\{f\in F_{n_k}:y^*_{fg}=\mathsf 1\}|
$$
is not smaller than
$$
|\{f\in F_{n_k}: (y_{m_0})_{fgg_{m_0}^{-1}}=\mathsf 2\}|-|\{f\in F_{n_k}:(y_{m_0})_{fgg_{m_0}^{-1}}=\mathsf 1\}|.
$$
Since $y_{m_0}\in Y_{\mathsf A\mathsf B}$, \eqref{noco} implies that the latter expression is at least $\gamma|F_{n_k}|$. We have proved \eqref{noco1} also in this case.

We have proved that $\underline D(B_{y^*},A_{y^*})\ge\gamma>0$. Now, the Banach density comparison property of $G$ implies that there exists an injection $\tilde\varphi$ from $A_{y^*}$ to $B_{y^*}$ determined by a block code. Thus, by Theorem~\ref{tutka} (2), we get $[\mathsf 1]\preccurlyeq[\mathsf 2]$ in the transitive subshift $Y$, and by restriction to a closed \inv\ set the same holds in $Y_{\mathsf A\mathsf B}$, which, by an application of $\pi_{\mathsf A\mathsf B}^{-1}$, translates to $\mathsf A\preccurlyeq\mathsf B$ in~$X$.
\end{proof}

\subsubsection{Comparison property via finitely generated subgroups}

\begin{definition}
In a group $G$, a set $R$ such that $\bigcup_{n=1}^\infty (R\cup R^{-1})^n=G$ is called a \emph{generator} of $G$. A group having a finite generator is called \emph{finitely generated}.
\end{definition}

As we shall see, the comparison property of a group is determined by its all finitely generated subgroups.

\begin{lemma}\label{supinf}
Let $G$ act on a \zd\ compact metric space $X$. Let $\mathsf A,\mathsf B\subset X$ be two disjoint clopen sets. Then
$$
\sup_H \inf_{\mu\in\M_H(X)}(\mu(\mathsf B)-\mu(\mathsf A))=\sup_{H'} \inf_{\mu\in\M_{H'}(X)}(\mu(\mathsf B)-\mu(\mathsf A))=\inf_{\mu\in\M_G(X)}(\mu(\mathsf B)-\mu(\mathsf A))
$$
where $H$ ranges over all finitely generated subgroups of $G$ and $H'$ ranges over all subgroups of $G$.
\end{lemma}

\begin{proof}
The inequality $\le$ on the left hand side is trivial, while the second inequality $\le$ follows easily from the fact that every measure invariant under the action of $G$ is invariant under the action of $H'$ for any subgroup $H'$ of $G$.

We need to prove the last missing inequality.
By Proposition \ref{prop} (1), we have $\inf_{\mu\in\M_G(X)}(\mu(\mathsf B)-\mu(\mathsf A))=\underline D(\mathsf B,\mathsf A)$. Then, for any positive $\delta$, there exists a finite set $F$ such that
$$
\frac1{|F|}(|B_x\cap Fg|-|A_x\cap Fg|)>\underline D(\mathsf B,\mathsf A)-\delta
$$
for every $x\in X$ and all $g\in G$, in particular for all $g\in H$, where $H$ is the subgroup generated by $F$. Thus, for every $x\in X$, we have
$$
\inf_{g\in H}\frac1{|F|}(|B_x\cap Fg|-|A_x\cap Fg|)\ge\underline D(\mathsf B,\mathsf A)-\delta.
$$
Since $F\subset H$ and $g\in H$, we have $A_x\cap Fg = (A_x\cap H)\cap Fg$. Note that $A_x\cap H$ equals the set $A_x$ defined for the induced action of $H$ on $X$ (and analogously for $B_x$). Thus, the expression on the left hand side above equals $\underline D_F(B_x,A_x)$ evaluated for the action of $H$ on $X$. Now, Lemma \ref{bbb} implies $\underline D(B_x,A_x)\ge\underline D(\mathsf B,\mathsf A)-\delta$ for every $x\in X$ (where $\underline D(B_x,A_x)$ is evaluated for the action of $H$ on $X$, and $\underline D(\mathsf B,\mathsf A)$ is evaluated for the action of $G$ on $X$), and Proposition \ref{prop} (1) yields
$$
\inf_{\mu\in\M_H(X)}(\mu(\mathsf B)-\mu(\mathsf A))\ge\underline D(\mathsf B,\mathsf A)-\delta=\inf_{\mu\in\M_G(X)}(\mu(\mathsf B)-\mu(\mathsf A))-\delta.
$$
After applying the supremum over $H$ on the left we can ignore $\delta$ on the right.
\end{proof}

\begin{prop}\label{44}
A countable amenable group $G$ has the comparison property if every finitely generated subgroup $H$ of $G$ has it.
\end{prop}

\begin{proof}
Let $G$ act on a \zd\ compact metric space $X$ and let $\mathsf A,\mathsf B\subset X$ be two disjoint clopen sets satisfying $\underline D(\mathsf B,\mathsf A)>0$. By the preceding lemma (and by Proposition \ref{prop} (1) used twice), there exists a finitely generated subgroup $H$ of $G$ such that the inequality $\underline D(\mathsf B,\mathsf A)>0$ holds also if $\underline D$ is evaluated for the action of $H$. By the comparison property of $H$, we get that $\mathsf A\preccurlyeq\mathsf B$ in this latter action. But this clearly implies the same subequivalence in the action by $G$.
\end{proof}

\begin{remark}
By the proof of Lemma \ref{supinf}, if $(H_n)_{n\in\N}$ is an increasing \sq\ of subgroups of $G$ such that  $G=\bigcup_{n=1}^\infty H_n$ then
$$
\inf_{\mu\in\M_G(X)}(\mu(\mathsf B)-\mu(\mathsf A))=\lim_{n\to\infty}\inf_{\mu\in\M_{H_n}(X)}(\mu(\mathsf B)-\mu(\mathsf A)).
$$
Thus, in Proposition \ref{44}, the assumption can be weakened to the existence of
an increasing \sq\ $(H_n)_{n\in \mathbb{N}}$ of subgroups of $G$ such that $G=\bigcup_{n=1}^\infty H_n$, and every $H_n$ has the comparison property.
\end{remark}

\begin{remark}
The converse implication in Proposition \ref{44} is a bit mysterious. On the one hand, since there are no examples of countable amenable groups without the comparison property, clearly, there is no counterexample for the implication in question. On the other hand, we failed to deduce the comparison property of a subgroup of $G$ from the comparison property of the group $G$.
\end{remark}

\subsubsection{Subexponential groups}

\begin{definition}
A finitely generated group $G$ with a generator $R$ has \emph{subexponential growth} if $|(R\cup R^{-1})^n|$ grows subexponentially, i.e.,
$$
\lim_{n\to\infty}\frac1n\log|(R\cup R^{-1})^n|=0.
$$
\end{definition}

It is very easy to see that subexponential growth of a finitely generated group $G$ implies subexponential growth of $|K^n|$ for any finite set $K\subset G$ and thus does not depend on the choice of a finite generator.

\begin{definition}
A countable group $G$ (not necessarily finitely generated) is called \emph{subexponential} if every its finitely generated subgroup has subexponential growth.
\end{definition}

It is a standard fact that a group $G$ is amenable if and only if so is every finitely generated subgroup of $G$. It is also known that finitely generated groups with subexponential growth are amenable \cite{AS}, hence every subexponential group is amenable. This is why we can omit the amenability assumption when dealing with subexponential groups. Examples of subexponential groups are: Abelian, nilpotent and virtually nilpotent groups. These examples have polynomial growth, but there are also examples of countable groups with intermediate growth rates \cite{Gr}. By a recent result \cite{BGT}, all finitely generated groups, which admit an increasing \sq\ of sets $(A_n)_{n\in\N}$ with $G=\bigcup_{n=1}^\infty A_n$ and $|A_n^2|<C|A_n|$ for some constant $C>0$, are virtually nilpotent and hence subexponential. In particular, this applies to finitely generated groups possessing a symmetric F\o lner \sq\ $(F_n)_{n\in\N}$ satisfying Tempelman's condition $|F_n^{-1}F_n|\le C|F_n|$.

\subsection{Comparison property of subexponential groups}\label{cztery}

This subsection contains our next important result: every subexponential group has the comparison property. The theorem is preceded by a few key definitions and lemmas.
\smallskip

\subsubsection{Correction chains}
We now introduce the key tool in the proof of the main result. The term $(\phi,E)$-chain reflects a remote analogy to $(f,\eps)$-chains in \tl\ dynamics.  Throughout this subsection, we let $A,B$ denote two disjoint subsets of a countable group $G$.

\begin{definition}Given a partially defined bijection $\phi:A'\to B'$, where $A'\subset A$ and $B'\subset B$, such that all multipliers $\phi(a)a^{-1}$ belong to a finite set $E\subset G$, by a \emph{$(\phi,E)$-chain of length $2n$} (or briefly just \emph{a chain}) we will mean a \sq\ $\mathbf C=(a_1,b_1,a_2,b_2,\dots,a_n,b_n)$ of \,$2n$ \emph{different} elements alternately belonging to $A$ and $B$, such that
$$
\text{for each }i=1,2,\dots,n,\ \  b_i\in Ea_i,
$$
and
$$
\text{for each }i=1,2,\dots,n-1, \ \ b_i\in B',\ \ a_{i+1}\in A' \text{ \ and \ } b_i = \phi(a_{i+1})
$$
(in particular, $b_i\in Ea_{i+1}$).
\end{definition}
The $(\phi,E)$-chains starting at a point $a_1\in A\setminus A'$ and ending at a point $b_n \in B\setminus B'$ are of special importance, as they allow one to ``correct'' the mapping and include $a_1$ in the domain and $b_n$ in the range.

\begin{definition}
A $(\phi,E)$-chain $\mathbf C=(a_1,b_1,a_2,b_2,\dots,a_n,b_n)$ will be called a \emph{$\phi$-correction chain} if $a_1\in A\setminus A'$ and $b_n\in B\setminus B'$.
With each $\phi$-correction chain $\mathbf C$ we associate the \emph{correction of $\phi$ along $\mathbf C$}. The corrected map denoted by $\phi^{\mathbf C}$ is defined on $A'\cup\{a_1\}$ onto $B'\cup\{b_{n}\}$, as follows: for each $i=1,2,\dots,n$ we let
$$
\phi^{\mathbf C}(a_i) = b_i,
$$
and for all other points $a\in A'$ we let $\phi^{\mathbf C}(a)=\phi(a)$.
\end{definition}

The correction may be visualized as follows (solid arrows in the top row represent the map $\phi$ and in the bottom row they represent $\phi^\mathbf C$; the dashed arrows represent the ``$E$-proximity relation'' $b\in Ea$):
\begin{gather*}
a_1\dashrightarrow b_1\longleftarrow a_2\dashrightarrow b_2\longleftarrow a_3\ \dots\ b_{n-1}\longleftarrow a_n\dashrightarrow b_n\\
\Downarrow\\
a_1\longrightarrow b_1\dashleftarrow a_2\longrightarrow b_2\dashleftarrow a_3\ \dots\ b_{n-1}\dashleftarrow a_n\longrightarrow b_n
\end{gather*}
(the dashed arrows become solid, the solid arrows are removed from the map). Notice that $\phi^{\mathbf C}$ still has all its multipliers $\phi^{\mathbf C}(a) a^{- 1}$ in the set $E$.
\smallskip

The problem with the correction chains is that the corresponding corrections of $\phi$ usually cannot be applied simultaneously. The correction chains may collide with each other, i.e., pass through common points and then the corresponding corrections rule each other out. To manage this problem we need to learn more about the possible collisions and then carefully select a family of mutually non-colliding correction chains. The details of this selection are given below.

\begin{definition}
Two $\phi$-correction chains \emph{collide} if they have a common point.
\end{definition}

Since the starting points of $\phi$-correction chains belong to $A\setminus A'$, the ending points belong to $B\setminus B'$, other odd points (counting along the chain) belong to $A'$, other even points belong to $B'$, where the above four sets are disjoint, and each even point is tied to the following odd point by the inverse map $\phi^{-1}$, each collision between two $\phi$-correction chains, say
$\mathbf C=(a_1,b_1,a_2,b_2,\dots,a_n,b_n)$ and $\mathbf C'=(a'_1,b'_1,a'_2,b'_2,\dots,a'_m,b'_m)$, is of one of the following three types:
\begin{itemize}
	\item \emph{common start}: $a_1=a'_1$,
	\item \emph{common end}: $b_n=b'_m$,
	\item all other collisions occur in pairs $(b_i,a_{i+1})=(b'_j,a'_{j+1})$ for some $1\le i<n$ and $1\le j<m$.
\end{itemize}
Of course, two chains may have more than one collision. Note that the definition of a $(\phi,E)$-chain eliminates the possibility of ``self-collisions'' in one chain.

\begin{definition}
Given a $(\phi,E)$-chain $\mathbf C=(a_1,b_1,a_2,b_2,a_3,\dots,a_n,b_n)$, the \sq\
$\mathbf n(\mathbf C)=(p_1,q_1,p_2,q_2,\dots,p_{n-1},q_{n-1},p_n)$, where
$p_i=b_ia_i^{-1}$ $(i=1,2,\dots,n)$ and $q_i=b_ia_{i+1}^{-1}$ $(i=1,2,\dots,n-1)$,
will be called the \emph{name} of $\mathbf C$.
\end{definition}
Notice that the name is always a \sq\ of elements of $E$, of length $2n-1$.

\begin{lemma}\label{shorter}
If two different $\phi$-correction chains have the same name (note that their lengths are then equal) and collide with each other then each of them collides also with a strictly shorter $\phi$-correction chain.
\end{lemma}

\begin{proof}
It is obvious that if two $\phi$-correction chains with the same name, say
$$
\mathbf C=(a_1,b_1,a_2,b_2,a_3,\dots,a_n,b_n), \ \ \mathbf C'=(a'_1,b'_1,a'_2,b'_2,a'_3,\dots,a'_n,b'_n),
$$
have the common start $a_1=a_1'$ or the common end $b_n=b_n'$, or a common pair
$(b_i, a_{i+1})=(b'_i,a'_{i+1})$ with the same index $i=1,2,\dots,n-1$, then the chains are equal. The only possible collision between two different $\phi$-correction chains with the same name is that they have a common pair $(b_i,a_{i+1})=(b'_j,a'_{j+1})$ with $i\neq j$. Let $i_0$ be the smallest index appearing in the role of $i$ or $j$ in the collisions
of $\mathbf C$ with $\mathbf C'$ and assume that it plays the role of $i$ (with some corresponding $j$).  Then
$$
(a_1,b_1,a_2,b_2,a_3,\dots,a_{i_0},b_{i_0},a_{i_0+1},b'_{j+1},a'_{j+2},\dots,a'_n,b'_n)
$$
is a $\phi$-correction chain (it has no self-collisions) of length strictly smaller than $2n$, and clearly it collides with both $\mathbf C$ and $\mathbf C'$.
\end{proof}

We enumerate $E$ (arbitrarily) as $\{g_1,g_2,\dots,g_k\}$. We define
$$
\mathbf N=\bigcup_{n=1}^\infty E^{\times2n-1},
$$
which means the disjoint union of the $(2n-1)$-fold Cartesian products of copies of $E$.
This set can be interpreted as the collection of all ``potential'' names of the correction chains of
any partially defined bijection from $A$ to $B$ with the multipliers in $E$. The enumeration of $E$ induces the following linear order on $\mathbf N$:
$$
\mathbf n<\mathbf n'\ \ \iff\ \ |\mathbf n|<|\mathbf n'| \ \vee \ (\,|\mathbf n|=|\mathbf n'| \ \wedge\ \mathbf n<\mathbf n'\,),
$$
where $|\mathbf n|$ denotes the length of $\mathbf n$ and the last inequality is with respect to the lexicographical order on $E^{\times|\mathbf n|}$.

\begin{definition}
A $\phi$-correction chain $\mathbf C$ is \emph{minimal} if it does not collide with any other $\phi$-correction chain whose name precedes $\mathbf n(\mathbf C)$ in the above defined order
on~$\mathbf N$.
\end{definition}

\begin{lemma}\label{minnox}
Minimal $\phi$-correction chains do not collide with each other.
\end{lemma}
\begin{proof}
If two $\phi$-correction chains with different names collide, one of them is not minimal. If two $\phi$-correction chains with the same name collide, by Lemma \ref{shorter} none of them is minimal.
\end{proof}

\begin{lemma}\label{minch}
Assume that $E$ is a symmetric set containing the unity $e$ and let $a_1\in A\setminus A'$. If there is a $\phi$-correction chain $\mathbf C$ of length $2n$, starting at $a_1$, then there exists a minimal $\phi$-correction chain of length at most $2n$ contained in the finite set $E^{s(n)}a_1$ (where $s(n)$ depends only
on $|E|$ and $n$).
\end{lemma}

\begin{proof}
If $\mathbf C$ itself is not minimal then it collides with a $\phi$-correction chain $\mathbf C_1$ with $\mathbf n(\mathbf C_1)<\mathbf n(\mathbf C)$ in $\mathbf N$. Clearly, $\mathbf C_1$ is entirely contained in $E^{4n}a_1$. If $\mathbf C_1$ is not minimal, then it collides with some $\mathbf C_2$, whose name precedes that of $\mathbf C_1$ (and hence also that of $\mathbf C$). Now, $\mathbf C_2$ is contained in $E^{6n}a_1$. This recursion may be repeated at most $\sigma_n-1=\sum_{i=1}^n|E|^{2n}-1$ times, because this number estimates the number of names preceding $\mathbf n(\mathbf C)$. So, before $\sigma_n$ steps are performed, a minimal $\phi$-correction chain must occur. Its length is at most $2n$ and it is entirely contained in $E^{2n\sigma_n}a_1$.
\end{proof}

It is the following lemma, where subexponentiality of the group comes into play. We also exploit the notion of tilings.

\begin{lemma}\label{key}
Let $G$ be a subexponential group. Let $\CT$ be a tiling of $G$ and let $\mathcal S$ denote the set of all shapes of $\CT$. Denote $E=\bigcup_{S\in\mathcal S}SS^{-1}$. Let $A,B$ be disjoint subsets of $G$ satisfying, for some $\eps>0$ and every tile $T$ of $\CT$, the inequality
$$
|B\cap T|-|A\cap T|>\eps|T|.
$$
Let $N\ge 1$ be such that for any $n\ge N$,
$$
\frac1n\log|(E^2)^n|<\log(1+\eps)
$$
(by the subexponentiality assumption, since $E^2$ is finite, such an $N$ exists).
Then, for any partially defined bijection $\phi:A'\to B'$ with $A'\subset A,\ B'\subset B$, such that all multipliers $\phi(a)a^{-1}$ are in $E$, for every point $a_1\in A\setminus A'$, there exists a $\phi$-correction chain of length at most $2N$, starting at $a_1$ (and ending in $B\setminus B'$).
\end{lemma}

\begin{proof}
For each tile $T$ of $\CT$ we have
$$
\frac{|B\cap T|}{|A\cap T|}\ge\frac{\eps|T|}{|A\cap T|}+1\ge 1+\eps
$$
(including the case when the denominator equals $0$).
Clearly, any \emph{$\CT$-saturated} finite set $Q$, i.e, being a union of tiles of $\CT$,
also satisfies
$$
\frac{|B\cap Q|}{|A\cap Q|}\ge1+\eps.
$$
For a set $P\subset G$, we define the \emph{$\CT$-saturation} $P^\CT$ of $P$ as the union of all tiles intersecting $P$:
$$
P^\CT= \bigcup\{T\in\CT:P\cap T\neq\emptyset\}.
$$
Obviously, $P^\CT\subset EP$.

Consider a point $a_1\in A\setminus A'$ (if $A\setminus A'=\emptyset$ then the statement of the theorem holds trivially). Let $T$ be the tile of $\CT$ containing $a_1$, i.e., $T=\{a_1\}^\CT$. Since $T$ contains $a_1$ (and thus $|A\cap T|\ge 1$), we have $|B\cap T|\ge 1+\eps$. There exist $(\phi,E)$-chains of length $2$ from $a_1$ to every $b\in B\cap T$. Now, there are two options:
\begin{itemize}
	\item either at least one of these chains is a $\phi$-correction chain (and then the construction is finished),
	\item or none of these chains is a $\phi$-correction chain, i.e., $B'\cap T=B\cap T$.
\end{itemize}
In the latter option we have $|B'\cap T|=|B\cap T|\ge1+\eps$, i.e., denoting
$$
P_1=\{a_1\} \text{ \ and \ } Q_1=T=P_1^{\CT},
$$
we have
$$
|B'\cap Q_1|\ge1+\eps.
$$

From now on we continue by induction. Suppose that for some $n\ge 1$ we have defined a $\CT$-saturated set $Q_n$ such that
\begin{enumerate}
	\item for every $b\in B\cap Q_n$ there exists a $(\phi,E)$-chain of length at most $2n$ from $a_1$ to $b$,
	\item $B\cap Q_n= B'\cap Q_n$ (i.e., there are no $\phi$-correction chains starting at $a_1$ and ending in
	$Q_n$), and
	\item $|B'\cap Q_n|\ge(1+\eps)^n$.
\end{enumerate}
Then we define $P_{n+1}=\phi^{-1}(Q_n)=\phi^{-1}(B'\cap Q_n)$. Bijectivity of $\phi$ implies that $|P_{n+1}|\ge(1+\eps)^n$. Let $Q_{n+1}$ denote the $\CT$-saturation $P_{n+1}^{\CT}$. Every point $b\in B\cap Q_{n+1}$ is of the form $g\phi^{-1}(b')$ with $g\in E$ and $b'\in B'\cap Q_n$, and, by (1), $b'$ can be reached from $a_1$ by a $(\phi,E)$-chain of length at most $2n$. Thus there exists a $(\phi,E)$-chain of length at most $2(n+1)$ from $a_1$ to every $b\in B\cap Q_{n+1}$. There are two options:
\begin{itemize}
	\item either at least one of these chains is a $\phi$-correction chain (then the construction is finished),
	\item or $B\cap Q_{n+1}=B'\cap Q_{n+1}$.
\end{itemize}
Suppose the latter option occurs. Since $Q_{n+1}$ is $\CT$-saturated, we have
$$
|B'\cap Q_{n+1}|=|B\cap Q_{n+1}|\ge(1+\eps)|A\cap Q_{n+1}|\ge(1+\eps)|P_{n+1}|\ge(1+\eps)^{n+1}.
$$
Now, (1)--(3) are fulfilled for $n+1$, so the induction can be continued.

Notice that for each $n$, $Q_n\subset EP_n$ and, by symmetry of the set $E$, $P_{n+1}\subset EQ_n$. As a consequence, we have $Q_{n+1}\subset E^{2n+1} a_1\subset (E^2)^{n+1} a_1$, and if the latter of the above options occurs, we have
$$
|(E^2)^{n+1}|\ge|Q_{n+1}|\ge|B'\cap Q_{n+1}|\ge(1+\eps)^{n+1},
$$
which implies that $n+1<N$ by the assumption. So, $n=N-2$ is the last integer for which nonexistence of $\phi$-correction chains of length $2(n+1)$ is possible. In the worst case scenario a correcting chain of length $2N$ must already exist.
\end{proof}

\begin{remark}It is absolutely crucial in the proof that we are using a tiling, not a \qt\ leaving some part of $G$ uncovered by the tiles. In such case, $a_1$ may be uncovered by the tiles, moreover, we would have no control as to how many elements of $P_{n+1}=\phi^{-1}(Q_n)$ are ``lost'' in the untiled part of $G$.
\end{remark}

\subsubsection{Proof of the comparison property of subexponential groups}
\begin{theorem}\label{main}
Every subexponential group $G$ has the comparison property.
\end{theorem}

\begin{proof}
By Proposition \ref{44}, it suffices to prove the theorem for finitely generated groups $G$ with subexponential growth, and Theorem \ref{ujowe} allows us to focus on the Banach density comparison property. So, let $G$ be a finitely generated group with subexponential growth. Let $A,B\subset G$ be disjoint and satisfy $\underline D(B,A)>0$. All we need is, in the subshift $Y^{AB}$, to construct an injection $\tilde\varphi:A\to B$ determined by a block code.

By Definition \ref{bad}, there exists a finite set $F\subset G$ such that $\underline D_F(B,A)>5\eps$ for some positive $\eps$. By Theorem~\ref{ext}, there exists an $(F,\eps)$-\inv\ tiling $\CT$ of $G$. We let $\mathcal S$ denote the set of all shapes of $\CT$. By Lemma \ref{bdc}, for every shape $S$ of $\CT$ we have $\underline D_S(B,A)>\eps$, in particular,
$$
|B\cap T|-|A\cap T|>\eps|T|,
$$
for every tile $T$ of $\CT$. Let $E=\bigcup_{S\in\mathcal S}SS^{-1}$ and say $E= \{g_1, g_2, \dots, g_k\}$.

We will  build the desired injection $\tilde\varphi:A\to B$ in a series of steps. The first approximation of $\tilde\varphi$ is the map $\phi_1$ defined on a subset of $A$ by
a procedure similar to that used in the proof of Lemma \ref{compisweakcomp}: we let $A_1=A\cap g_1^{-1}(B)$, and $B_1=g_1(A_1)\subset B$ and then, for each $j=2,3,\dots,k$ we define inductively
$$
A_j = A\setminus\Bigl(\bigcup_{i=1}^{j-1}A_i\Bigr)\cap g_j^{-1}\left(B\setminus\Bigl(\bigcup_{i=1}^{j-1} B_i\Bigr)\right)\ \text{and}\ B_j= g_j A_j\subset B.
$$
On each set $A_j$ (with $j=1,2,\dots,k$), $\phi_1$ is defined as the multiplication on the left by $g_j$. We let $A'_1=\bigcup_{i=1}^k A_i\subset A$ and $B'_1=\bigcup_{i=1}^k B_i\subset B$ denote the domain and range of $\phi_1$, respectively. The rule behind the construction of $\phi_1$ is as follows: for each $a\in A$ we first check whether $g_1a\in B$ and for those $a$ for which this is true, we assign $\phi_1(a)=g_1 a$. For other points $a$ we check whether $g_2a\in B$ and, unless $g_2a$ has already been assigned as $\phi_1(a')$ (for some $a'\in A$) in the previous step, we assign $\phi_1(a)=g_2 a$. And so on: at step $i$ we assign $\phi_1(a)=g_i a$ if $g_ia\in B$, unless $g_ia$ has already been assigned as $\phi(a')$ (for some $a'\in A$) at steps $1,2,\dots,i-1$. We stop when $i=k$. From this description it is easy to see that $\phi_1$ is an injection from $A_1'$ into $B_1'\subset B$. In fact, it is also seen that if $a_1,a_2\in A$ and
$$
a_1(y^{AB})|_{E^k} = a_2(y^{AB})|_{E^k},
$$
then either $\phi_1(a_1)a_1^{- 1}=\phi_1(a_2) a_2^{- 1}$ or both values of $\phi_1(a_1)$ and $\phi_1(a_2)$ are undefined. Using the criterion \eqref{clr} (for a one-element family $\A$), we conclude that $\phi_1$ restricted to its domain $A_1'$ is determined by a block code (with the coding horizon $E^k$). We remark, that the block code determines some extension of $\phi_1$ to the whole group, but we do not care about the values of the code outside $A_1'$ and we still treat $\phi_1$ as undefined outside $A_1'$. If $A'_1=A$ (which is rather unlikely in infinite groups), then the proof is finished.

Otherwise we continue the construction involving the correction chains and the associated corrections. By Lemma \ref{key}, for an appropriate $N$, every element $a_1\in A\setminus A_1'$ is the start of a $\phi_1$-correction chain of length at most $2N$. Next, by Lemma \ref{minch}, within $E^{s(N)}a_1$ there is a minimal $\phi_1$-correction chain of length at most $2N$. Finally, by Lemma \ref{minnox}, all minimal $\phi_1$-correction chains of lengths at most $2N$ do not collide with each other. Thus we can perform simultaneous corrections along all $\phi_1$-correction chains of lengths at most $2N$. The corrected map will be denoted by $\phi_2$. For each $a\in A\setminus A'_1$ perhaps we have not yet included $a$ in the domain $A'_2$ of $\phi_2$, but we have included in $A'_2$ at least one new point from $E^{s(N)}a\cap (A\setminus A_1')$. Clearly, $\phi_2$ sends $A'_2$ into $B$ and the multipliers of $\phi_2$ are contained in $E$.

We will now argue why $\phi_2$ is determined by a block code. Notice that given $a\in A$, finding all $\phi_1$-correction chains of lengths bounded by $2N$ starting at or passing through $a$ requires examining the values of $\phi_1$ at most in the set $E^{2N}a$. Then, given such a chain, we can decide whether it is minimal or not by examining all $\phi_1$-correction chains of lengths bounded by $2N$ which collide with it. For this, viewing the values of $\phi_1$ on the set $E^{4N}a$ suffices. Now suppose that $a_1,a_2\in A$ and
$$
a_1(y^{AB})|_{E^{k+4N}}=a_2(y^{AB})|_{E^{k+4N}}.
$$
Since $E^k$ is the coding horizon for $\phi_1$, we have
$$
a_1(\bar\phi_1)|_{E^{4N}}=a_2(\bar\phi_1)|_{E^{4N}},
$$
where $\bar\phi_1$ is defined as the symbolic element over the alphabet $E\cup\{\emptyset\}$ by the rule
$$
(\bar\phi_1)_g=\begin{cases}\phi_1(g)g^{-1}& \text{if }g\in A'_1,\\ \emptyset &\text{otherwise,}\end{cases}
$$
($g\in G$). This implies that $(r_1 a_1, s_1 a_1, r_2 a_1, s_2 a_1, \dots, r_n a_1, s_n a_1)$ is a (minimal) $\phi_1$-correction chain if and only if $(r_1 a_2, s_1 a_2, r_2 a_2, s_2 a_2, \dots, r_n a_2, s_n a_2)$ is a (minimal) $\phi_1$-correction chain, whenever $n\le N$ and all $r_i$ and $s_i$ belong to $E^{2N}$. Hence either both $a_1$ and $a_2$ lie on minimal $\phi_1$-correction chains of length at most $2N$, or both do not. In the latter case, since $a_1(y^{AB})|_{E^{k}}=a_2(y^{AB})|_{E^{k}}$, either $\phi_2(a_1)a_1^{-1}=\phi_1(a_1)a_1^{-1}=\phi_1(a_2)a_2^{-1}=\phi_2(a_2)a_2^{-1}$ or both $\phi_2(a_1)$ and $\phi_2(a_2)$ are undefined. In the former case, the lengths and names of the two minimal $\phi_1$-correction chains are the same, moreover $a_1$ and $a_2$ occupy equal positions in the corresponding chains. This implies that the multipliers $\phi_2(a_1)a_1^{-1}$ and $\phi_2(a_2)a_2^{-1}$ (although different than those for $\phi_1$) will both be defined and equal. So, $\phi_2$ is indeed determined by a block code.

The above process can be now repeated: the next map $\phi_3$ is obtained by performing simultaneous corrections along all minimal $\phi_2$-correction chains of lengths not exceeding $2N$. Again, for every $a\in A\setminus A'_2$, at least one point from each
set $E^{s(N)}a$ is included in the domain $A'_3$ of $\phi_3$ (the intersection $(A\setminus A'_2)\cap E^{s(N)}a$ is nonempty as it contains $a$, and often $a$ will be the new point included in $A'_3$). By the same arguments as before, the map $\phi_3$ is an injection from $A'_3$ into $B$ determined by a block code (with the coding horizon $E^{k+ 4N}$), and the multipliers of $\phi_3$ remain in $E$.

We claim that after a finite number $m$ of analogous steps all points of $A$ will be included in the domain of $\phi_m$, i.e., $\phi_m$ will be the desired injection $\tilde\varphi$ from $A$ into $B$. Indeed, a point $a\in A\setminus A_1'$ remains outside the domains of all the maps $\phi_i$ with $i\le m$ only if the number of all other points (except $a$) in $(A\setminus A_1')\cap E^{s(N)}a$ is at least $m- 1$ (because in each step at least one new point from this set is included in the domain). This is clearly impossible for $m> |E^{s(N)}|$, hence the desired finite number $m$ exists. By induction, all the maps $\phi_i$ ($i=1,2,\dots,m$) are determined by block codes (the coding horizon for the code which determines $\tilde\varphi=\phi_m$ is at most the set $E^{k+4Nm}$). This ends the proof.
\end{proof}

\subsubsection{Two questions}
As we have already mentioned, the problem whether all countable amenable groups have the comparison property is rather difficult. On the other hand, based on the experience with subexponential groups, one might hope that other additional assumptions might help as well. We formulate two relaxed, yet still open, versions of Question \ref{3.7}.

\begin{ques}
\begin{enumerate}
	\item Do all countable amenable residually finite groups have the comparison property?
	\item Do all countable amenable left (right) orderable groups have the comparison property?
\end{enumerate}
\end{ques}

\section{Encodable tiling systems}\label{s7}

In this last section of the paper, we shall prove the full version of the Symbolic Extension Entropy Theorem for two important classes of countable amenable groups: those which have the comparison property and those which are residually finite.

\subsection{Encodable systems of \qt s}

\begin{definition}
A \tl\ dynamical system will be called \emph{perfectly encodable} if it has an isomorphic symbolic extension. We will say that it is \emph{encodable} if it has a principal symbolic extension.
\end{definition}
For $\Z$-actions a full characterization of perfectly encodable systems privided in \cite{BuD}. In particular, any aperiodic (i.e., free) system with zero entropy is perfectly encodable. In the general case of actions of countable amenable groups an analogous theorem is unknown. The difficulty lies in encoding the zero entropy tiling system of Theorem~\ref{fs}. However, we are able to perfectly encode a zero entropy F\o lner system of disjoint \qt s, and the rest of this subsection is devoted to proving this:

\begin{theorem}\label{nine}
Let $G$ be a countable amenable group. There exists a perfectly encodable F\o lner system of disjoint \qt s $\hat{\mathbf T}$ of \tl\ entropy zero.
\end{theorem}

\begin{proof} The largest part of the proof is devoted to constructing a perfectly ecodable F\o lner system $\mathbf T$ of \qt s which are not yet disjoint, but they support an additional information allowing to create a conjugate disjoint version $\hat{\mathbf T}$. The construction of $\mathbf T$ starts with a zero entropy free action of $G$ on a \zd\ space $X$ whose existence is guaranteed by Theorem~\ref{DHZ}. All dynamical \qt s $\T_k$ ($k\in\N$) appearing below are \tl\ factors of $X$, delivered by Theorem~\ref{DH}, in particular they have \tl\ entropy zero. For the joining $\mathbf T$ we choose their ``natural joining'', i.e., as they appear joined in $X$. Unfortunately, encodability of a \sq\ of disjoint \qt s provided directly by Theorem \ref{DH1} (i.e., \cite[Corollary 3.5]{DH}) is uncertain and we need to introduce a slight modification in the constructions in \cite{DH} of both the $\varepsilon$-disjoint and disjoint \qt s.

\smallskip\noindent
\emph{Revision of the construction in \cite{DH}}. We need to  recall some portions of
the proofs of \cite[Lemma 3.4 and Corollary 3.5]{DH}. The first one contains a construction of a factor map $x\mapsto\CT_x$ where $x\in X$ ($X$ is a free \zd\ system) and $\CT_x$ is an $\varepsilon$-\qt\ with the set of shapes $\CS=\{F_{n_1},\dots,F_{n_r}\}$ (throughout, $(F_n)_{n\in\N}$ denotes a nested and symmetric F\o lner \sq\ starting with $F_1=\{e\}$). The tiles are distributed over $G$ in the reversed order: at first we distribute (for all $x\in X$) tiles with the largest shape $F_{n_r}$, then those with the shape $F_{n_{r-1}}$ and so on, until the smallest shape $F_{n_1}$. In each step $j=r,\dots,1$ we proceed as follows: we cover the space $X$ by finitely many clopen sets $U_{j,1},\dots,U_{j,m_j}$ such that, for each $i$, the images $g(U_{j,i})$ are pairwise disjoint for different $g\in F_{n_j}$. Next, we proceed by an (inner) induction over $i=1,2,\dots,m_j$ (each step of the resulting double induction is indexed by a pair $(j,i)$). At step $(j,i)$, we accept as tiles of $\CT_x$ these sets of the form $F_{n_j}g$ which satisfy:
\begin{enumerate}
	\item $g(x)\in U_{j,i}$ and
	\item $F_{n_j}g\setminus V_{j,i}$ is a $(1\!-\!\varepsilon)$-subset of
	$F_{n_j}g$, with $V_{j,i}$ abbreviating a complicated formula describing the union of all
	tiles accepted in all preceding steps (i.e., in steps $(j',i')$, where either $j'>j$ or
	$j'=j$ and $i'<i$).
\end{enumerate}
Later, in the proof of \cite[Corollary 2.5]{DH}, the tiles of the disjoint \qt\ $\hat\CT_x$, are exactly the above sets $F_{n_j}g\setminus V_{j,i}$. This is all we need to recall from \cite{DH}. Just observe, that the construction associates to each tile of $\CT_x$ a double index $(j,i)$ ($j=r,r-1,\dots,1,\ i=1,2,\dots,m_j$) which introduces a partial order among the tiles, such that if two different tiles are not disjoint then one strictly precedes another, and later the disjoint tiles are obtained by subtracting from each tile the union of all preceding tiles. The problem which we must solve now is that the partial order among the tiles in $\CT_x$ depends not only on $\CT_x$ but also on $x$. Thus, even if we prove that the $\varepsilon$-disjoint \qt s $\T=\{\CT_x:x\in X\}$ created for a decreasing to zero \sq\ $(\epsilon_k)_{k\in\N}$ constitute an encodable system of \qt s $\mathbf T=\bigvee_{k\in\N}\T_k$, this will \emph{not imply} encodability of the corresponding system of disjoint \qt s~$\hat{\mathbf T}$.

To resolve the problem at a minimized cost of changes in the original construction, we need to do three things:
\begin{enumerate}
	\item Choose the initial free system $X$ to be minimal (this is always possible, because
	each free system has a minimal subsystem which is also free).
	\item For each $j=r,r-1,\dots,1$ construct the cover $U_{j,1},U_{j,2},\dots,U_{j,m_j}$ in
	a more specific way: Choose the first clopen set $U_{j,1}$ arbitrarily (yet so that the
	sets $g(U_{j,1})$ are disjoint for $g\in F_{n_j}$). By minimality, there are finitely many
	elements $g_{j,1}= e,\ g_{j,2},\ \dots,\ g_{j,m_j}$ of $G$ such that
	$g_{j,1}(U_{j,1}),\ g_{j,2}(U_{j,1}),\ \dots,\ g_{j,m_j}(U_{j,1})$ cover $X$ (we may
	assume that all sets in this cover are indispensable). Now, for each $i=1,\dots,m_j$,
	define
	$$
	U_{j,i}=g_{j,i}(U_{j,1})\setminus\Bigl(\bigcup_{i'=1}^{i-1}g_{j,i'}(U_{j,1})\Bigr).
	$$
	It is clear that the sets $U_{j,i}$ ($i=1,\dots,m_j$) have the required properties (each
	of them is clopen, has disjoint images under $g\in F_{n_j}$, and jointly they cover $X$).
	\item Apply the following duplicating of shapes of $\T$: replace each symbol $``S\,"$
	($S\in\CS$) by two symbols $``S_\mathsf p"$ and $``S_\mathsf n"$ (the symbolic
	representation of $\T$ after duplicating will use an alphabet of cardinality
	$2r$). For each $j$ and $S=F_{n_j}$ place the symbols $``S_\mathsf p"$ at
	centers of all tiles with shape $S$ associated with the index $(j,1)$. Otherwise (for
	indices $(j,i)$, $i>1$) place the symbols $``S_\mathsf n"$.
\end{enumerate}
The rest of the construction is unchanged. What we have gained is captured in the lemma below.

\begin{lemma}\label{oo}
The dynamical \qt s $\T$ (we mean the version obtained via the above revision of the construction including the duplicating of shapes) and its disjoint version $\hat\T$ are \tl ly conjugate.
\end{lemma}

\begin{proof}
The revision enables one to recognize, for each $j=1,2,\dots,r$, which tiles with the shape $S=F_{n_j}$ are associated with the double index $(j,1)$. Call them \emph{primary tiles} (the subscripts $_\mathsf p$ and $_\mathsf n$ stand for ``primary'' and ``non-primary''). The association of the indices $(j,i)$ to non-primary tiles is also possible: if $c$ is the center of a non-primary tile of $\CT_x$ with shape $S=F_{n_j}$ then we examine all the elements $g_{j,2}^{-1}c,\ g_{j,3}^{-1}c,\ \dots,\ g^{-1}_{j,m_j}c$. The term $i$ in the double index $(j,i)$ associated with the considered tile $Sc$ can be determined as the smallest index $i$ for which $g^{-1}_{j,i}c$ is a center of a primary tile (we skip the elementary verification that this works). Once the indices $(j,i)$ are determined for all tiles (and thus the partial order among the tiles), the disjoint version $\hat\T$ is also determined: given $x\in X$, $\hat\CT_x$ is obtained by subtracting, from each tile of $\CT_x$, all its predecessors (we may also need to perform an adjustment of centers, as described in subsection \ref{spqt}). This is clearly a block code, so $\hat\T$ is a \tl\ factor of $\T$. In order for $\hat\T$ to be conjugate to $\T$ it suffices to apply to $\hat\T$ duplication of shapes, by which each tile of $\hat\T$ will ``remember'' the shape of the tile of $\T$ from which it was created. We omit more formal details of this easy step.
\end{proof}

\smallskip
Here the revision ends, but we continue establishing properties (independent of the above revision) of the \qt s provided by Theorem \ref{DH}. Let $K\subset G$ be a finite set. Easy examples show that there exists no pair $(F,\varepsilon)$, where $F\subset G$ is finite and $\varepsilon>0$, such that $(F,\varepsilon)$-invariance and $\varepsilon$-disjointness of a general \qt\ jointly guarantee $K$-separation of its set of centers. However, the \qt s provided by Theorem \ref{DH} are specific (due to the partial order among the tiles) and thus we can prove what follows (note that our revision of the construction from \cite{DH} is not employed):

\begin{lemma}\label{sep}
Let $K\subset G$ be a finite set and let $\varepsilon>\frac12$. If the shapes $F_{n_1},\dots,F_{n_r}$ of the $\varepsilon$-\qt s $\CT_x$ ($x\in X$) constructed in the proof of Theorem \ref{DH} (i.e., of \cite[Lemma~3.4]{DH}) are $(K^{-1}K,\varepsilon)$-invariant, then $C(\CT_x)$ is $K$-separated.
\end{lemma}

\begin{proof} Consider two tiles $T\neq T'$ of $\CT_x$ and denote by $(j,i)$ and $(j',i')$ the indices associated to $T$ and $T'$, respectively. If these tiles are disjoint then $|T\cap T'|=0$. If not, then one of them, say $T'=S'c'$, strictly precedes the other, say $T=Sc$, (i.e., $(j',i')$ precedes $(j,i)$). In this case $T\setminus T'$ is a $\frac12$-subset of $T$, i.e., $|T\cap T'|<\frac12|T|$. Moreover, since the F\o lner \sq\ $(F_n)_{n\in\N}$ is nested and $j'\ge j$, we also have $S\subset S'$. We can thus write
$$
|T\cap T'|=|Sc\cap S'c'|=|(S\cap S'c'c^{-1})^{-1}|=
|S^{-1}\cap c(c')^{-1}(S')^{-1}|\ge |S^{-1}\cap c(c')^{-1}S^{-1}|.
$$
Suppose that $Kc$ and $Kc'$ are not disjoint. Then $c(c')^{-1}\in K^{-1}K$, and since, by symmetry of the F\o lner \sq, $S^{-1}$ is $(K^{-1}K,\frac12)$-\inv, it is $(c(c')^{-1},1)$-invariant (see observation (1) above Definition \ref{tyi}), i.e., $|S^{-1}\cap c(c')^{-1}S^{-1}|>\frac12|S^{-1}|=\frac12|T|$. We have arrived at a contradiction.
\end{proof}

We are in a position to start the actual construction of a perfectly encodable F\o lner system of disjoint \qt s $\hat{\mathbf T}$. We fix a decreasing to zero \sq\ $(\epsilon_k)_{k\in\N}$ with $\epsilon_1<\frac12$. We will inductively construct a F\o lner system of (non-disjoint) \qt s $\mathbf T=\bigvee_{k\in\N}\T_k$ so that for each $k\in\N$, $\T_k$ is a dynamical $\epsilon_k$-\qt\ obtained via the above revised construction, with the collection of shapes $\CS_k\subset\{F_{n_{1,k}}, F_{n_{2,k}},\dots, F_{n_{r(\epsilon_k),k}}\}$, where $n_{1,k}<n_{2,k}<\dots<n_{r(\epsilon_k),k}<n_{1,{k+1}}$ and the dependence $\varepsilon\mapsto r(\varepsilon)$ is the same as in Theorem \ref{OW}. Due to duplication, each shape $S\in\CS_k$ will correspond to two symbols, $``S_\mathsf p"$ and $``S_\mathsf n"$. At the same time we will construct a decreasing \sq\ of subshifts $Z_k$ on three symbols\footnote{The number of symbols can be reduced to two, though not without some effort, see the Appendix.} together with a consistent \sq\ of \tl\ factor maps $\pi_k:Z_k\to\T_{[1,k]}=\bigvee_{i=1}^k\T_i$. The meaning of ``consistency'' is the same as in the proof of Theorem \ref{quasi}: $\pi_{k+1}$ composed with the natural projection $\pi_{[1,k]}:\T_{[1,k+1]}\to\T_{[1,k]}$ coincides with the restriction of $\pi_k$ to $Z_{k+1}$. The intersection $Z=\bigcap_{k\in\N} Z_k$ will be a symbolic extension of the entire system of \qt s $\mathbf T=\bigvee_{k\in\N}\T_k$. Later we will show that this extension is in fact isomorphic. This will prove perfect encodability of $\mathbf T$. By Lemma \ref{oo}, the disjoint version $\hat{\mathbf T}=\bigvee_{k\in\N}\hat\T_k$, being conjugate to $\mathbf T$, will also be perfectly encodable. The construction of $\mathbf T$ follows now.

\smallskip\noindent
{\bf Step 1}. We let $\T_1$ be the dynamical \qt\ whose only element is the tiling by singletons. This is an $\epsilon_1$-\qt\ (regardless of $\epsilon_1$) whose only shape is $F_1=\{e\}$ (i.e., $\CS_1=\bigl\{\{e\}\bigr\}$). We let $Z_1=\{-1,0,1\}^G$ (the full shift on three symbols). Clearly, $Z_1$ is a \tl\ extension of $\T_1$.

\smallskip\noindent
{\bf Step 2}. Define $m=\lceil\log_2(3\cdot2r(\epsilon_2))\rceil+1$. Fix a set $U_2\subset G$ of cardinality $m$ and containing the unity. Theorem \ref{DH} provides a zero entropy dynamical $\epsilon_2$-\qt\ $\T_2$ with the collection of shapes $\CS_2\subset\{F_{n_{1,2}},F_{n_{2,2}},\dots,F_{n_{r(\epsilon_2),2}}\}$, where $n_{1,2}<\cdots<n_{r(\epsilon_2),2}$. By Lemma \ref{sep}, choosing $n_{1,2}$ large enough, we can ensure that the set of centers $C(\CT_2)$ of every $\CT_2\in\T_2$ is $U_2$-separated. We use the revised version of Theorem \ref{DH}, and thus, for each tiling $\CT_2\in\T_2$ and each shape $S\in\CS(\T_2)$, we can determine the primariness of the tiles of $\CT_2$ with the shape $S$ (by observing the symbols $``S_\mathsf p"$ versus $``S_\mathsf n"$). Also note that since $X$ is minimal, so is $\T_2$.

The collection $\{-1,1\}^{U_2\setminus\{e\}}$ has cardinality $2^{m-1}\ge 3\cdot 2r(\epsilon_2)$. Thus, to every symbol $``S_\mathsf s"$, where $S\in\CS_2$ and $\mathsf s\in\{\mathsf p,\mathsf n\}$, one can disjointly associate a family of three different blocks $\{B^{(2)}_{S,\mathsf s,-1}, B^{(2)}_{S,\mathsf s,0}, B^{(2)}_{S,\mathsf s,1}\}$ from the above collection (the superscript $^{(2)}$ refers to the step of the construction).

We will now describe a rule of assigning to every pair $\CT_{[1,2]}=(\CT_1,\CT_2)\in\T_{[1,2]}$ (in fact, to every $\CT_2\in\T_2$, because $\CT_1$ is unique) an uncountable family denoted by $\pi_2^{-1}(\CT_{[1,2]})$ of symbolic elements $z\in Z_1$ which will constitute the preimage of $\CT_{[1,2]}$ in a symbolic extension of $\T_{[1,2]}$. Namely, given $\CT_2$ we allow $z\in Z_1$ to be a member of $\pi_2^{-1}(\CT_{[1,2]})$ if the following holds:
\begin{enumerate}
	\item If $\CT_{2,c}=``S_\mathsf s"$ (i.e., $c\in C(\CT_2)$ is the center of a primary or
	non-primary tile $Sc$ of $\CT_2$) then we require that $z_c=0$ and
	$z|_{U_2c\setminus\{c\}}=B^{(2)}_{S,\mathsf s,i}$, where $i\in\{-1,0,1\}$ (it is essential
	that the sets $U_2c$ are disjoint for different $c\in C(\CT_2)$).
	\item All independent choices of the above indices $i$ for different centers $c\in
	C(\CT_2)$ are represented in the elements $z\in \pi_2^{-1}(\CT_{[1,2]})$.
	\item We define the \emph{background of $\CT_2$} as the complement of $U_2C(\CT_2)$,
	and we require that $z_g=1$ for every $z\in \pi_2^{-1}(\CT_{[1,2]})$ and all $g$ in this
	background.
\end{enumerate}
We define $Z_2=\bigcup_{\CT_{[1,2]}\in\T_1\times\T_2}\pi_2^{-1}(\CT_{[1,2]})$. It should be obvious that $Z_2$ is closed and shift-\inv. The factor map $\pi_2$ functions as follows: given $z\in Z_2$ we look for the positions of the symbols $0$ in $z$. These are exactly the centers of the tiles of such \qt\ $\CT_2$ that $(\CT_1,\CT_2)=\pi_2(z)$. For every center $c$ the block $z|_{U_2c\setminus\{c\}}$ has the form $B^{(2)}_{S,\mathsf s,i}$, where $S\in\CS_2$, $\mathsf s\in\{\mathsf p,\mathsf n\}$ and $i\in\{-1,0,1\}$.\footnote{The ``trit'' (analog of ``bit'' but with three values) of information carried by the index $i\in\{-1,0,1\}$ is, at this step, superfluous, but will be essentially used in the following steps.} We can now determine that $S$ is the shape of the tile of $\CT_2$, centered at $c$, while $\mathsf s$ tells us whether the tile is primary or not. We have deduced that in the symbolic representation of $\CT_2$, $\CT_{2,c}=``S_\mathsf s"$. In this manner $z$ allows to fully reconstruct $\CT_2$ (with the duplicated alphabet) using a block code with coding horizon $U_2$. It is clear that the set denoted by $\pi_2^{-1}(\CT_{[1,2]})$ is indeed the preimage of $\CT_{[1,2]}$ by the above mapping $\pi_2$.

\smallskip\noindent
{\bf Step $k+1$}. Given $k\ge 2$ suppose that for each $2\le l\le k$ we have selected a minimal dynamical $\epsilon_l$-\qt\ $\T_l$ with the collection of shapes $\CS_l\subset \{F_{n_{1,l}},F_{n_{2,l}},\dots,F_{n_{r(\epsilon_l),l}}\}$, where $n_{r(\epsilon_{l-1}),\,l-1}<n_{1,\,l}<n_{2,\,l}<\cdots<n_{r(\epsilon_l),\,l}$, represented as a subshift over the duplicated alphabet (of cardinality at most $2r(\epsilon_l)$) allowing to differentiate between primary and non-primary tiles. We also assume that we have constructed a subshift $Z_k$ on three symbols, and a \tl\ factor map $\pi_k:Z_k\to \T_{[1,k]}$. We assume that there exist finite sets $U_k\subset V_k$ with $\frac{|U_k|}{|V_k|}\le\frac1{k-1}$ such that for each $\CT_k\in\T_k$ the set of centers $C(\CT_k)$ is $V_k$-separated and the factor map $\pi_k$ is given by a block code with coding horizon $U_k$ (at step 2 we have taken $V_2=U_2$). Moreover, we require certain structure of the fibers (preimages of points) of $\pi_k$, captured in the conditions (1)-(3) below. Given a $k$-tuple  $\CT_{[1,k]}=(\CT_1,\CT_2,\cdots,\CT_k)\in\T_{[1,k]}$ and  $c\in C(\CT_k)$ consider the restriction $\CT_{[1,k]}|_{U_kc}$. Since each $\CT_l$ is symbolic ($l\le k$), $\CT_{[1,k]}$ is also symbolic and this restriction is in fact a (shifted) block on finitely many symbols over the domain $U_k$. Let $\mathcal D_k$ denote the (finite) family of all such blocks
$$
\mathcal D_k=\{\CT_{[1,k]}|_{U_kc}: \ \CT_{[1,k]}\in\T_{[1,k]}, \ c\in C(\CT_k)\}.
$$

\begin{enumerate}
	\item For every $D\in\mathcal D_k$ there are exactly three different blocks
	$B^{(k)}_{D,-1},B^{(k)}_{D,0}$ and $B^{(k)}_{D,1}$ belonging to $\{-1,0,1\}^{U_k}$ such
	that whenever $D=\CT_{[1,k]}|_{U_kc}$ for some
	$\CT_{[1,k]}\in\T_{[1,k]}$ and $c\in C(\CT_k)$, and
	$z\in\pi_k^{-1}(\CT_{[1,k]})$ then $z|_{U_kc}=B^{(k)}_{D,i}$ for some $i\in\{-1,0,1\}$.
	\item For any fixed $\CT_{[1,k]}$, all independent choices of the above indices $i$
	for different centers $c\in C(\CT_k)$ are represented in the elements
	$z\in\pi_k^{-1}(\CT_{[1,k]})$ (it is essential that the sets $U_kc$ are pairwise disjoint).
	\item The restrictions of all elements $z\in\pi_k^{-1}(\CT_{[1,k]})$ to the complement of
	the set $U_kC(\CT_k)$ (called the background of $\CT_k$) are equal.
\end{enumerate}

We now need to construct $\T_{k+1}$, $Z_{k+1}$ and define $\pi_{k+1}$. Since $\T_k$ is minimal, it is transitive, say $\T_k=\bar O(\CT^\bullet_k)$. By Propositon \ref{syn}, there exists a finite set $U$ containing $e$, such that the set $C(\CT^\bullet_k)$ of all centers of $\CT^\bullet_k$ is $U^{-1}$-syndetic. Since $U^{-1}$-syndeticity is clearly an invariant and closed property, the same holds for each $\CT_k\in\T_k$, that is to say, in every shifted set $U\!g$ there exists at least one center $c$ of some tile $T$ of $\CT_k$. We define $m=\lceil\log_2 (3\cdot|U|2r(\epsilon_{k+1}))\rceil\!+\!1$. Further, there exists a (much larger) finite set $\hat U\supset U$ such that for each $\CT_k\in\T_k$, in every shifted copy $\hat Ug$ there are at least $m$ centers of $\CT_k$ (it suffices that $\hat U$ contains $m$ disjoint shifted copies of $U$). We define $U_{k+1}$ as $U_k\hat UU$. We also choose a finite set $V_{k+1}\supset U_{k+1}$ with $\frac{|U_{k+1}|}{|V_{k+1}|}\le\frac1k$.

The revised version of Theorem \ref{DH} combined with Lemma \ref{sep} provides a zero entropy minimal dynamical $\epsilon_{k+1}$-\qt\ $\T_{k+1}$ with at most $r(\epsilon_{k+1})$ shapes belonging to the F\o lner \sq: $\CS_{k+1}\subset \{F_{n_{1,k+1}},F_{n_{2,k+1}},\dots,F_{n_{r(\epsilon_{k+1}),k+1}}\}$, where $n_{r(\epsilon_k),\,k}<n_{1,\,k+1}<n_{2,\,k+1}<\cdots<n_{r(\epsilon_{k+1}),\,k+1}$, and such that for every $\CT_{k+1}\in\T_{k+1}$ the set of centers $C(\CT_{k+1})$ is $V_{k+1}$-separated. The \qt\ is represented as a subshift over the duplicated alphabet $\bigl\{``S_\mathsf s":S\in\CS_{k+1},\,\mathsf s\in\{\mathsf p,\mathsf n\}\bigr\}$, allowing to determine the primariness of the tiles.

There are at most $|U|2r(\epsilon_{k+1})$ triples $(u,S,\mathsf s)$, where $u\in U$, $S\in\CS_{k+1}$ and $\mathsf s\in\{\mathsf p,\mathsf n\}$, while there are at least $2^{m-1}\ge3\cdot|U|2r(\epsilon_{k+1})$ words of length $m$, over the alphabet $\{-1,0,1\}$ (i.e., functions from $\{1,2,\dots,m\}\to\{-1,0,1\}$), in which $0$ occurs exactly once, at the first position. Thus, to every triple $(u,S,\mathsf s)$ one can disjointly associate a family $\{W_{u,S,\mathsf s,-1}, W_{u,S,\mathsf s,0}, W_{u,S,\mathsf s,1}\}$ of three different such words.

For each $(k\!+\!1)$-tuple $\CT_{[1,k+1]}=(\CT_1,\CT_2,\dots,\CT_k,\CT_{k+1})\in \T_{[1,k+1]}$ we will now select a subset of $\pi_k^{-1}(\CT_{[1,k]})$ where $\CT_{[1,k]}=(\CT_1,\CT_2,\dots,\CT_k)$, which will constitute the preimage $\pi_{k+1}^{-1}(\CT_{[1,k+1]})$. Recall that all elements $z\in\pi_k^{-1}(\CT_{[1,k]})$ are equal on the background of $\CT_k$, while on every set $U_kc$ ($c\in C(\CT_k)$) there occur three possible blocks $B^{(k)}_{D,i}$ ($i\in\{-1,0,1\}$), where $D=\CT_{[1,k]}|_{U_kc}$. We will soon restrict these possibilities in a way that depends on $\CT_{k+1}$.

We enumerate the set $\hat U$ as $\{g_1,g_2,\dots,g_{|\hat U|}\}$ starting with the elements of $U$, i.e., so that $U=\{g_1,g_2,\dots,g_{|U|}\}$. Let $c_0\in C(\CT_{k+1})$, i.e., for some $S\in\CS_{k+1}$, $Sc_0$ is a tile of $\CT_{k+1}$. In $U\!c_0$ there is at least one center of $\CT_k$. We let $c_1$ be the first one in the enumeration of $U\!c_0$ as $\{g_1c_0,g_2c_0,\dots,g_{|U|}c_0\}$. We denote by $u$ the element $c_1c_0^{-1}\in U$. Next, in $\hat Uc_1$ there are at least $m$ centers of $\CT_k$. After excluding $c_1$, we have at least $m-1$ such centers. We let $c_2,c_3,\dots,c_m$ be the first $m-1$ of them in the ordering of $\hat Uc_1$ as $\{g_1c_1,g_2c_1,\dots,g_{|\hat U|}c_1\}$.

Within $z|_{U_{k+1}c_0}$ we will encode the information about $u$ (which represents the ``distance'' between $c_0$ and $c_1$), the shape $S$, the subscript $\mathsf s$ according to which the tile $Sc_0$ is primary or not, plus one extra trit of information for future use. This will be achieved by encoding (within $z|_{U_{k+1}c_0}$) one of the three words $\{W_{u,S,\mathsf s,-1}, W_{u,S,\mathsf s,0}, W_{u,S,\mathsf s,1}\}$. To this end, we simply require that the indices $i$ in the blocks $B^{(k)}_{D_j,i}$, where $D_j=\CT_{[1,k]}|_{U_kc_j}$  ($j=1,\dots,m$) follow one of the words $W_{u,S,\mathsf s,-1}$ or $W_{u,S,\mathsf s,0}$, or $W_{u,S,\mathsf s,1}$. Formally, we require that:
$$
\exists_{i'\in\{-1,0,1\}}\ \forall_{j=1,\dots,m}\ \ z|_{U_kc_j}=B^{(k)}_{D_j,W_{u,S,\mathsf s,i'}(j)}.
$$
Roughly speaking, on the set $\bigcup_{j=1}^m U_kc_j$ we have reduced the number of possibilities from $3^m$ (represented by all possible configurations of the indices $i$) to just $3$ (represented by the new index $i'$). Since each $c_j$ belongs to $\hat Uc_1\subset \hat UUc_0$, the above restrictions affect $z$ only on the set $U_k\hat UUc_0=U_{k+1}c_0$. As the set $C(\CT_{k+1})$ is $U_{k+1}$-separated, there is no collision between the above restrictions introduced for different centers $c_0\in C(\CT_{k+1})$. For fixed $\CT_{[1,k+1]}$ we allow all independent choices of the indices $i'$ for different centers $c_0\in C(\CT_{k+1})$ to be represented in the elements $z\in\pi_{k+1}^{-1}(\CT_{[1,k+1]})$.

Additionally, we introduce two ``background rules'':
\begin{enumerate}
	\item \emph{The ``small background''}: if $c$ is a center of $\CT_k$ within $U_{k+1}c_0$
	other than any $c_j$ ($j=1,\dots,m$), then we require that for all
	$z\in\pi^{-1}_{k+1}(\CT_{[1,k+1]})$, $x|_{U_kc}=B^{(k)}_{D,1}$ (where
	$D=\CT_{[1,k]}|_{U_kc}$).
	With this rule, the block $z|_{U_{k+1}c_0}$ may assume one of only three possible forms
	(corresponding to the new index $i'$). The collection of these three blocks depends only
	on the restriction $D'=\CT_{[1,k+1]}|_{U_{k+1}c_0}$, hence we can denote these three blocks
	as $B^{(k+1)}_{D',-1}$, $B^{(k+1)}_{D',0}$ and $B^{(k+1)}_{D',1}$.
	\item \emph{The ``large background''}: If $c$ is a center of $\CT_k$ outside
	$U_{k+1}C(\CT_{k+1})$, we also require that for all $z\in\pi^{-1}_{k+1}(\CT_{[1,k+1]})$,
	$z|_{U_kc}=B^{(k)}_{D,1}$, where $D=\CT_{[1,k]}|_{U_kc}$.
\end{enumerate}
%We remark, that since the set $C(\CT_k)$ is $U_k$-separated, the the background rules lead to no collisions, neither.
This concludes the definition of $\pi^{-1}_{k+1}(\CT_{[1,k+1]})$. We let
$$
Z_{k+1}=\bigcup\{\pi^{-1}_{k+1}(\CT_{[1,k+1]}):\  \CT_{[1,k+1]}\in\T_{[1,k+1]}\}.
$$
Clearly, by construction, $Z_{k+1}\subset Z_k$. We skip the elementary verification that $Z_{k+1}$ is closed and shift-\inv.
\smallskip

We will now describe the functioning of the code $\pi_{k+1}$. Let $z\in Z_{k+1}$. Clearly, $z\in Z_k$ and by the inductive assumption, we can determine the image $\CT_{[1,k]}=\pi_k(z)$ by a block code with the coding horizon $U_k$. The $k$-tuple $\CT_{[1,k]}$ will play the role of the projection of the desired image $\CT_{[1,k+1]}$ onto the first $k$ coordinates, and it only remains to determine $\CT_{k+1}$ \emph{given} $\CT_{[1,k]}$. This will automatically guarantee consistency of $\pi_{k+1}$ with the preceding maps $\pi_l$ ($l\le k$). In particular, we can locate all centers $c\in C(\CT_k)$, and, for every such center we can determine the block $D=\CT_{[1,k]}|_{U_kc}$. Next, for every such pair $c$ and $D$ we check whether $z|_{U_kc}=B^{(k)}_{D,0}$. If yes, then we denote $c$ by $c_1$ and we know that the center $c_0$ of a tile of $\CT_{k+1}$ lies within $U^{-1}c_1$, say $c_0=u^{-1}c_1$. We need to determine three pieces of data: $u$, the shape $S\in\CS_{k+1}$ of the tile of $\CT_{k+1}$ centered at $c_0$, and its primariness. In $\hat Uc_1$ we can easily locate the first $m-1$ (other than $c_1$) centers of $\CT_k$ in the ordering of $\hat Uc_1$ as $\{g_1c_1,g_2c_1,\dots,g_{|\hat U|}c_1\}$, and call them $c_2,c_3,\dots,c_m$. By the rules of creating $Z_{k+1}$, the blocks $z|_{U_kc_j}$ will have only two forms, either $B^{(k)}_{D_j,-1}$ or $B^{(k)}_{D_j,1}$, where $D_j = \CT_{[1,k]}|_{U_kc_j}$. The indices $-1,1$, together with the initial 0, will form a word $W\in\{-1,0,1\}^{\{1,2,\dots,m\}}$ equal to one of the words $W_{u,S,\mathsf s,i'}$ for a unique combination of parameters $u\in U$, $S\in\CS(\CT_{k+1})$, $\mathsf s\in\{\mathsf p,\mathsf n\}$, $i'\in\{-1,0,1\}$. Now we can determine $c_0$ as $u^{-1}c_1$ and we know that $\CT_{k+1}$ has a tile (primary or not, according to the value of $\mathsf s$) centered at $c_0$ with the shape $S$, i.e., that $\CT_{k+1,c_0}=``S_\mathsf s"$. In this manner, we have recognized the tile and its primariness by viewing the set $U_{k+1}=U_k\hat UU$ shifted to the center of this tile. The trit of information carried by the index $i'$ is, at this step, superfluous, but clearly crucial in further steps.

\medskip
Once the induction is completed, we define $Z$ as the decreasing intersection of the subshifts $Z_k$ ($k\in\N$). It is clear that $Z$ is a symbolic extension of the countable joining $\mathbf T=\bigvee_{k\in\N}\T_k$. The factor map $\pi:Z\to\mathbf T$
is defined as the limit of the blocks codes $\pi_k$: by consistency this limit exists (each code $\pi_k$ allows to determine the first $k$ layers of the image by $\pi$).

\medskip
We shall now argue that $Z$ is an isomorphic extension of $\mathbf T$ by showing that the factor map $\pi$ is injective except on a set of universal measure zero (i.e., of measure zero for all \im s on $Z$). It suffices to show that $\pi^{-1}(\boldsymbol\CT)$ ($\boldsymbol\CT=(\CT_1,\CT_2,\dots)\in\mathbf T$) is a singleton except when $\boldsymbol\CT$ belongs to some set of universal measure zero on $\mathbf T$. A way to prove it is by showing that the set of all elements $\boldsymbol\CT\in\mathbf T$ which have multiple preimages by $\pi$, i.e., the set
$$
\mathsf A=\{\boldsymbol\CT:|\pi^{-1}(\boldsymbol\CT)|>1\}
$$
(which is clearly Borel-measurable in $\mathbf T$) has universal \im\ zero. An element $\boldsymbol\CT\in\mathbf T$ is in $\mathsf A$ if there exists $g\in G$ and two elements $z,z'\in\pi^{-1}(\boldsymbol\CT)$ with $z_g\neq z'_g$. Thus $\mathsf A=\bigcup_{g\in G}\mathsf A_g$, where $\mathsf A_g=\{\boldsymbol\CT:\exists_{z,z'\in\pi^{-1}(\boldsymbol\CT)}\ z_g\neq z'_g\}$. It suffices to prove that for every $g\in G$, $\mathsf A_g$ has universal measure zero. Because $\mathsf A_g=g(\mathsf A_e)$, we can consider only $g=e$. Let $\boldsymbol\CT\in\mathsf A_e$ and let $z,z'$ be as in the definition of $\mathsf A_e$.
For each $k$ we then have $z,z'\in\pi_k^{-1}(\CT_{[1,k]})$, where $\CT_{[1,k]}$ is the projection of $\boldsymbol\CT$ on the first $k$ coordinates. Thus, $z_e\neq z'_e$ is possible only when $e$ does not belong to the background of $\CT_k$, i.e., when $e\in U_kC(\CT_k)$. Thus, for each $k$,
$$
\mathsf A_e\subset\mathsf B=\{\boldsymbol\CT=(\CT_1,\CT_2,\dots):e\in U_kC(\CT_k)\}.
$$
Given $\boldsymbol\CT\in\mathbf T$, we have, in the notation of Proposition \ref{prop}
(with $x$ replaced by $\boldsymbol\CT$), the following equality:
$$
B_{\boldsymbol\CT}=\{g: g\boldsymbol\CT\in\mathsf B\}= U_kC(\CT_k).
$$
By left invariance and subadditivity of upper Banach density in amenable groups (see subsections \ref{bd}, \ref{bdr}), for each $k\ge 2$ we have
$$
\overline D(B_{\boldsymbol\CT})= \overline D(U_kC(\CT_k))\le|U_k|\overline D(C(\CT_k))\le|U_k| \frac 1{|V_k|}\le\frac1{k-1}.
$$
We have proved that for any $\boldsymbol\CT\in\mathbf T$, $\overline D(B_{\boldsymbol\CT})\le\frac1{k-1}$ for every $k\ge 2$, i.e., that $\overline D(B_{\boldsymbol\CT})=0$. By Proposition \ref{prop}, we obtain
$$
\sup_{\mu\in\M_G(\mathbf T)}\mu(\mathsf A_e)\le \sup_{\mu\in\M_G(\mathbf T)}\mu(\mathsf B)=0.
$$
This ends the proof that the extension $\pi:Z\to\mathbf T$ is isomorphic, so $\mathbf T$ is perfectly encodable. Let us remind once more, that by Lemma \ref{oo}, the associated system of \emph{disjoint} \qt s $\hat{\mathbf T}$ is conjugate to $\mathbf T$ and thus it is also perfectly encodable. The proof of Theorem \ref{nine} is now complete.
\end{proof}

\subsection{Encodability of tiling systems versus the comparison property}\label{piec}

Encodable F\o lner systems of disjoint \qt s is just a step in our pursuit towards facing the true challenge which is the creation of encodable tiling systems. Only such systems will allow us to built genuine symbolic extensions. As the theorem below shows, the comparison property is crucial in this aspect.

\begin{comment}
The fact that the dynamical \qt\ $x\mapsto\CT_x$ is a \tl\ factor of the action of $G$ on $X$ is equivalent to the conjuction of the following two statements:
\begin{enumerate}
	\item for any finite set $F$ of $G$, if $x$ and $x'$ are sufficiently close to each other in $X$, then the set $F$ is tiled by $\CT_x$ and by $\CT_{x'}$ in the same way,
	\item for each $g\in G$ we have $\CT_{g(x)}=\{Tg^{-1}:T\in\CT_x\}$.
\end{enumerate}
\end{comment}

\begin{theorem}\label{ddd}
Let a countable amenable group $G$ act on a zero-dimensional compact metric space $X$. Suppose the action has a F\o lner system of disjoint \qt s $\hat{\mathbf T}$ as a \tl\ factor and admits comparison. Then $X$ has a tiling system as a \tl\ factor.
\end{theorem}

Before the proof, let us draw some corollaries. The first one is not very useful for us, but perhaps has an interest of its own. The second one is absolutely crucial for the rest of this paper. Recall that in large parts of Section \ref{s4} we have been dealing with \zd\ systems which had a system of tilings as a \tl\ factor. They were artificially created by joining an arbitrary \zd\ system with a tiling system. We can now characterize, in terms of comparison, these \emph{free} \zd\ systems which have a system of tilings as a \tl\ factor, without needing to join them with anything.

\begin{cor}\label{dt}
Let a countable amenable group $G$ act freely on a compact metric \zd\ space $X$. Then the action has a tiling system as a \tl\ factor if and only if it admits comparison. The forward implication holds without assuming that the action is free.
\end{cor}

\begin{proof} Because the action is free, by Theorem \ref{DH1}, it has a F\o lner system of disjoint \qt s $\hat{\mathbf T}$ as a \tl\ factor. If the action admits comparison, Theorem \ref{ddd} applies which ends the proof of backward implication. We save the somewhat lengthy proof of the forward implication for later.
\end{proof}

\begin{cor}\label{ddt}
Suppose $G$ has the comparison property. Then there exists an encodable zero entropy tiling system $\tilde{\mathbf T}$ of $G$.
\end{cor}

\begin{proof}
By Theorem \ref{nine}, there exists a (perfectly) encodable F\o lner system of disjoint \qt s $\hat{\mathbf T}$ of \tl\ entropy zero (for that, the comparison property is not used yet). By the comparison property, the action on $\hat{\mathbf T}$ admits comparison. Now, Theorem \ref{ddd} implies that there exists a tiling system $\tilde{\mathbf T}$ which is a \tl\ factor of $\hat{\mathbf T}$. Clearly, the isomorphic symbolic extension of $\hat{\mathbf T}$ (which exists by perfect encodability) is also a principal symbolic extension of $\tilde{\mathbf T}$ and thus $\tilde{\mathbf T}$ is encodable.
\end{proof}

\begin{proof}[Proof of Theorem \ref{ddd}] Firstly, assuming comparison, we will show that $X$ has, as a \tl\ factor, a F\o lner system of tilings $\check{\mathbf T}$ (which is not necessarily congruent let alone deterministic; we will take care of ensuring these properties later). In fact, we will show that if $\hat{\mathbf T}=\bigvee_{k\in\N} \hat\T_k$ then, for arbitrarily large indices $k\in\N$, the \qt s $\hat\T_k$ extend to some dynamical tilings $\check\T_k$ (still factors of $X$), which have only slightly worse invariance properties. This will imply that $X$ has the joining $\check{\mathbf T}=\bigvee_{k\in\N}\check\T_k$ (where $k$ ranges over the respective sub\sq) as a \tl\ factor and that $\check{\mathbf T}$ is a F\o lner system of tilings.

Fix a finite set $K\subset G$ and $\varepsilon>0$. Let $\delta> 0$ be so small that
$$
\frac{2\delta}{1-\delta}<\frac\eps{2|K|}.
$$
For some $k\in\N$ the dynamical \qt\ $\hat\T_k$ is $(K,\frac\eps2)$-invariant, disjoint and $(1\!-\!\delta)$-covering.
We denote by $\hat\CS_k$ the collection of all shapes used by this \qt. By choosing $k$ large enough, we can also assume that each shape $S\in\hat\CS_k$ has cardinality so large that the interval $(\frac{2\delta}{1-\delta}|S|,\frac\eps{2|K|}|S|)$ contains an integer $i_{\!S}$. In each shape $S\in\hat\CS_k$ we select (arbitrarily) a subset $B_S$ of cardinality $i_{\!S}$. Since $k$ is fixed from now on, we will skip it in the denotation of $\hat\T_k$, $\hat \CS_k$ and, above all, $\hat\CT_k$, giving room for the subscript $x$. Given $x\in X$ we let $\hat\CT_x\in\hat\T$ be the \qt\ corresponding to $x$ via the factor map from $X$ to $\hat{\mathbf T}$.  We now observe two subsets of $G$:
$$
A_x= G\setminus \bigcup\hat\CT_x\text{ \ \ and  \ \ }B_x = \bigcup_{Sc\in\hat\CT_x}B_Sc.
$$
Clearly, $\overline D(A_x)=1-\underline D(\bigcup\hat\CT_x)\le\delta$. Using Lemma~\ref{1.5} we easily get $\underline D(B_x)>(1-\delta)\cdot \frac{2\delta}{1-\delta} =2\delta$. By Corollary \ref{coro}, $\underline D(B_x,A_x)>\delta$. Define two subsets of $X$:
$$
\mathsf A=\{x\in X: e\in A_x\}\text{ \ \ and \ \ }\mathsf B=\{x\in X: e\in B_x\}.
$$
Since one can determine whether $e\in A_x$ (and likewise, whether $e\in B_x$) from the symbolic representation of $\hat\CT_x$ by viewing the symbols in the bounded horizon $\bigcup_{S\in\hat\CS}S^{- 1}$ around $e$, both sets $\mathsf A$ and $\mathsf B$ are clopen (and obviously disjoint) in $X$. The notation $A_x,\,B_x$ is now consistent with \eqref{kiki} and \eqref{kiko} for the sets $\mathsf A,\,\mathsf B$, respectively, hence, by Proposition \ref{prop} (1) (the last equality) we obtain $\underline D(\mathsf B,\mathsf A)\ge\delta>0$. The fact that the action on $X$ admits comparison implies that $\mathsf A\preccurlyeq\mathsf B$.

%\begin{comment}
Since we prefer to work with a symbolic system in place of the \zd\ system $X$, we will now build a symbolic factor $\hat X$ of $X$ carrying the minimum information needed to restore both the dependence $x\mapsto\hat\CT_x$ and the subequivalence $\mathsf A\preccurlyeq\mathsf B$. Let $\{\mathsf A_1,\mathsf A_2,\dots,\mathsf A_l\}$ and $g_1,g_2,\dots, g_l$ be the clopen partition of $\mathsf A$ and the associated elements of $G$ as in the definition of subequivalence, respectively. We define a factor map $\hat\pi:X\to\hat X\subset{\hat{\rm V}}^G$, where
$\hat{\rm V}={\rm V}\times\{0,1,\dots,l,l+1\}$, and ${\rm V}=\{``S\,":S\in\hat\CS\}\cup\{0\}$ is the alphabet of the symbolic representation of $\hat\T$, as follows:
$$
(\hat\pi(x))_g =
\begin{cases}
((\hat\CT_x)_g,i)&\text{ \ \ if \ }g(x)\in \mathsf A_i, \ \ i=1,2,\dots,l\\
((\hat\CT_x)_g,l+1)&\text{ \ \ if \ }g(x)\in\mathsf B\\
((\hat\CT_x)_g,0)&\text{ \ \ if \ }g(x)\notin\mathsf A\cup\mathsf B.\\
\end{cases}
$$
For $i=1,2,\dots,l$ denote $\hat{\mathsf A}_i=[\cdot,i]$, $\hat{\mathsf A}=\bigcup_{i=1}^l[\cdot,i]$ and $\hat{\mathsf B}=[\cdot,l+1]$. Clearly, $\hat\pi^{-1}(\hat{\mathsf A})=\mathsf A$, $\hat\pi^{-1}(\hat{\mathsf A}_i)=\mathsf A_i$ ($i=1,2,\dots,l$) and $\hat\pi^{-1}(\hat{\mathsf B})=\mathsf B$, which easily implies that $\hat{\mathsf A}\preccurlyeq\hat{\mathsf B}$ in the subshift $\hat X$, and the subequivalence involves the same elements $g_1,g_2,\dots, g_l$. Also for any $\hat x\in \hat X$ all \qt s $\hat\CT_x$ with $x\in\hat\pi^{-1}(\hat x)$ coincide, so we can denote them by $\hat\CT_{\hat x}$. In this manner the \qt\ $\hat\T$ turns out to be a \tl\ factor of the subshift $\hat X$. Moreover, whenever $x\in\hat\pi^{-1}(\hat x)$, we have $\hat A_{\hat x}=A_x,\ \hat B_{\hat x}=B_x$. From now on, we will skip the hats in the denotation of $\hat X$ (and of $\hat x\in\hat X$) remembering that we have replaced the zero-dimensional system $X$ by a subshift.
%\end{comment}

By Theorem \ref{tutka} (1), there exists a family of injections $\tilde\varphi_x:A_x\to B_x$ indexed by $x\in X$ determined by a block code. We are in a position to create, basing on the \qt s $\hat\CT_x$, the desired tilings $\check\CT_x$. Given $x\in X$, we define a transformation of the tiles $Sc\in\hat\CT_x$ as follows:
$$
\Phi_x(Sc)=Sc\cup\tilde\varphi_x^{-1}(B_Sc)\subset Sc\cup A_x
$$
(recall that $B_Sc$ is a part of the set $B_x$, so its preimage by $\tilde\varphi_x$ is a part of $A_x$). We will call the set $\tilde\varphi_x^{-1}(B_Sc)$ the \emph{added set}.
We define the center of the new tile $\Phi_x(Sc)$ as $c$. The shape of the new tile equals
$$
\Phi_x(Sc)c^{-1}=S\cup\tilde\varphi_x^{-1}(B_Sc)c^{-1}.
$$
Note that
$$
\tilde\varphi_x^{-1}(B_Sc)c^{-1}\subset E^{-1} (B_Sc)c^{-1}\subset E^{-1}S,
$$
which is a finite set (here $E$ is the finite set of multipliers used by $\tilde\varphi_x$, common for all $x\in X$). Since $\hat\CS$ is finite, the set $\check\CS$ of all new shapes is also finite. As the \qt\ $\hat\CT_x$ is disjoint, $\tilde\varphi_x$ restricted to $A_x$ is injective, and the image of $A_x$ is contained in $B_x=\bigcup_{Sc\in\hat\CT_x}B_Sc$, it is clear that the new \qt\
$$
\check\CT_x=\{\Phi_x(Sc):Sc\in\hat\CT_x\}
$$
is a tiling (disjoint and covering $G$ completely).

Further, for any tile $Sc$ of $\hat\CT_x$ the added set $\tilde\varphi_x^{-1}(B_Sc)$
has cardinality at most $|B_S|=i_S<\frac{\eps}{2|K|}|S|$. Thus
$$
|K\Phi_x(Sc)|\le |KSc| + |K|\cdot \frac{\eps}{2|K|}|S| = |KS|+\frac\eps2|S|.
$$
We can assume (at the beginning of the proof) that $e\in K$, and then $(K,\frac\eps2)$-invariance of $S$ is equivalent to the inequality $|KS|<(1+\frac\eps2)|S|$. Thus
$$
|K\Phi_x(Sc)|< (1+\eps)|S|\le (1+\eps)|\Phi_x(Sc)|,
$$
and so $\Phi_x(Sc)$ is $(K,\eps)$-invariant. Summarizing, we have constructed a mapping $(x,\hat\CT_x)\mapsto\check\CT_x$ from $X$ into tilings with a finite set $\check\CS$ of $(K,\eps)$-invariant shapes.

\smallskip
We need to show that the above is a \tl\ factor map.
To do so, we can use the criterion \eqref{clr}, i.e., we need to indicate a finite set $J\subset G$, such that for any $x,x'\in X$ and $g\in G$,
\begin{equation}\label{equ}
\hat\CT_x|_{Jg}=\hat\CT_{x'}|_{Jg} \implies (\check\CT_x)_g=(\check\CT_{x'})_g.
\end{equation}

We claim that the set $J=\{e\}\cup FE^{-1}R$ is good, where $F$ is the finite coding horizon of $\tilde\varphi_x$ (common for all $x\in X$), $E$ is the common set of multiplies, and $R=\bigcup\hat\CS$. In order to verify this claim, assume that with so defined $J$ the left hand side of \eqref{equ} holds for some $x,x'\in X$ and $g\in G$.
Since $g\in Jg$, we have $(\hat\CT_x)_g=(\hat\CT_{x'})_g$. If this common entry is $0$ then $g$ is not a center of any tile in neither $\hat\CT_x$ nor $\hat\CT_{x'}$, and then $g$ is not a center of any tile in neither $\check\CT_x$ nor $\check\CT_{x'}$, i.e., $(\check\CT_x)_g=(\check\CT_{x'})_g=0$. If the common entry is some $``S\,"$ with $S\in\hat\CS$ then we know that $g=c$ is a center of some tile in both $\check\CT_x$ and $\check\CT_{x'}$, their shapes have the same common part $S$ and may differ only in having different added sets. The added sets equal $\tilde\varphi_x^{-1}(B_Sc)c^{-1}$ and $\tilde\varphi_{x'}^{-1}(B_Sc)c^{-1}$, respectively. We need to show that
$$
\tilde\varphi_x^{-1}(B_Sc)c^{-1}=\tilde\varphi_{x'}^{-1}(B_Sc)c^{-1}.
$$
Since $FE^{-1}Rc=FE^{-1}Rg\subset Jg$, the left hand side of \eqref{equ} implies $\hat\CT_x|_{FE^{-1}Rc}=\hat\CT_{x'}|_{FE^{-1}Rc}$. Recall that the family $\{\tilde\varphi_x\}_{x\in X}$ is determined by a block code with coding horizon $F$. We deduce that $\tilde\varphi_x$ agrees with $\tilde\varphi_{x'}$ on the set $E^{-1}Rc$, which contains $E^{-1}Sc$, which contains $E^{-1}B_Sc$. But $E^{-1}B_Sc$ contains the union
$\tilde\varphi_x^{-1}(B_Sc)\cup \tilde\varphi_{x'}^{-1}(B_Sc)$. Since, as we have shown,
$\tilde\varphi_x$ and $\tilde\varphi_{x'}$ agree on this union, we conclude that $\tilde\varphi_x^{-1}(B_Sc)=\tilde\varphi_{x'}^{-1}(B_Sc)$. Thus we have shown that the tiling $\check\T$ is a \tl\ factor of $X$.

Restoring the indices $k$, and applying the above to all \qt s $\T_k$ ($k\in\N$), we create the desired F\o lner system of tilings $\check{\mathbf T}$ as a \tl\ factor of $X$. The next step in the proof is turning this system into a congruent and deterministic one, i.e., into a tiling system $\tilde{\mathbf T}$. Only congruency is essential, because determinism can be easily achieved later using Remark~\ref{remark} (by duplicating the shapes). Now, passing from the F\o lner system of tilings $\check{\mathbf T}$ to a congruent one is described in the proof of \cite[Lemma~5.1]{DHZ}, and here we only briefly sketch the method. Recall that $\check{\mathbf T}=\bigvee_{k\in\N}\check\T_k$. We let $\tilde\T_1=\phi_1(\check\T_1)=\check\T_1$ and then, in an inductive procedure, once the modification map $\phi_{[1,k]}:\check\T_{[1,k]}\to\tilde\T_{[1,k]}$ is constructed (where the image is already congruent), we extend this map to $\phi_{[1,k+1]}$ as follows: given $\check\CT_{[1,k+1]}=(\check\CT_1,\check\CT_2,\dots,\check\CT_k,\check\CT_{k+1})\in \check\T_{[1,k+1]}$, for each tile $\check T$ of $\check\CT_{k+1}$ we define its modification $\tilde T$ as the union all tiles of $\tilde\CT_k$ whose centers lie in $\check T$, where $\tilde\CT_k$ is the $k$th term in
$$
\phi_{[1,k]}(\check\CT_{[1,k]})=\phi_{[1,k]}(\check\CT_1,\check\CT_2,\dots,\check\CT_k) =(\tilde\CT_1,\tilde\CT_2,\dots,\tilde\CT_k).
$$
The tiling consisting of the new tiles $\tilde T$ is denoted by $\tilde\CT_{k+1}$ and is added as the last term in the definition of $\phi_{[1,k+1]}(\check\CT_{[1,k+1]})=(\tilde\CT_1,\tilde\CT_2,\dots,\tilde\CT_k,\tilde\CT_{k+1})$. We may need to apply adjustment of centers of $\tilde\CT_{k+1}$ in case some of them falls outside the new tiles, but this can be done using a \tl\ conjugacy (see subsection \ref{spqt}, the comment after Theorem \ref{DH1}).
Eventually we create a map $\phi$ sending each $\check{\boldsymbol\CT}=(\check\CT_k)_{k\in\N}\in
\check{\mathbf T}$ to a congruent system of tilings $\tilde{\boldsymbol\CT}=(\tilde\CT_k)_{k\in\N}$. We define $\tilde{\mathbf T}$ as the image $\phi(\check{\mathbf T})$. A careful verification that $\tilde{\mathbf T}$ is a F\o lner system of tilings and that $\phi$ is a \tl\ factor map is given in \cite{DHZ} and it is pointless to copy it here.
\end{proof}

\begin{proof}[Proof of the missing implication in Corollary \ref{dt}]
Suppose $G$ acts on a \zd\ compact metric space $X$ (we do not assume freeness of the action) and that it admits a tiling system $\mathbf T=\bigvee_{k\in\N}\T_k$ as a \tl\ factor. Let $\mathsf A,\mathsf B$ be disjoint clopen subsets of $X$ such that $\mu(\mathsf B)>\mu(\mathsf A)$ for all \im s $\mu$ on $X$. We need to show that $\mathsf A\preccurlyeq\mathsf B$.

As we have observed in Remark \ref{from0}, the infimum $\inf_{\mu\in\M_G(X)}(\mu(\mathsf B)-\mu(\mathsf A))$ is positive. Proposition \ref{prop} (1) implies that
$$
\underline D(\mathsf B,\mathsf A)\ge6\eps,
$$
for some $\eps>0$. By Lemma \ref{bbb}, there exists a finite set $F\subset G$ satisfying, for every $x\in X$, $\underline D_F(B_x,A_x)\ge5\eps$. For some $k$, the set of shapes $\CS=\CS_k$ of $\T=\T_k$ consists of $(F,\eps)$-\inv\ sets. Recall that $\T$ is a \tl\ factor of $X$ via a map $x\mapsto\CT_x$. Lemma~\ref{bdc} implies that for every $S\in\mathcal S$ and $x\in X$, we have
$$
\underline D_S(B_x,A_x)\ge \underline D_F(B_x,A_x)-4\eps>0,
$$
which yields $|A_xg^{-1}\cap S|<|B_xg^{-1}\cap S|$ for every $g\in G$.

We will now build an auxiliary symbolic factor $\hat X$ of $X$ carrying the minimum information about both the sets $\mathsf A,\mathsf B$ and the dynamical tiling. Namely, we define a factor map $\pi:X\to \hat X\subset {\hat{\rm V}}^G$, where $\hat{\rm V} = \{\mathsf 0,\mathsf 1,\mathsf 2\}\times{\rm V}$ (as usually, ${\rm V}=\{``S\,":S\in\CS\}\cup\{0\}$ is the alphabet of the symbolic representation of the dynamical tiling $\T$), as follows
$$
(\pi(x))_g=
\begin{cases}
(\mathsf1,``S\,")& \ \text{ if \ \ }g\in A_x, Sg\in\CT_x\\
(\mathsf2,``S\,")& \ \text{ if \ \ }g\in B_x, Sg\in\CT_x\\
(\mathsf0,``S\,")& \ \text{ if \ \ }g\notin A_x\cup B_x, Sg\in\CT_x\\
(\mathsf1,0)& \ \text{ if \ \ }g\in A_x, Sg\notin\CT_x\\
(\mathsf2,0)& \ \text{ if \ \ }g\in B_x, Sg\notin\CT_x\\
(\mathsf0,0)& \ \text{ if \ \ }g\notin A_x\cup B_x, Sg\notin\CT_x.
\end{cases}
$$
Clearly, the subshift $\hat X$ factors onto $\T$ and $\CT_{\hat x}=\CT_x$ whenever $x\in\pi^{-1}(\hat x)$.
Denote $\hat{\mathsf A} = [\mathsf 1,\cdot]$ and $\hat{\mathsf B} = [\mathsf 2,\cdot]$.
We have $\mathsf A=\pi^{-1}(\hat{\mathsf A})$ and $\mathsf B=\pi^{-1}(\hat{\mathsf B})$.

Thus it suffices to show that $\hat{\mathsf A}\preccurlyeq\hat{\mathsf B}$ in $\hat X$.
By Theorem \ref{tutka} (1), the proof will be ended once we will have constructed a family of injections $\tilde\varphi_{\hat x}:\hat A_{\hat x}\to \hat B_{\hat x}$ indexed by $\hat x\in\hat X$ and determined by a block code.

By the definition of $\pi$ we have, that if $\hat x=\pi(x)$ then $A_x=\hat A_{\hat x}$ and $B_x=\hat B_{\hat x}$, and the inequality $|A_xg^{-1}\cap S|<|B_xg^{-1}\cap S|$ translates to $|\hat A_{\hat x}g^{-1}\cap S|<|\hat B_{\hat x}g^{-1}\cap S|$ (for each $\hat x\in \hat X$, $S\in\mathcal S$ and $g\in G$). In other words, in every block $g(\hat x)|_S$ there are more symbols $\mathsf 2$ than $\mathsf 1$ (we just consider the first entries in the pairs which constitute the symbols). Since $\mathcal S$ is finite and for each $S\in\mathcal S$ there are only finitely many blocks $C\in{\hat{\rm V}}^S$, we have globally a finite number of possible blocks $C$ appearing in the role $g(\hat x)|_S$ (with $\hat x\in \hat X$, $g\in G$ and $S\in\mathcal S$). For every block $C$ in this finite collection we select arbitrarily an injection $\varphi_{C}:\{s\in S:C(s)=(\mathsf 1,\cdot)\}\to\{s\in S:C(s)=(\mathsf 2,\cdot)\}$, where $S$ is the domain of $C$.

Fix some $\hat x\in \hat X$ and $a\in\hat A_{\hat x}$. Let $Sc$ be the tile of $\CT_{\hat x}$ containing $a$ and let $C = c(\hat x)|_S$. We define
$$
\tilde\varphi_{\hat x}(a) = \varphi_C(ac^{-1})c.
$$
Since $C(ac^{-1})=\hat x_a =(\mathsf 1,\cdot)$, $\varphi_C(ac^{-1})$ is defined and satisfies $C(\varphi_C(ac^{-1}))=(\mathsf 2,\cdot)$, and thus
$\hat x_{\varphi_C(ac^{-1})c}=(\mathsf 2,\cdot)$, i.e., $\tilde\varphi_{\hat x}(a)\in\hat B_{\hat x}$. Notice that $\tilde\varphi_{\hat x}(a)$ belongs to the same tile of $\CT_{\hat x}$ as $a$. Injectivity of so defined $\tilde\varphi_{\hat x}$ is very easy. Consider $a_1\neq a_2\in \hat A_{\hat x}$. If both elements belong to the same tile of $\CT_{\hat x}$, then their images are different by injectivity of $\varphi_C$, where $C= c(\hat x)|_S$. If they belong to different tiles, their images also belong to different tiles, hence are different. The last thing to check is the condition \eqref{clr}, which will establish that the family $\{\tilde\varphi_{\hat x}\}_{\hat x\in\hat X}$ is determined by a block code. We claim that the horizon
$E=\bigcup_{S\in\mathcal S}SS^{-1}$ is good. Indeed, suppose, for some $\hat x_1,\hat x_2\in \hat X$ and $a_1\in\hat A_{\hat x_1},a_2\in\hat A_{\hat x_2}$, that
\begin{equation} \label{add}
a_1(\hat x_1)|_E = a_2(\hat x_2)|_E.
\end{equation}
Let $Sc$ be the central (i.e., containing the unity) tile of $\CT_{a_1(\hat x_1)}$. Then the second entry of the pair constituting the symbol $(a_1 (\hat x_1))_c$ equals $``S\,"$.
Since $c\in\bigcup_{S\in\mathcal S}S^{-1}\subset E$, by \eqref{add} we obtain that the second entry of the symbol $(a_2 (\hat x_2))_c$ also equals $``S\,"$, so that $Sc$ is the central tile of $\CT_{a_2(\hat x_2)}$. Further, since $Sc\subset E$, by \eqref{add} we have $a_1(\hat x_1)|_{Sc}= a_2(\hat x_2)|_{Sc}$ and hence $ca_1(\hat x_1)|_S= ca_2(\hat x_2)|_S$. That is, these two restrictions define the same block $C\in {\hat{\rm V}}^S$. This implies that both $\tilde\varphi_{\hat x_1}(a_1)$ and $\tilde\varphi_{\hat x_2}(a_2)$ are defined with the help of the same injection
$\varphi_C$, and
$$
\tilde\varphi_{\hat x_1}(a_1) = \varphi_C(a_1c_1^{-1})c_1, \ \ \ \ \tilde\varphi_{\hat x_2}(a_2) = \varphi_C(a_2c_2^{-1})c_2,
$$
where $c_1$ is the center of the tile of $\CT_{\hat x_1}$ containing $a_1$, and $c_2$
is the center of the tile of $\CT_{\hat x_2}$ containing $a_2$. By shift equivariance of the dynamical tiling, we easily see that $c_1=ca_1$ and $c_2=ca_2$, which yields
$$
\tilde\varphi_{\hat x_1}(a_1)a_1^{-1} = \varphi_C(c^{-1})c=\tilde\varphi_{\hat x_2}(a_2)a_2^{-1}.
$$
This is exactly the condition $\eqref{clr}$ and the proof is finished.
\end{proof}

Combining Theorem \ref{main} with Corollary \ref{dt} we obtain:

\begin{cor}
If $G$ is a subexponential group then every free action of $G$ on a \zd\ compact metric space has a tiling system as a \tl\ factor.
\end{cor}

We conclude this section with a question. Let us say that a countable amenable group $G$ has the \emph{tiling property} if any free action of $G$ on a \zd\ compact metric space has a tiling system as a \tl\ factor. In such case, by Corollary~\ref{dt}, any free action on a \zd\ compact metric space admits comparison. It is easy to see that the property of having a tiling system as a factor cannot be extended (without modifying the definition) to non-free actions. However, there are \emph{a priori} no obvious reasons why admitting comparison for non-free action could not be implied by the tiling property. Thus the following question is very natural:

\begin{ques}
Is it true that if $G$ has the tiling property (which depends on free actions only) then it also has the comparison property (which depends on all actions; of course in both cases we restrict our attention to actions on \zd\ compact metric spaces).
\end{ques}

\subsection{Symbolic extensions for actions of selected groups}\label{7.3}

We are ready to discuss the full version of the Symbolic Extension Entropy Theorem for selected countable amenable groups.

\subsubsection{What goes wrong in general countable amenable groups}
It is clear that in order to create a purely symbolic extension of $X$ (or of $\bar X$) with $h^{\pi}=\EA$, in place of a quasi-symbolic one, at least using the techniques known to us, one would need to use, in the construction of Section \ref{s5}, an encodable tiling system $\mathbf T$ of \tl\ entropy zero. Having an encodable tiling system $\mathbf T$ with entropy zero and its principal symbolic extension $Z$ at our disposal, we can use $\mathbf T$ in the construction of Section \ref{s5} and then turn the resulting quasi-symbolic extension $\bar Y=Y\vee\mathbf T$ into purely symbolic by extending it, in a natural way, to $Y\vee Z$.

Let us make it clear that attempting to build a purely symbolic extension as in Section \ref{s5}, using a F\o lner system of disjoint \qt s instead of a tiling system, is simply not going to work, even if it is perfectly encodable with \tl\ entropy zero. There are two ways to explain that. The first one reveals the technical obstacles, the second one actually kills the idea.

\smallskip\noindent(1) A disjoint \qt\ $\CT_k$ leaves a small portion of the group, say a set $B$ of small upper Banach density, uncovered. When building a symbolic preimage $y$ of $\bar x_{[1,k]}$ (or even a quasi-symbolic preimage $\bar y$ which would include a copy of the \qt\ associated to $\bar x_{[1,k]}$), we would have to decide where the information about $\bar x_{[1,k]}|_B$ should be encoded in $y$ (or $\bar y$). Any attempt to do so, eventually turns out to be equivalent to trying to distribute the set $B$ amongst the tiles of the \qt\ $\CT_k$ so that to each tile $T$ we associate a portion of $B$ of cardinality relatively small compared to $|T|$. And we need to do it using a block code. So, in fact, we are trying to factor the dynamical \qt\ onto a tiling. In view of Theorem \ref{nine}, if we were able to factor each \qt\ onto a tiling, we could also build an encodable tiling system with \tl\ entropy zero.

\smallskip\noindent(2) Just observe that if we were able to build (using no matter what technique) a purely symbolic extension with $h^{\pi}=\EA$ then a tiling system with \tl\ entropy zero (which exists, by Theorem \ref{fs}) should admit a principal symbolic extension. So, we are back dealing with the problem of existence of encodable tiling systems. By the way, we have just proved that if the group $G$ admits one encodable zero entropy tiling system then all zero entropy tiling systems are encodable.
\smallskip

The problem with existence of an encodable tiling system (let alone with entropy zero) is rather serious. In spite of many efforts, we failed to prove it for general countable amenable groups. In fact, we failed to construct any tiling system that would admit \emph{any} symbolic extension. Let us give the reader a glimpse into the obstacles. In \cite[Theorem 5.2]{DHZ} we do create a tiling system from a F\o lner system of \qt s. The main step is \cite[Theorem 4.3]{DHZ} in which we transform a disjoint $\varepsilon$-\qt\ $\CT$ onto a tiling. The framework of the proof is identical as in the proof of Theorem \ref{ddd} above:  the uncovered part $B$ of the group is distributed amongst the tiles of $\CT$ so that to each tile $T$ we ``attach'' a portion of $B$ of cardinality smaller than $\varepsilon|T|$. Moreover, the attached portion is contained in $FT$ for some finite set $F$. However, without the comparison property, the ``algorithm of attaching'' is not governed by a block code, so the tiling is not a \tl\ factor of the \qt. We use a version of Hall's Marriage Lemma \cite{HALL} and deciding to which tile a given element $b\in B$ should be attached requires examining the \qt\ of the entire group. Although this method allows to build a tiling system from a F\o lner system of disjoint \qt s, it provides no tools to prove its encodability. This really looks paradoxical, because, as we proved later in \cite{DHZ}, the resulting tiling still has \tl\ entropy zero, so the number of possible configurations of the tiles and the uncovered areas in some large F\o lner set $F_n$ is relatively small. Thus we should be able to encode the ``algorithm of attaching'' using a small number of symbols and a small percentage of the space available in $F_n$. This would lead to a tiling obtainable from $\CT$ by a block code with coding horizon $F_n$. However, the F\o lner set $F_n$ \emph{does not tile} the group $G$. In any covering of $G$ by shifted copies of $F_n$ these copies overlap and the encoded information may simply conflict with each other in the overlapping areas. For a non-conflictive encoding of the ``algorithm of attaching'' we need a disjoint F\o lner tiling of some higher order. Unfortunately, we only have at our disposal \qt s of higher order. Any such \qt, in spite of covering a subset of $G$ of lower Banach density extremely close to $1$ may leave uncovered areas containing huge (although bounded and rare) portions of the set $B$ and we would have no indication as to where (i.e., to which tiles) the elements of these portions should be attached. The situation is better in groups having a symmetric F\o lner \sq\ $(F_n)_{n\in\N}$ satisfying so-called Tempelman's condition, which guarantees a bounded proportion between $|F_n^{-1}F_n|$ and $|F_n|$. One can show that then there exists a symmetric F\o lner \sq\ $(F_n)_{n\in\N}$, also satisfying the Tempelman's condition, and moreover, such that $(F_n^2)_{n\in\N}$ is also a F\o lner \sq\ (see \cite{Ph}). The group can be covered by shifted copies of $F_n^2$ so that the corresponding shifted copies of $F_n$ are disjoint (this is an easy fact for any finite set $F$). In such case, one can encode the ``algorithm of attaching'' occurring within each shifted copy of $F_n^2$ using the space available in the disjoint shifted copies of $F_n$. This idea works indeed, and it has been proved in \cite{Ph} that such groups have the comparison property (hence, by Corollary \ref{ddt}, they admit an encodable tiling system). However, the significance of this result is faded by another relatively recent result of \cite{BGT}, which says that groups satisfying a similar (seemingly slightly stronger) requirement are subexponential (hence they have the comparison property by Theorem \ref{main}). The requirement is that every finitely generated subgroup admits a symmetric, exhausting (i.e., whose union is the whole subgroup) F\o lner \sq\ with the Tempelman's condition. We refrain from detailed investigating whether the result of \cite{Ph} can be deduced from that of \cite{BGT} or not. Certainly, it is nearly covered.

Finally, one could hope to encode the ``algorithm of attaching'' occurring within each ``new'' tile (of the tiling) using the space available in the ``old'' tile (of the \qt). That is,
we could try to encode the new shape using some finite number of symbols, in form of a block over a small portion of the old tile. Unfortunately, this simple idea also fails, because we have no control over the number of shapes of the tiling that are created from one shape of the \qt. It is true that every tile of the tiling build from a tile $T$ of the \qt\ (by attaching to it small portions of $B$) is contained in $FT$ for some finite set $F$ (independent of $T$), but we have no control over the size of $F$. Even by attaching just one element of $B\cap FT$ to $T$ at a time, we can produce a number of new tiles much larger than $l^{|T|}$ (where $l$ is some a priori assumed cardinality of the alphabet used for the coding) so that encoding the new shape within $T$ becomes impossible. This does not stand in a contradiction with \tl\ entropy zero, because the number of configurations of tiles in a huge F\o lner set $F_n$, much larger than all the shapes of the tiling, still can be small relative to the size of $F_n$ (and we have already discussed why this is useless).
\smallskip

\subsubsection{Groups with the comparison property}

A class of countable amenable groups, in which we can claim the full version of the Symbolic Extension Entropy Theorem is that of groups with the comparison property.

\begin{theorem}\label{coc}
Suppose $G$ is a countable amenable group with the comparison property. Then, for every action of $G$ on a compact metric space $X$, we have the equivalence: a function $\EA$ on $\MGX$ is a finite and affine \se\ of the \ens\ $\H$ of $X$ if and only if there exists a symbolic extension $\pi:Y\to X$ such that $h^{\pi}=\EA$ on $\MGX$.
\end{theorem}

\begin{proof} Only one implication needs a proof, and the proof is now straightforward. Let a finite (hence bounded) affine \se\ $\EA$ be given. By Corollary \ref{ddt} there exists an encodable tiling system $\tilde{\mathbf T}$ of $G$ with entropy zero, having a principal (hence also of entropy zero) symbolic extension $Z$. We use $\tilde{\mathbf T}$ to create a quasi-symbolic extension $\bar Y=Y\vee\tilde{\mathbf T}$ of $X$ with the extension entropy function matching $\EA$, as in Theorem~\ref{quasi}. Finally, we extend $\bar Y = Y\vee\tilde{\mathbf T}$ to a symbolic system $Y\vee Z$, without changing the extension entropy function.
\end{proof}

\subsubsection{Residually finite groups} Note that full comparison property of the group $G$ is not necessary for the above proof to work. What we need is just one encodable tiling system of entropy zero. The knowledge about tiling options in general countable amenable groups is very limited. It is unknown whether all such groups are \emph{monotileable}, i.e., for some F\o lner \sq\ $(F_n)_{n\in\N}$ and every $n\in\N$, admit a tiling with just one shape $F_n$.

An example of monotileable groups are \emph{residually finite} (see e.g. \cite{Ma} for definition) countable amenable groups. Monotileability of such groups is shown in \cite[Theorem 1]{W0}. Moreover, in such groups there exists a tiling system consisting of one-shape tilings. Such tiling system obviously has entropy zero and is encodable (the proof of Theorem \ref{nine} applies with $r(\epsilon_k) = 1$ for every $k\in\N$). We remark, that it is an open problem whether all residually finite amenable groups have the comparison property. This is why the extension of the full version of our Symbolic Extension Entropy Theorem to this class is, in the present state of knowledge, independent of Theorem \ref{coc}.

\begin{theorem}\label{fc}
Suppose $G$ is a residually finite countable amenable group. Then, for every action of $G$ on a compact metric space $X$, we have the equivalence: a function $\EA$ on $\MGX$ is a finite and affine \se\ of the \ens\ $\H$ of $X$ if and only if there exists a symbolic extension $\pi:Y\to X$ such that $h^{\pi}=\EA$ on $\MGX$.
\end{theorem}

\section*{Appendix A}

\setcounter{theorem}{0}

\renewcommand{\thetheorem}{A.\arabic{theorem}}

It is known since a long time that any $\Z$-action with \tl\ entropy zero can be extended to a zero entropy subshift with two symbols (combine \cite[Theorem 7.4]{BFF} with e.g. \cite[Theorem 7.2.3]{D1}). Since one symbol produces only the trivial subshift, two is clearly the necessary minimum. Two symbols have this nice feature that symbolic elements in $\{0,1\}^G$ can be identified with subsets of the group. In any construction of a symbolic extension (of some $\Z$-action) over two symbols, at some place, more or less explicitly, it is used that for any finite set $F\subset\N$, $F=F+n$ if and only if $n=0$. It is not necessarily so in more abstract groups and this section is added just to cope with this slight difficulty. The goal is to prove, in anticipation of possible questions, that in the symbolic extension in Theorem \ref{nine}, we can, although not without some extra effort, reduce the number of symbols from three to two. We will use this opportunity to develop a small (nevertheless somewhat excessive for our goal) ``theory of recognizability''.
\smallskip

\subsubsection*{A.1. Recognizability}
Throughout this section we assume that $G$ is an infinite group with unity~$e$.

\begin{definition}A finite set $A\subset G$ %containing the unit
has \emph{recognizable origin} if the only $g\in G$ such that $Ag=A$ is the unity $e$.
\end{definition}
This property has the following interpretation: for every \emph{shifted copy} of $A$, (i.e., for any set of the form $Ag$, $g\in G$) we know exactly where its \emph{origin} (i.e., the element $g$) is.

\begin{lemma} Any one-element set has recognizable origin. Let $A\subset G$ be a finite set
%containing the unit
of cardinality at least 2. Let $g\notin AA^{-1}A$. Then $B=A\cup\{g\}$ has recognizable origin. (Note that $g\notin A$, so $|B|=|A|+1$.)
\end{lemma}

\begin{proof} The first statement is trivial.
Suppose that $B$ does not have recognizable origin, i.e., there exists $h\in G$, $h\neq e$, such that $Bh=B$. We have $Ah\cup\{gh\} = A\cup\{g\}$ and since $gh\neq g$, we get $g\in Ah$ (and also $g\in Ah^{-1}$). On the other hand, since $|A|\ge 2$, we get $A\cap Ah\neq\emptyset$ implying $h\in A^{-1}A$ (and $h^{-1}\in A^{-1}A$) and thus $g\in AA^{-1}A$, a contradiction.
\end{proof}

\begin{remark}
It follows from the proof that $g$ can be selected from any a priori given infinite subset of $G$.
\end{remark}

We remark that if $|A|=1$, there may be no $g\notin A$ such that $A\cup\{g\}$ has recognizable origin. This happens if all elements of the group are of order 2 (i.e., $g^2=e$ for all $g\in G$). This is why the lemma produces sets with recognizable origin of all possible finite cardinalities except 2.

\begin{definition}Let $\{A_1,A_2,\dots,A_k\}$ be a collection of finite sets of equal cardinalities. We say that the collection is \emph{recognizable with recognizable origins} if for any $1\!\le\!i,j\!\le\!k$ and $g\in G$, the only possibility that $A_ig=A_j$ is when $i=j$ and $g=e$.
\end{definition}

This property has the following interpretation: given any set of the form $A_ig$, one can recognize which of the sets from the collection has been shifted and how (in particular, the sets $A_i$ must all be different).

\begin{lemma}\label{rec} Let $\{A_i,\ 1\!\le\!i\!\le\!k\}$ be a collection of finite sets of equal cardinalities larger than or equal to 2. Then there exist elements $g_i\notin A_i$ ($i=1,2,\dots,k$) such that the collection $\{B_i,\ 1\!\le\!i\!\le\!k\}$ where $B_i = A_i\cup\{g_i\}$, is recognizable with recognizable origins.
\end{lemma}

\begin{proof}
We choose $g_1\notin A_1A_1^{-1}A_1$, so that $B_1$ has recognizable origin.
From here on we proceed by induction. Suppose that for some $1\!\le\!i\!\le\!k-1$ we have selected $g_1\notin A_1,\dots,g_i\notin A_i$ so that the collection $\{B_1,\dots,B_i\}$ is recognizable with recognizable origins. We select $g_{i+1}\notin A_{i+1}A_{i+1}^{-1}A_{i+1}$ (so that $B_{i+1}$ has recognizable origin) and moreover, we choose $g_{i+1}$ so it does not belong to any of the sets $g_jA_j^{-1}A_{i+1}\cup A_jg_j^{-1}A_{i+1}$ ($1\!\le\!j\!\le\!i$).
Suppose that the collection $\{B_1,\dots,B_{i+1}\}$ is not recognizable. The only possibility is that $B_{i+1}=B_jg$ for some $j=1,2,\dots,i$ and $g\in G$. That is,
$$
A_{i+1}\cup\{g_{i+1}\}=A_jg\cup \{g_jg\}.
$$
One option is that $g_{i+1}=g_jg$ and $A_{i+1}=A_jg$. This leads to $g=g_j^{-1}g_{i+1}$ and $g\in A_j^{-1}A_{i+1}$, hence $g_{i+1}\in g_jA_j^{-1}A_{i+1}$, which is impossible.
Otherwise, we have $g_{i+1}\in A_jg$ and $g_jg\in A_{i+1}$ leading to $g\in A_j^{-1}g_{i+1}\cap g_j^{-1}A_{i+1}$, and hence $g_{i+1}\in A_jg_j^{-1}A_{i+1}$, which is also impossible.
\end{proof}

\begin{remark}
It follows from the proof that the elements $g_1,g_2,\dots,g_k$ can be selected from any a priori given infinite subset of $G$.
\end{remark}

\begin{definition}Let $\{A_i,\ 1\!\le\!i\!\le\!k\}$ be a collection of finite sets of equal cardinalities larger than or equal 2, each containing the unity. A family of their shifted copies
$$
\{A_ig_{i,j}: 1\!\le\!i\!\le\!k,\ j\in\N\}
$$
is said to be \emph{fully recognizable} if, for any $i_0\in\{1,2,\dots,k\}$, the inclusion
$$
A_{i_0}g\subset\bigcup\{A_ig_{i,j}: 1\!\le\!i\!\le\!k,\ j\in\N\}
$$
is possible only if $i_0=i$ and $g=g_{i,j}$ for some $j\in\N$.
\end{definition}
This property has the following interpretation: in the above union we can recognize all component sets together with their origins (in particular, the collection $\{A_i,\ 1\!\le\!i\!\le\!k\}$ must be recognizable with recognizable origins).

\begin{lemma}\label{mar}
If the collection $\{A_1,A_2,\dots,A_k\}$ is recognizable with recognizable origins and the following two sets
$$
\bigcup\{A^{-1}_{i'}A_iA^{-1}_iA_{i''}:1\!\le\!i',i,i''\!\le\!k\}
$$
(later called \emph{the margin of full recognizability}) and
$$
\{g_{i',j'}(g_{i'',j''})^{-1}:1\!\le\!i',i''\!\le\!k,\ j',j''\in\N,\ (i',j')\neq(i'',j'')\}
$$
are disjoint, then the family $\{A_ig_{i,j}: 1\!\le\!i\!\le\!k,\ j\in\N\}$ is fully recognizable.
\end{lemma}

\begin{proof}
Suppose that $A_{i_0}g$ is contained in the union $\bigcup\{A_ig_{i,j}: 1\!\le\!i\!\le\!k,\ j\in\N\}$. Either $A_{i_0}g$ matches one component sets $A_ig_{i,j}$ or not. If it does, then by recognizability with recognizable origins of the collection $\{A_1,A_2,\dots,A_k\}$,
we have $i_0=i$ and $g=g_{i,j}$, as required. If it does not, then $A_{i_0}g$ intersects two different sets $A_{i'}g_{i',j'}$ and $A_{i''}g_{i'',j''}$ (with $(i',j')\neq(i'',j'')$), hence the sets $A_{i_0}^{-1}A_{i'}g_{i',j'}$ and $A_{i_0}^{-1}A_{i''}g_{i'',j''}$ are not disjoint (both contain $g$), which, after elementary rearrangements leads to $g_{i',j'}(g_{i'',j''})^{-1}\in A^{-1}_{i'}A_{i_0}A^{-1}_{i_0}A_{i''}$.
By assumption, this cannot happen.
\end{proof}

\subsubsection*{A.2. Reduction of the number of symbols}
\begin{theorem} Let $G$ be a countable amenable group. There exists a perfectly encodable F\o lner system $\hat{\mathbf T}$ of disjoint \qt s of \tl\ entropy zero which has an isomorphic symbolic extension on two symbols.
\end{theorem}

\begin{proof} In the proof of Theorem \ref{nine} we need to change only the steps 1 and 2 of the construction of the (intermediate) system of (non-disjoint) \qt s $\mathbf T$.
In step 1 we let $X_1$ be the full shift over two symbols $X_1=\{0,1\}^G$. It is step 2, which requires the essential modification.

\smallskip\noindent
{\bf Step 2}. We define $M=3\cdot2r(\epsilon_2)$. Now, we choose a family $\mathcal U$ consisting of $M$ sets, which is recognizable with recognizable origins (Lemma \ref{rec} guarantees that such a family exists). We let $U_2$ be the margin of full recognizability for the family $\mathcal U$ (see Lemma \ref{mar}; note that $U_2\supset\bigcup\mathcal U$). Then, as in the proof of Theorem \ref{nine}, we let $\T_2$ be a zero entropy dynamical $\epsilon_2$-\qt\ with the collection of shapes $\CS_2\subset\{F_{n_{1,2}}, F_{n_{2,2}},\dots,F_{n_{{r(\epsilon_2)},2}}\}$\  ($n_{1,2}<n_{2,2}<\cdots<n_{n_{r(\epsilon_2)},2}$) and such that for every $\CT_2\in\T_2$ the set of centers $C(\CT_2)$ is $U_2$-separated. We use the revised version of Theorem \ref{DH}, and thus, for each tiling $\CT_2\in\T_2$ and each shape $S\in\CS(\T_2)$, we can determine the primariness of the tiles of $\CT_2$ with the shape $S$ (by observing the symbols $``S_\mathsf p"$ versus $``S_\mathsf n"$). To every symbol $``S_\mathsf s"$, where $S\in\CS_2$ and $\mathsf s\in\{\mathsf p,\mathsf n\}$, one can disjointly associate a family of three different sets $\{B^{(2)}_{S,\mathsf s,-1}, B^{(2)}_{S,\mathsf s,0}, B^{(2)}_{S,\mathsf s,1}\}\subset\mathcal U$. We allow $z\in Z_1$ to be a member of $\pi_2^{-1}(\CT_{[1,2]})$ (where $\CT_{[1,2]}\in\T_{[1,2]}$) if the following holds:
\begin{enumerate}
	\item Whenever $\CT_{2,c}=``S_\mathsf s"$ (i.e., $c\in C(\CT_2)$ is the center of a primary
	or non-primary tile $Sc$ of $\CT_2$) then we require that $z|_{U_2c}=\mathbbm
	1_{B^{(2)}_{S,\mathsf s,i}c}|_{U_2c}$ for some $i\in\{-1,0,1\}$ (it is essential that the
	sets $U_2c$ contain the sets ${B^{(2)}_{S,\mathsf s,i}c}$ and are disjoint for	different
	$c\in C(\CT_2)$).
	\item All independent choices of the above indices $i$ for different centers $c\in
	C(\CT_2)$ are represented in the elements $z\in \pi_2^{-1}(\CT_{[1,2]})$.
	\item We define the \emph{background of $\CT_2$} as the complement of $U_2C(\CT_2)$,
	and we require that $z_g=0$ for every $z\in \pi_2^{-1}(\CT_{[1,2]})$ and all $g$ in this
	background.
\end{enumerate}
The map $\pi_2$ now functions by a slightly different rule: By Lemma \ref{mar} and since, for every $\CT_2\in\T_2$, the set of centers $C(\CT_2)$ is $U_2$-separated, any $z\in Z_2$ equals the characteristic function of a fully recognizable family of shifted copies of members of the collection $\mathcal U$. This allows to locate all center sets $c\in C(\CT_2)$ and recognize the sets $B^{(2)}_{S,\mathsf s,i}c$ attached to them (with determining the parameters $S,\ \mathsf s$ and $i$), using a block code with coding horizon~$U_2$.

This ends the description of the modification of step 2. From now on we have, as before, three blocks admitted to encode every shape and all further steps of the construction of $\mathbf T$ remain unchanged. The passage from $\mathbf T$ to the system of disjoint \qt s $\hat{\mathbf T}$ stays the same.
\end{proof}

\section*{Acknowledgements}

The research of the first author is supported by NCN (National Science Center, Poland) grant 2013/08/A/ST1/00275. The paper was written during a series of visits of the first author in Fudan Univeristy, Shanghai, China.

The second author was supported by NSFC (National Natural Science Foundation of China) Grants 11671094, 11722103 and 11731003.

\backmatter

\bibliographystyle{amsplain} % global bibliography

%\bibliography{DownZh}

\begin{thebibliography}{10}

\bibitem{AS}
George~M. Adel{\cprime}son-Vel{\cprime}ski\u{\i} and Yu.~A. \v{S}re\u{\i}der,
  \emph{The {B}anach mean on groups}, Uspehi Mat. Nauk (N.S.) \textbf{12}
  (1957), no.~6(78), 131--136.

\bibitem{AJ75}
Mustafa~A. Akcoglu and Andr\'{e}s del Junco, \emph{Convergence of averages of
  point transformations}, Proc. Amer. Math. Soc. \textbf{49} (1975), 265--266.


\bibitem{BBF}
Mathias Beiglb\"{o}ck, Vitaly Bergelson, and Alexander Fish, \emph{Sumset
  phenomenon in countable amenable groups}, Adv. Math. \textbf{223} (2010),
  no.~2, 416--432.

\bibitem{Bow}
Lewis Bowen, \emph{Measure conjugacy invariants for actions of countable sofic
  groups}, J. Amer. Math. Soc. \textbf{23} (2010), no.~1, 217--245.


\bibitem{Bo}
Mike Boyle, \emph{Lower entropy factors of sofic systems}, Ergodic Theory
  Dynam. Systems \textbf{3} (1983), no.~4, 541--557.


\bibitem{BD}
Mike Boyle and Tomasz Downarowicz, \emph{The entropy theory of symbolic
  extensions}, Invent. Math. \textbf{156} (2004), no.~1, 119--161.


\bibitem{BFF}
Mike Boyle, Doris Fiebig, and Ulf Fiebig, \emph{Residual entropy, conditional
  entropy and subshift covers}, Forum Math. \textbf{14} (2002), no.~5,
  713--757.


\bibitem{BH}
Mike Boyle and David Handelman, \emph{Orbit equivalence, flow equivalence and
  ordered cohomology}, Israel J. Math. \textbf{95} (1996), 169--210.


\bibitem{BGT}
Emmanuel Breuillard, Ben Green, and Terence Tao, \emph{The structure of
  approximate groups}, Publ. Math. Inst. Hautes \'{E}tudes Sci. \textbf{116}
  (2012), 115--221.

\bibitem{B}
Julian Buck, \emph{Smallness and comparison properties for minimal dynamical
  systems}, preprint (2013), arXiv:1306.6681.

\bibitem{Bu}
David Burguet, \emph{${C}^2$ surface diffeomorphisms have symbolic extensions},
  Invent. Math. \textbf{186} (2011), no.~1, 191--236.


\bibitem{BuD}
David Burguet and Tomasz Downarowicz, \emph{Uniform generators, symbolic
  extensions with an embedding, and structure of periodic orbits}, J. Dynam.
  Differential Equations (2018), to appear,
  https://doi.org/10.1007/s10884--018--9674--y.

\bibitem{CZ}
Nhan-Phu Chung and Guohua Zhang, \emph{Weak expansiveness for actions of sofic
  groups}, J. Funct. Anal. \textbf{268} (2015), no.~11, 3534--3565.


\bibitem{Cu}
Joachim Cuntz, \emph{Dimension functions on simple {$C\sp*$}-algebras}, Math.
  Ann. \textbf{233} (1978), no.~2, 145--153.



\bibitem{DP}
Alexandre~I. Danilenko and Kyewon~Koh Park, \emph{Generators and {B}ernoullian
  factors for amenable actions and cocycles on their orbits}, Ergodic Theory
  Dynam. Systems \textbf{22} (2002), no.~6, 1715--1745.


\bibitem{D}
Dou Dou, \emph{Minimal subshifts of arbitrary mean topological dimension},
  Discrete Contin. Dyn. Syst. \textbf{37} (2017), no.~3, 1411--1424.


\bibitem{D0}
Tomasz Downarowicz, \emph{Entropy structure}, J. Anal. Math. \textbf{96}
  (2005), 57--116.

\bibitem{D1}
\bysame, \emph{Entropy in dynamical systems}, New Mathematical Monographs,
  vol.~18, Cambridge University Press, Cambridge, 2011.


\bibitem{DFR}
Tomasz Downarowicz, Bartosz Frej, and Pierre-Paul Romagnoli, \emph{Shearer's
  inequality and infimum rule for {S}hannon entropy and topological entropy},
  Dynamics and numbers, Contemp. Math., vol. 669, Amer. Math. Soc., Providence,
  RI, 2016, pp.~63--75.


\bibitem{DH}
Tomasz Downarowicz and Dawid Huczek, \emph{Dynamical quasitilings of amenable
  groups}, Bull. Pol. Acad. Sci. Math. \textbf{66} (2018), no.~1, 45--55.


\bibitem{DHZ}
Tomasz Downarowicz, Dawid Huczek, and Guohua Zhang, \emph{Tilings of amenable
  groups}, J. Reine Angew. Math. (2016), to appear,
  https://doi.org/10.1515/crelle--2016--0025.

\bibitem{DM}
Tomasz Downarowicz and Alejandro Maass, \emph{Smooth interval maps have
  symbolic extensions: the antarctic theorem}, Invent. Math. \textbf{176}
  (2009), no.~3, 617--636.



\bibitem{DN}
Tomasz Downarowicz and Sheldon Newhouse, \emph{Symbolic extensions and smooth
  dynamical systems}, Invent. Math. \textbf{160} (2005), no.~3, 453--499.


\bibitem{Fo}
Erling F{\o}lner, \emph{On groups with full {B}anach mean value}, Math. Scand.
  \textbf{3} (1955), 243--254.


\bibitem{FH}
Bartosz Frej and Dawid Huczek, \emph{Minimal models for actions of amenable
  groups}, Groups Geom. Dyn. \textbf{11} (2017), no.~2, 567--583.


\bibitem{FH1}
\bysame, \emph{Faces of simplices of invariant measures for actions of amenable
  groups}, Monatsh. Math. \textbf{185} (2018), no.~1, 61--80.

\bibitem{GPS1}
Thierry Giordano, Ian~F. Putnam, and Christian~Fr. Skau, \emph{Topological
  orbit equivalence and {$C^*$}-crossed products}, J. Reine Angew. Math.
  \textbf{469} (1995), 51--111.

\bibitem{GPS2}
\bysame, \emph{Full groups of {C}antor minimal systems}, Israel J. Math.
  \textbf{111} (1999), 285--320.


\bibitem{GTW}
Eli Glasner, Jean-Paul Thouvenot, and Benjamin Weiss, \emph{Entropy theory
  without a past}, Ergodic Theory Dynam. Systems \textbf{20} (2000), no.~5,
  1355--1370.

\bibitem{GW}
Eli Glasner and Benjamin Weiss, \emph{Weak orbit equivalence of {C}antor
  minimal systems}, Internat. J. Math. \textbf{6} (1995), no.~4, 559--579.


\bibitem{Gr}
Rostislav~I. Grigorchuk, \emph{Degrees of growth of finitely generated groups
  and the theory of invariant means}, Izv. Akad. Nauk SSSR Ser. Mat.
  \textbf{48} (1984), no.~5, 939--985.


\bibitem{HALL}
Philip Hall, \emph{On representatives of subsets}, J. London Math. Soc.
  \textbf{10} (1935), 26--30.

\bibitem{He}
Gustav~Arnold Hedlund, \emph{Endomorphisms and automorphisms of the shift
  dynamical system}, Math. Systems Theory \textbf{3} (1969), 320--375.


\bibitem{HYZ}
Wen Huang, Xiangdong Ye, and Guohua Zhang, \emph{Local entropy theory for a
  countable discrete amenable group action}, J. Funct. Anal. \textbf{261}
  (2011), no.~4, 1028--1082.

\bibitem{H}
Dawid Huczek, \emph{Zero-dimensional extensions of amenable group actions},
  preprint (2015), arXiv:1503.02827.

\bibitem{K}
David Kerr, \emph{Dimension, comparison and almost finitness}, J. Eur. Math.
  Soc. (JEMS), to appear, arXiv:1710.00393.

\bibitem{KL}
David Kerr and Hanfeng Li, \emph{Entropy and the variational principle for
  actions of sofic groups}, Invent. Math. \textbf{186} (2011), no.~3, 501--558.


\bibitem{Kr}
Wolfgang Krieger, \emph{On entropy and generators of measure-preserving
  transformations}, Trans. Amer. Math. Soc. \textbf{149} (1970), 453--464.


\bibitem{L}
Elon Lindenstrauss, \emph{Pointwise theorems for amenable groups}, Invent.
  Math. \textbf{146} (2001), no.~2, 259--295.

\bibitem{Ma}
Wilhelm Magnus, \emph{Residually finite groups}, Bull. Amer. Math. Soc.
  \textbf{75} (1969), 305--316.

\bibitem{M}
Micha\l\ Misiurewicz, \emph{Topological conditional entropy}, Studia Math.
  \textbf{55} (1976), no.~2, 175--200.
  
  \bibitem{MOP}
Jean Moulin~Ollagnier and Didier Pinchon, \emph{The variational principle},
  Studia Math. \textbf{72} (1982), no.~2, 151--159.

\bibitem{N}
Isaac Namioka, \emph{F\o lner's conditions for amenable semi-groups}, Math.
  Scand. \textbf{15} (1964), 18--28.

\bibitem{OW}
Donald~Samuel Ornstein and Benjamin Weiss, \emph{Entropy and isomorphism
  theorems for actions of amenable groups}, J. Analyse Math. \textbf{48}
  (1987), 1--141.

\bibitem{P}
Alan L.~T. Paterson, \emph{Amenability}, Mathematical Surveys and Monographs,
  vol.~29, American Mathematical Society, Providence, RI, 1988.

\bibitem{Ph}
Maxence Phalempin, \emph{Representation of congruent sequences of tilings on
  amenable groups}, Internship Report (unpublished), University of Rennes,
  2016.

\bibitem{MR1}
Mikael R{\o}rdam, \emph{On the structure of simple {$C\sp*$}-algebras tensored
  with a {UHF}-algebra. {II}}, J. Funct. Anal. \textbf{107} (1992), no.~2,
  255--269.


\bibitem{MR2}
\bysame, \emph{The stable and the real rank of {$\mathcal{Z}$}-absorbing
  {$C\sp*$}-algebras}, Internat. J. Math. \textbf{15} (2004), no.~10,
  1065--1084.


\bibitem{R}
A.~Rosenthal, \emph{Finite uniform generators for ergodic, finite entropy, free
  actions of amenable groups}, Probab. Theory Related Fields \textbf{77}
  (1988), no.~2, 147--166.


\bibitem{Se}
Jacek Serafin, \emph{A faithful symbolic extension}, Commun. Pure Appl. Anal.
  \textbf{11} (2012), no.~3, 1051--1062.


\bibitem{Sew}
Brandon Seward, \emph{Krieger's finite generator theorem for actions of
  countable groups {I}}, Invent. Math., to appear.

\bibitem{Sl}
Konstantin Slutsky, \emph{Lecture notes on topological full groups of {C}antor
  minimal systems},
  http://homepages.math.uic.edu/kslutsky/papers/Topological--full--groups.pdf.
  
  \bibitem{STZ}
Anatoly~M. Stepin and Azad~T. Tagi-Zade, \emph{Variational characterization of
  topological pressure of the amenable groups of transformations}, Dokl. Akad.
  Nauk SSSR \textbf{254} (1980), no.~3, 545--549.

\bibitem{S}
Yuhei Suzuki, \emph{Almost finiteness for general \'{E}tale groupoids and its
  applications to stable rank of crossed products}, Int. Math. Res. Not. IMRN
  (2018), to appear, https://doi.org/10.1093/imrn/rny187.

\bibitem{SG}
G\'abor Szab\'o, \emph{Private communication}, 2017.

\bibitem{Va}
Veeravalli~S. Varadarajan, \emph{Groups of automorphisms of {B}orel spaces},
  Trans. Amer. Math. Soc. \textbf{109} (1963), 191--220.


\bibitem{vN}
John von Neumann, \emph{Zur allgemeinen theorie des masses}, Fund. Math.
  \textbf{13} (1929), 73--116.

\bibitem{Su}
\v{S}tefan \v{S}ujan, \emph{Generators for amenable group actions}, Monatsh.
  Math. \textbf{95} (1983), no.~1, 67--79.


\bibitem{WZ}
Thomas Ward and Qing Zhang, \emph{The {A}bramov-{R}okhlin entropy addition
  formula for amenable group actions}, Monatsh. Math. \textbf{114} (1992),
  no.~3-4, 317--329.


\bibitem{W0}
Benjamin Weiss, \emph{Monotileable amenable groups}, Topology, ergodic theory,
  real algebraic geometry, Amer. Math. Soc. Transl. Ser. 2, vol. 202, Amer.
  Math. Soc., Providence, RI, 2001, pp.~257--262.


\bibitem{W}
Wilhelm Winter, \emph{Decomposition rank and {$\mathcal{Z}$}-stability},
  Invent. Math. \textbf{179} (2010), no.~2, 229--301.


\bibitem{Z}
Ruifeng Zhang, \emph{Topological pressure of generic points for amenable group
  actions}, J. Dynam. Differential Equations \textbf{30} (2018), no.~4,
  1583--1606.


\bibitem{ZCY}
Dongmei Zheng, Ercai Chen, and Jiahong Yang, \emph{On large deviations for
  amenable group actions}, Discrete Contin. Dyn. Syst. \textbf{36} (2016),
  no.~12, 7191--7206.


\end{thebibliography}

\printindex

\end{document}